\documentclass[openright, american,11pt]{amsart}

\usepackage[latin1]{inputenc}

\usepackage{amsmath}
\usepackage{amssymb}
\usepackage{babel}

\usepackage{amsfonts}
\input xy
\xyoption{all}
\usepackage[dvips]{graphicx}
\usepackage{boxedminipage}
\usepackage{color}
\usepackage{MnSymbol}

\newcommand{\ZZ}{{\mathbb Z}}
\newcommand{\Z}{{\mathbb Z}}

\newcommand{\CC}{{\mathbb C}}
\newcommand{\RR}{{\mathbb R}}
\newcommand{\R}{{\mathbb R}}

\newcommand{\TT}{{\mathbb T}}
\newcommand{\T}{{\mathbb T}}

\newcommand{\A}{{\mathcal A}}

\newcommand{\C}{{\mathcal C}}

\renewcommand{\P}{{\mathcal P}}
\renewcommand{\L}{{\mathcal L}}

\newcommand{\Log}{\text{Log}}

\newcommand{\td}{{\text{''}}}
\newcommand{\tg}{{\text{``}}}

\newtheorem{thm}{Theorem}[section]

\newtheorem{defi}[thm]{Definition}

\newtheorem{prop}[thm]{Proposition}
\newtheorem{proposition}[thm]{Proposition}

          {\theoremstyle{definition}
}
          {\theoremstyle{definition}

}

\newenvironment{exo}{\it{}{}}

\newcommand{\comment}[1]{}

\begin{document}
\title{A bit of tropical geometry}
\author{Erwan Brugall\'e}
\author{Kristin Shaw}
\address{Erwan Brugallé, 
École polytechnique,
Centre Mathématiques Laurent Schwatrz, 91 128 Palaiseau Cedex, France}
\email{erwan.brugalle@math.cnrs.fr}
\address{Kristin Shaw, Departement of Mathematics, University of
  Toronto, 40 St. George St., Toronto, Ontario, CANADA
M5S 2E4.}
\email{shawkm@math.toronto.edu }
\date{\today}
\thanks{To a great extent  this text is an updated translation  from
  French of  \cite{Br11}. 
In addition to the original text, we
included in Section \ref{sec:further}  an account
 of some developments of tropical geometry that occured
 since the publication of \cite{Br11} in 2009. \\
We are very grateful to Oleg Viro for his encouragement and extremely
helpful comments. We also thank Ralph Morisson  and the anonymous referees for their useful suggestions on a
preliminary version of this text.}

\maketitle

What kinds of strange spaces with mysterious properties hide behind
the enigmatic name of \textit{tropical geometry}? In the tropics, just
as in other geometries, it is difficult to find a simpler example than
that of a line.
So
let us start from here.

A tropical line consists of three usual half lines in the directions
$(-1, 0), (0, -1)$ and $(1, 1)$ emanating from any point in the plane
(see Figure \ref{intro}a). Why call this strange object a line,
 in the
tropical sense or any other? If we look more closely we find that
tropical lines share some of the familiar geometric properties of
``usual" or ``classical" lines in the plane. 
For instance,  most pairs of 
 tropical lines 
intersect in a single point\footnote{It might happen that two distinct
   tropical lines intersect in infinitely many points, as a ray of a
   line might be contained in the parallel ray of the other. We will
   see in Section \ref{sec:trop int} that there is a more
   sophisticated notion of 
\textit{stable intersection} of two tropical curves, and
that \textit{any two} tropical lines have a unique stable intersection
point.} (see Figure \ref{intro}b). 
Also for most choices of pairs of points in the plane there is  a unique tropical line passing through the two points\footnote{The situation here is
  similar to the case
  of intersection of tropical lines: there exists the
  notion of \textit{stable tropical line} passing through two points
  in the plane, and  \textit{any
    two} such points define a unique stable tropical line.} (see Figure \ref{intro}c).

\begin{figure}[h]
\begin{center}
\begin{tabular}{ccccc}
\includegraphics[width=3cm, angle=0]{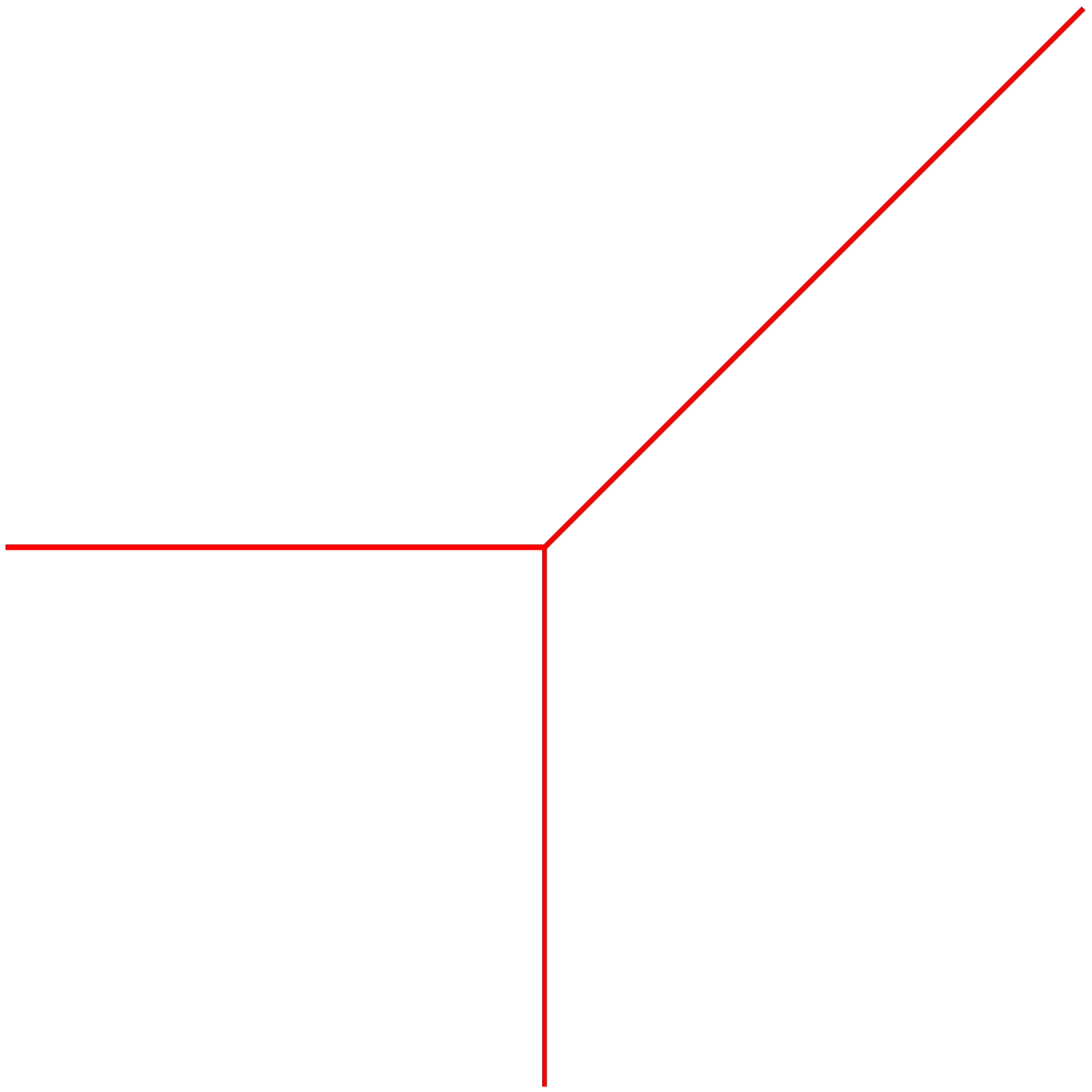}&\hspace{3ex} &
\includegraphics[width=3cm, angle=0]{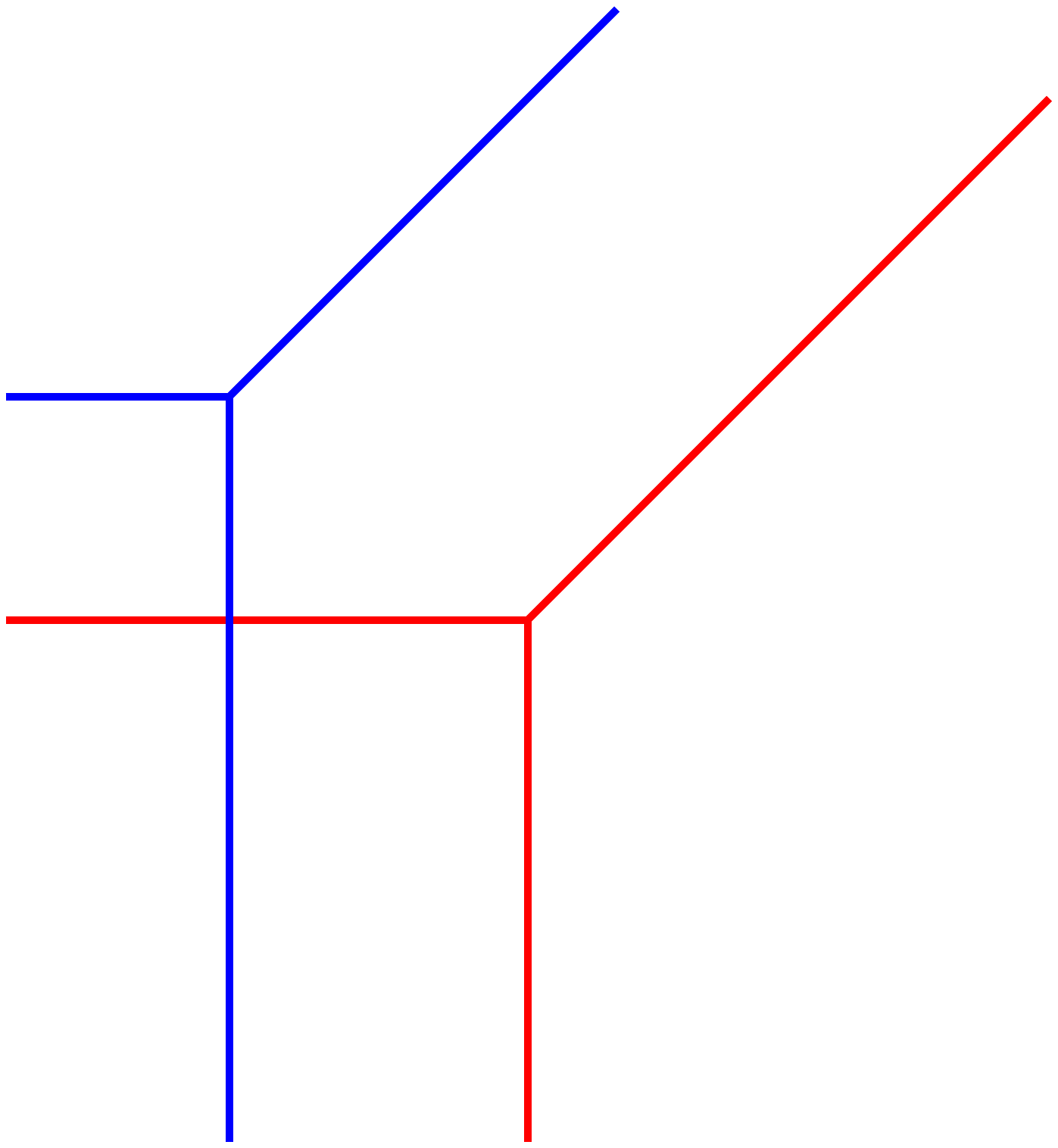}&\hspace{3ex} &
\includegraphics[width=3cm, angle=0]{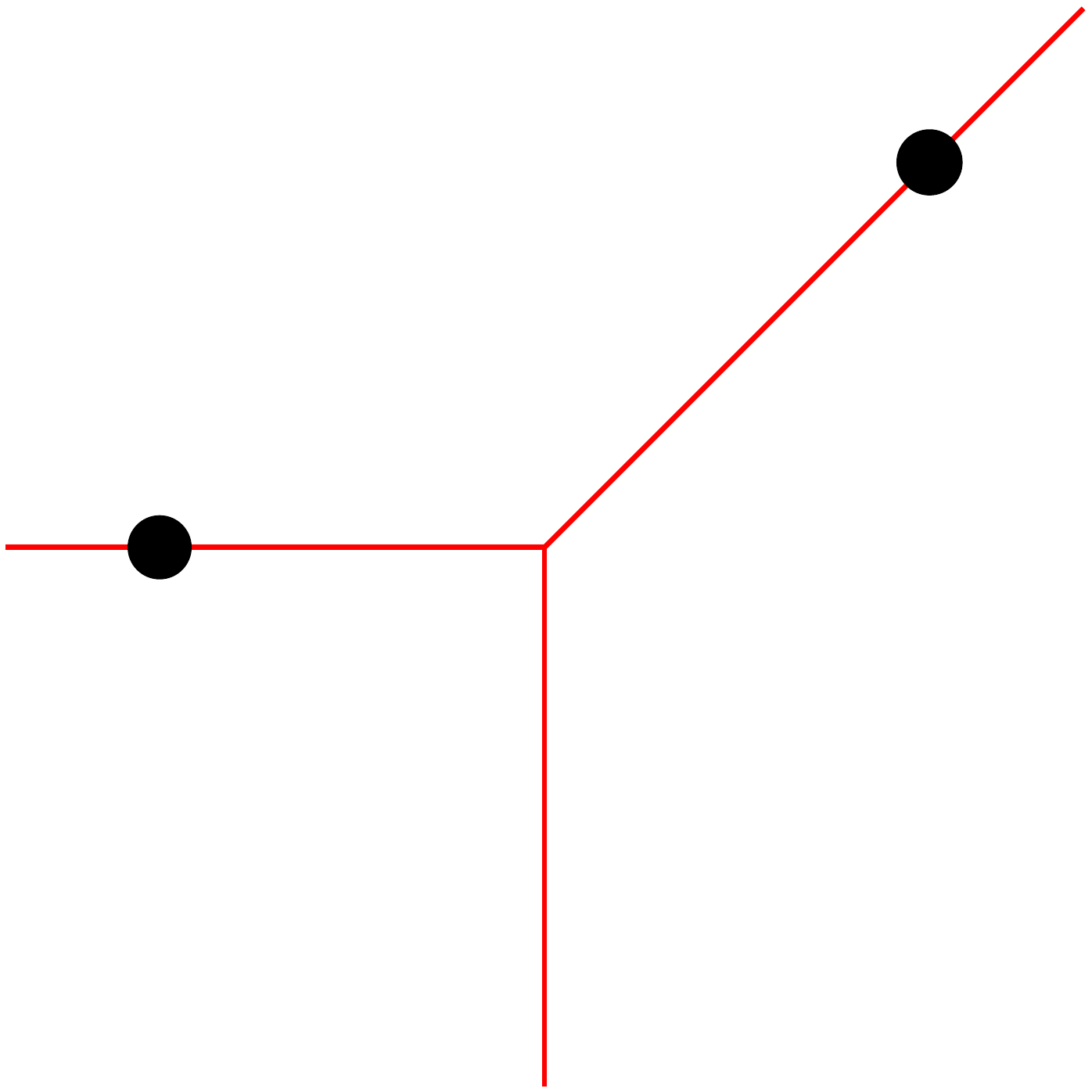}
\\ \\a)  && b)
 && c)
\end{tabular}
\end{center}
\caption{The tropical line}
\label{intro}
\end{figure}

What is even more important, although not at all visible from the
picture, is that classical and tropical lines are both given by an
equation of the form $ax +by + c = 0$. In the realm of standard
algebra, where addition is addition and multiplication is
multiplication,
 we can determine without too much difficulty the
classical line given by such an equation. 
In the tropical
world, addition is replaced by the maximum and multiplication is
replaced  by addition. Just by  doing this, all of our objects
drastically change form! 
In fact, 
even ``being equal to 0"
takes on a very different meaning... 

Classical  and tropical geometries are developed following the same principles but from two different methods of calculation. They are simply the
 geometric 
faces of two different algebras.

\vspace{1ex}

Tropical geometry is not a 
game for bored mathematicians looking for something to do.  In fact, the classical world can be \textit{degenerated} to the tropical world in such a  way so that the tropical objects conserve many properties of the original classical ones.
Because of this,
 a tropical statement has a strong chance of having a similar
 classical analogue. The advantage is that  tropical objects are
 piecewise linear and 
 thus much simpler to study than their classical
 counterparts!

We could therefore summarise the approach of tropical geometry as follows:

\begin{center}
\textit{Study simple objects to
provide theorems 
concerning  complicated objects.}
\end{center}

\vspace{1ex}

The first part of this text focuses on tropical algebra,  tropical
curves and some of their properties. We then explain why  classical
and tropical geometries are related by showing 
how the
classical world can be degenerated to the tropical one. We
illustrate this principle by considering a method known as
\textit{patchworking}, which is used to construct real algebraic
curves via objects called \textit{amoebas}. 
Finally, in the last section we 
  go beyond these  topics to showcase some 
  further developments in the field. 
These examples are a little more challenging 
 than the rest of the text, as they illustrate
 more recent  directions of research.
We conclude 
by delivering some 
bibliographical references.  

\vspace{1ex}
Before diving into the subject, we should explain the use of the word
``tropical". It is not due to the exotic forms of the objects under
consideration, nor to the appearance of the  previously mentioned
``amoebas". Before the  term  ``tropical algebra", the more
cut-and-dry name of ``max-plus algebra" was used. Then, in honour of
the work of their Brazilian colleague, Imre Simon, the computer
science researchers  at the University of Paris 7 decided to trade the
name of ``max-plus" for ``tropical".  Leaving the last word to
Wikipedia\footnote{March 15, 2009.}, the origin of the word
``tropical" \textit{simply reflects the French view on Brazil}.

\section{Tropical algebra}\label{sec:algebra}

\subsection{Tropical operations}

Tropical algebra is the set of real numbers where addition is replaced
by taking the maximum, and multiplication is replaced by the usual sum.
In other words, we define two new
operations on $\R$, called \textit{tropical addition} and \textit{multiplication} and
denoted $``+"$ and $``\times"$ respectively, in the following way:
$$``x+y" = \max (x, y ) \qquad ``x \times y" = x+y$$

In this entire text, quotation marks will be placed around an expression to indicate that the operations should be regarded as tropical. Just as in classical algebra we often abbreviate $``x \times y"$ to $``xy"$.  To familiarise ourselves with these two new strange operations, let's do some simple calculations:
$$\tg 1+1\td=1, \ \ \tg 1+2\td=2, \ \ 
\tg 1+2+3\td=3, \ \  \tg 1\times 2\td=3,  \ \   \tg 1\times
(2+(-1))\td=3, $$
$$   \tg
1\times (-2)\td=-1, \ \ \tg (5+3)^2 \td= 10.
$$

These two tropical operations have many properties in common with the
usual addition and multiplication. For example, they are both
commutative and tropical multiplication $``\times"$ is distributive
with respect to tropical  addition $``+"$ (i.e. $``(x+y)z "= ``xz +
yz"$). There are however two major differences. First of all, tropical
addition does not have an identity element in $\R$
 (i.e. there is no element $x$ such that $\max \{y, x\} = y$ for all $y \in \R$). 
Nevertheless, we can naturally extend our two tropical  operations to $-\infty$ by 
$$ \forall x \in \T, \quad ``x + (-\infty)" = \max (x, -\infty) = x \quad \text{and} \quad ``x \times (-\infty)" = x + (-\infty) = -\infty, $$
where $\T  = \R \cup \{ -\infty \}$ are the \textit{tropical
numbers}. Therefore, after adding $-\infty$ to $\R$, tropical addition
now has an identity element.
 On the other hand, a major difference remains between tropical and classical addition: an element of $\R$ does not have an additive ``inverse". Said in another way, tropical subtraction does not exist. Neither can we solve this problem by adding more elements to $\T$ to try to cook up additive inverses. In fact, $``+"$ is said to be \textit{idempotent}, meaning that $``x+x" = x$ for all $x$ in $\T$. 
Our only choice is to  get used to the lack of tropical  additive inverses! 

Despite this last point, the tropical numbers $\T$ equipped with the operations $``+"$ and $``\times"$ satisfy all of the other properties of a field. For example, $0$ is the identity element for tropical multiplication, and every element $x$ of $\T$ different from $-\infty$ has a multiplicative inverse $``\frac{1}{x}" = -x$. 
Then $\T$ satisfies almost all of the axioms of a field, so by convention  we say that it is a   \textit{semi-field}. 

One must take care when writing tropical formulas!  As, $``2x" \neq ``x+x"$ but $ ``2x" = x+2$.
Similarly, $``1x" \neq x$ but $``1x" = x+1$, 
and  
 once 
 again $``0x" = x$ and $``(-1)x" = x-1$.

\subsection{Tropical polynomials}
After having defined tropical addition and multiplication, we naturally come to consider functions of the form 
$P(x) = ``\sum_{i = 0}^d a_ix^i"$ with the $a_i$'s in $\T$, in other words,  tropical polynomials\footnote{In fact, we consider tropical polynomial functions instead of tropical polynomials. Note that two different tropical polynomials may still define the same function.}. 
By rewriting $P(x)$ in classical notation, we obtain
$P(x)=\max_{i=1}^d(a_i + ix ) $. 
Let's look at some examples of tropical polynomials:
$$\tg x\td = x, \ \ \  \tg 1+ x\td = \max(1,x), \ \ \  
\tg 1+ x +3x^2\td =\max(1,x,2x+3), $$
$$\tg 1+ x +3x^2+(-2)x^3\td =\max(1,x,2x+3, 3x-2).$$

Now let's find the roots of a tropical polynomial. Of course, we must
first ask, what is a tropical root? In doing this  we encounter a
recurring problem in tropical mathematics: a classical notion may have
many equivalent definitions, yet when we pass to the tropical world
these could turn out to be different,  as we will see. 
Each equivalent definition of the same classical object potentially
produces as many different tropical objects.

The most basic definition of classical roots of a polynomial $P(x)$ is an element $x_0$ such that $P(x_0) = 0$. If we attempt to replicate this definition in tropical algebra, we must look for elements $x_0$ in $\T$ such that $P(x_0) = -\infty$. Yet, if $a_0$ is the constant term of the polynomial $P(x)$ then $P(x) \geq a_0$ for all $x$ in $\T$. Therefore, if $a_0 \neq -\infty$, the polynomial $P(x)$ 
would not have any
roots... This definition
is surely not adequate.

We may take an alternative, yet equivalent, classical definition: an element,  $x_0 \in \T$ is a classical root of a polynomial $P(x)$ if there exists a polynomial $Q(x)$ such that $P(x) = (x-x_0)Q(x)$. We will soon see that this definition is the correct one for tropical algebra. To understand it, let's take a geometric point of view of the problem. 
 A tropical polynomial is a piecewise linear function and each piece has an integer slope (see Figure 2). 
 What is also apparent from the Figure 2 is that a tropical polynomial is convex, or ``concave up". This is because it is the maximum of a collection of linear functions.

We call \textit{tropical roots} of the polynomial $P(x)$ all points $x_0$ of $\T$ for which the graph of $P(x)$ has a corner at $x_0$. Moreover, the difference in the slopes of the two pieces adjacent to a corner gives the \textit{order} of  the corresponding root. Thus, the polynomial $``0+x"$ has a simple root at $x_0 = 0$, the polynomial $``0 + x + (-1)x^2 \td$ has simple roots $0$ and $1$ and the polynomial $``0+x^2 \td$ has a double root at $0$. 

\begin{figure}[h]
\begin{center}
\begin{tabular}{ccccc}
\includegraphics[width=4cm, angle=0]{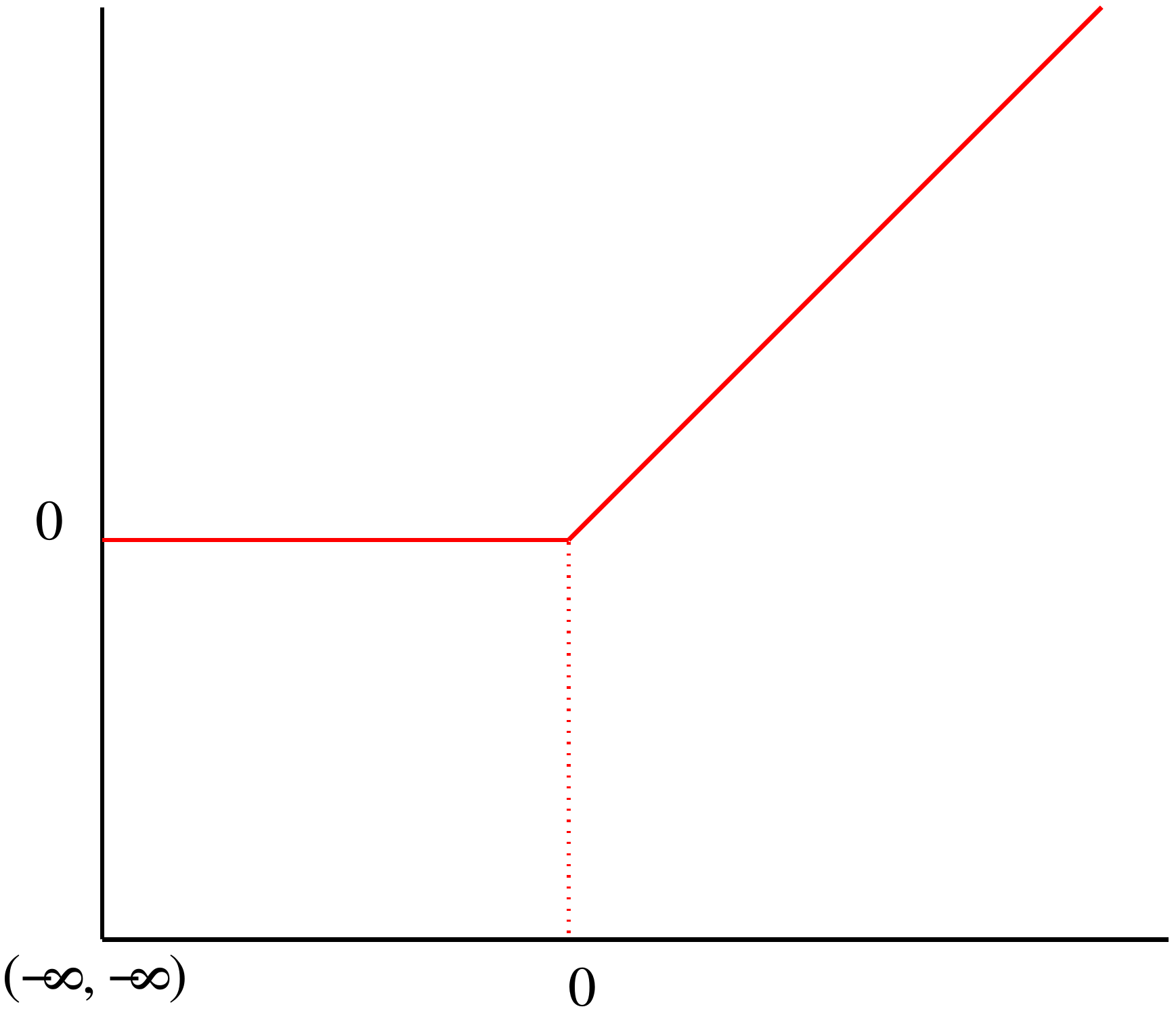}&\hspace{3ex} &
\includegraphics[width=4cm, angle=0]{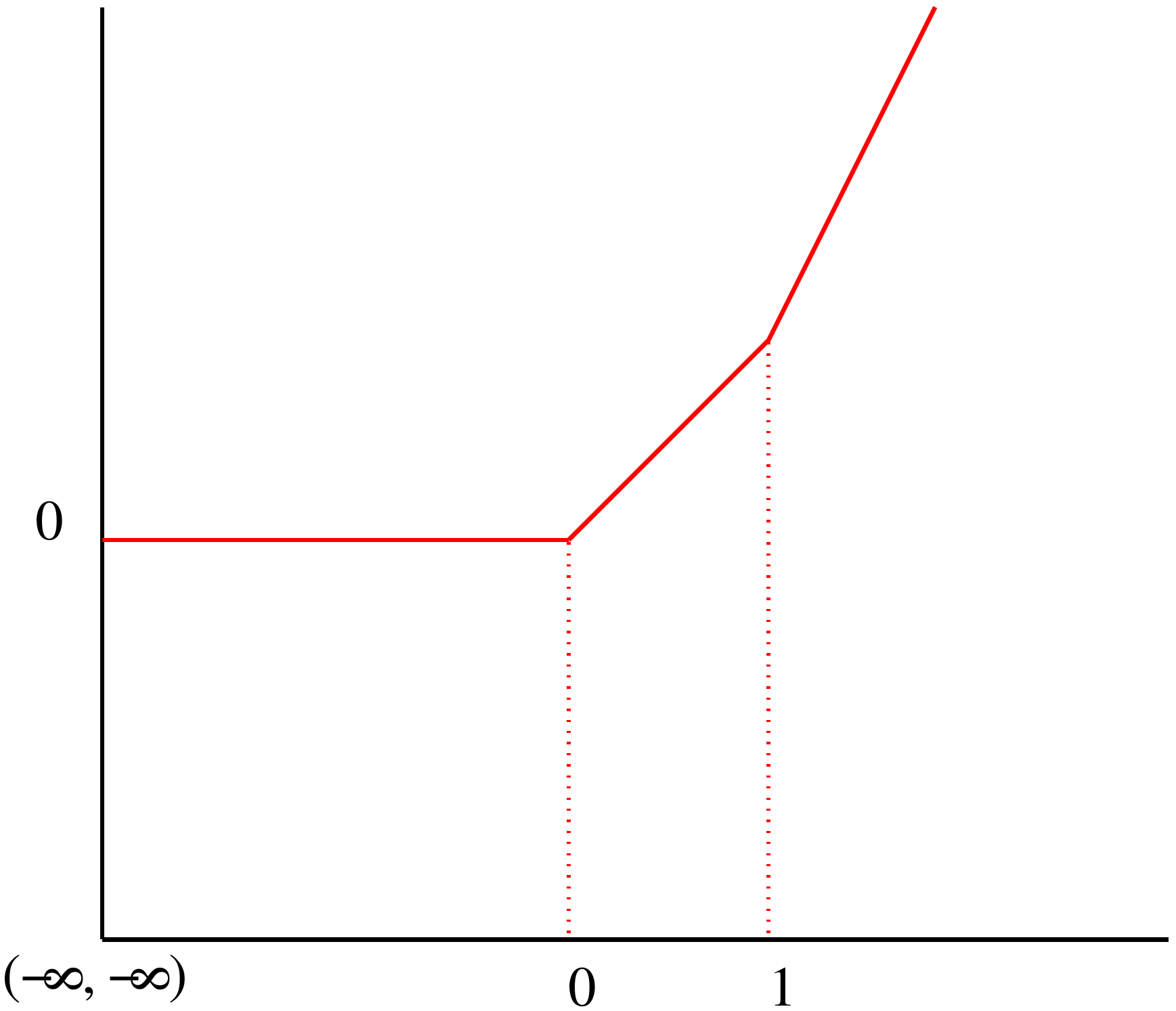}&\hspace{3ex} &
\includegraphics[width=4cm, angle=0]{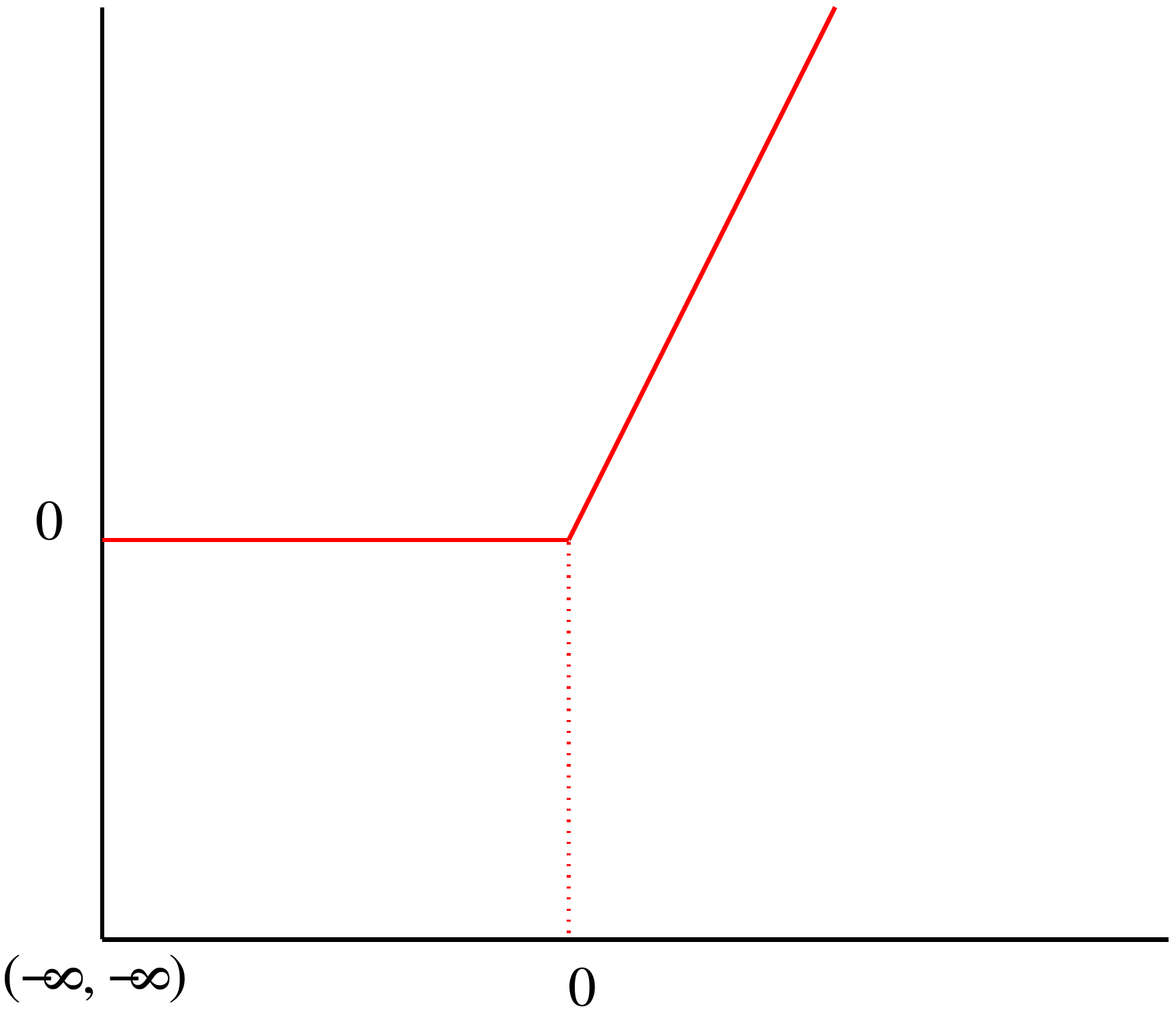}
\\ \\a) $P(x)=\tg 0+x \td$ && b) $P(x)=\tg 0+x + (-1)x^2\td$
 && b) $P(x)=\tg 0+x^2\td$
\end{tabular}
\end{center}
\caption{The graphs of some tropical polynomials}
\label{graphes}
\end{figure}

The roots of a tropical polynomial 
 $P(x)=\tg \sum_{i=0}^d a_ix^i\td=\max_{i=1}^d(a_i + ix )$
 are therefore exactly the tropical numbers $x_0$ for which there exists 
 a pair $i \neq j$ such that 
$P(x_0)=a_i+ix_0=a_j+jx_0$. 
We say that the maximum of $P(x)$ is obtained (at least) twice at $x_0$. 
In this case, the order of the root at $x_0$ is the maximum of $|i-j|$ for all 
possible pairs $i$, $j$ which realise this maximum at $x_0$. 
For example,  the maximum of $P(x)=\tg 0+x+x^2\td$ is obtained 3 times at $x_0 = 0$ and the order of 
this root is 2. Equivalently, $x_0$ is a tropical root of order
at least $k$ of $P(x)$ if
there exists  a
tropical polynomial $Q(x)$ such that $P(x) =\tg (x+x_0)^kQ(x) \td$. 
Note that the factor $x-x_0$ in classical algebra gets transformed to the 
factor  $\tg x+x_0\td$, since the root of the polynomial
$\tg x+x_0\td$ is
$x_0$ and not  $-x_0$.

This definition of a tropical root seems to be much more satisfactory than the first one.
In fact, using this definition we have the following proposition.

\begin{proposition}\label{prop:T alg closed}
The tropical semi-field is algebraically closed. In other words, every tropical polynomial of 
degree $d$ has exactly $d$ roots when counted with multiplicities. 
\end{proposition}

For example, one may check that we have the following factorisations\footnote{Once again the 
equalities hold in terms of polynomial functions not on the level of the 
polynomials.
For example, $\tg 0+x^2\td$ and $ \tg (0+x)^2\td$ are equal as polynomial functions but not as polynomials.}:
$$\tg 0+x + (-1)x^2\td = \tg (-1)(x+0)(x + 1)\td \ \ \ and \ \ \ \tg 0 +
x^2\td = \tg (x+0)^2\td $$

\subsection{Exercises}
\begin{exo}
\begin{enumerate}
\item Why does the idempotent property of tropical addition prevent the
existence of inverses for this operation?
\item Draw the graphs of the tropical polynomials  $P(x)=\tg x^3+2x^2+3x
  +(-1)\td$ and $Q(x)=\tg x^3+(-2)x^2+2x+(-1)\td$, and determine their tropical 
  roots. 
\item Let $a  \in \R$ and  $b, c \in \T$. Determine the roots of the 
polynomials $ \tg ax+b\td$ and $ \tg ax^2+bx+c\td$. 
\item Prove that $x_0$ is a tropical root of order
at least $k$ of $P(x)$ if and only if
there exists  a
tropical polynomial $Q(x)$ such that $P(x) =\tg (x+x_0)^kQ(x) \td$.
\item Prove Proposition \ref{prop:T alg closed}.
\end{enumerate}
\end{exo}

\section{Tropical curves}\label{sec:curves}

\subsection{Definition}
Carrying on 
boldly,
 we can increase the number of variables in our polynomials. A tropical polynomial in two variables is written 
$P(x,y)=\tg \sum_{i,j}a_{i,j}x^iy^j\td $, or better yet
$P(x,y)=\max_{i,j}(a_{i,j}+ix+jy)$ in classical notation. In this way
our tropical polynomial is again a convex piecewise
 linear
 function,
and the \textit{tropical curve} $C$ defined by $P(x, y)$ is the corner locus of
this function. Said in another way, a tropical curve $C$ consists of
all points $(x_0, y_0)$ in $\T^2$ for which the maximum of $P(x, y)$ is obtained at
least twice at $(x_0, y_0)$.

We should point out that 
up until Section \ref{sec:further}
we will focus 
on tropical curves contained in $\R^2$ and not in $\T^2$. This 
does not affect at all the generality of what will be discussed here, however it renders the definitions, the statements, and our drawings simpler and easier to understand.

Let us look
at the tropical line defined by the polynomial
$P(x,y)=\tg \frac{1}{2}+2x +(-5)y \td$. 
We must find the points 
$(x_0,y_0)$ in $\R^2$ that satisfy one of the following three 
systems of equations:
$$2+x_0=\frac{1}{2}\ge -5+y_0, \ \ \ \ \ \ \ \ \ -5+y_0=\frac{1}{2}\ge 2+x_0,
\ \ \ \ \ \ \ \ \ 2+x_0=-5+y_0\ge \frac{1}{2}.$$ 

We see that our tropical line is made up of three standard half-lines
$$\{(-\frac{3}{2},y) \ | \ y\le \frac{11}{2} \},  \  \{(x,\frac{11}{2}) \ |
\ x\le -\frac{3}{2} 
\}, \text{ and } \{(x,x+7) \ | 
\ x\ge  -\frac{3}{2} \}$$ (see  Figure \ref{droite}a).

\begin{figure}[h]
\begin{center}
\begin{tabular}{ccc}
\includegraphics[width=4cm, angle=0]{Figures/Droite.pdf}&
\includegraphics[width=4cm, angle=0]{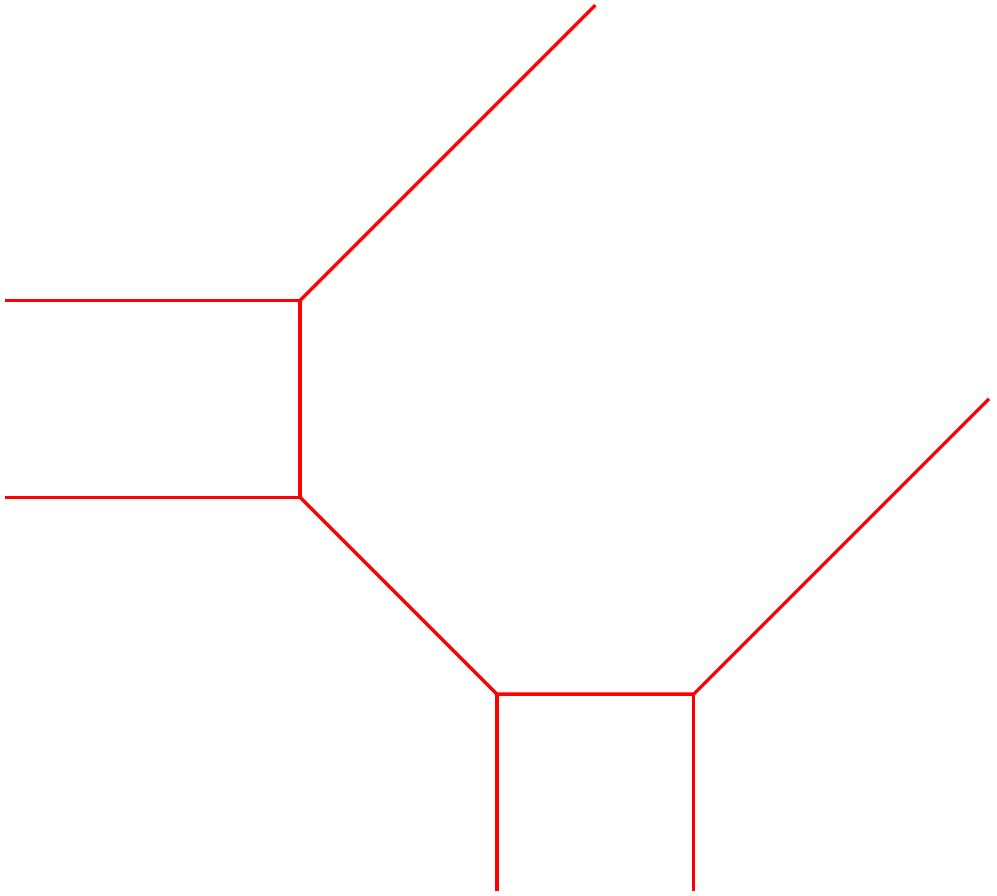}&
\includegraphics[width=4cm, angle=0]{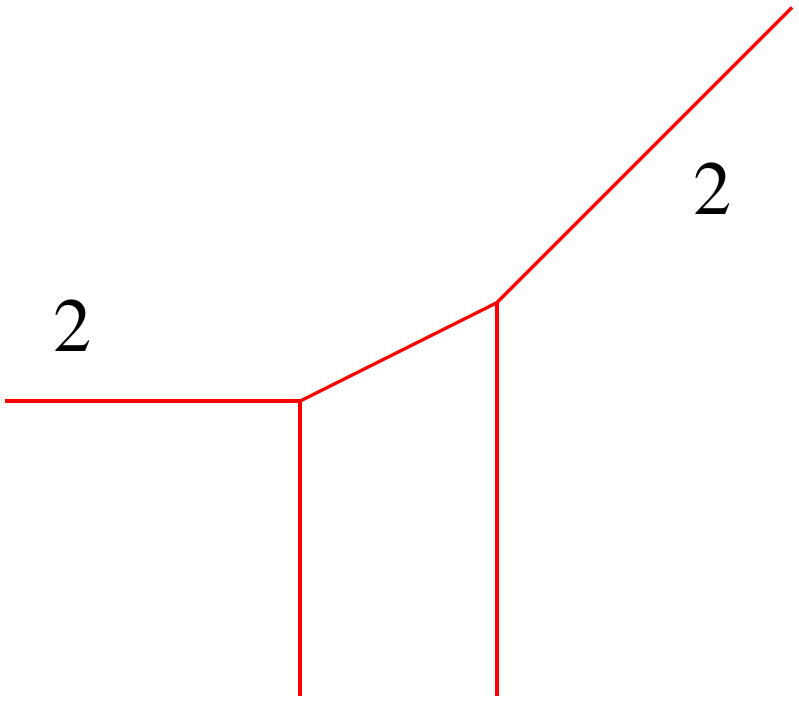}
\\ \\a) $\tg \frac{1}{2}+2x +(-5)y\td$ & b)  $\tg 3+ 2x + 2y + 3xy+y^2+x^2\td$ 
 & c) $\tg 0+ x + y+y^2+(-1)x^2\td$

\end{tabular}
\end{center}
\caption{Some tropical curves}
\label{droite}
\end{figure}

We are still missing one bit of information to properly define a tropical curve. 
The corner locus of a tropical polynomial in two variables consists of 
line segments and half-lines, which we call \textit{edges}. These intersect at points
which we will call \textit{vertices}. Just as in the case of polynomials in one variable,
for each edge of a tropical curve, we must take into account the difference in the slope of $P(x, y)$ 
on the two sides of the edge. 
Doing this, we arrive at the following formal definition of a tropical curve.

\begin{defi}
Let $P(x,y)=\tg \sum_{i,j} a_{i,j}x^iy^j\td $ be a tropical polynomial.  The tropical curve 
 $C$ defined by  $P(x,y)$  is the set of points $(x_0,y_0)$ of $\R^2$ 
 such that there exists pairs  $(i,j)\ne (k,l)$ satisfying
$P(x_0,y_0)=a_{i,j}+ix_0+jy_0=a_{k,l}+kx_0+ly_0$.

We define the weight $w_e$ of an edge $e$ of $C$ to be the maximum of the 
greatest common divisor (gcd) of the numbers $|i-k|$ and 
$|j-l|$ for all pairs $(i,j)$ and $(k,l)$ which correspond to this edge. That is to say,
$$w_e=\max_{M_e} \left(\gcd(|i-k|, |j-l|) \right), $$
where 
$$M_e=\{(i,j), (k,l)\ |\ \forall x_0\in e, P(x_0,y_0)=a_{i,j}+ix_0+jy_0=a_{k,l}+kx_0+ly_0 \}.$$
\end{defi}

In Figure \ref{droite}, the weight of an edge is only indicated
if  the weight 
is at least two. For example, in the case of the tropical 
line, all edges are of weight 1. Thus, Figure \ref{droite}a represents 
fully the tropical line. Two examples of tropical curves of degree 2 are
shown in Figure \ref{droite}b and c. The tropical conic in Figure \ref{droite}c 
has two edges of weight 2. 
Note that given an edge $e$ of a tropical curve defined by $P(x,y)$, the set $M_e$ can have any cardinality 
between $2$ and $w_e+1$: in the example of Figure \ref{droite}c, we have $M_e=\{(0,0), (0,1),(0,2) \}$ if $e$ is the 
 horizontal edge, and $M_e=\{(2,0),(0,2) \}$ if $e$ is the other edge of weight $2$.

\subsection{Dual subdivisions}
To recap, a  tropical polynomial is given by the maximum of a finite number of 
linear
functions corresponding to monomials of $P(x, y)$. Moreover, the points of the plane $\R^2$ for which  at least two of these monomials realise the maximum are exactly the points of the tropical curve $C$ defined by $P(x, y)$. Let us refine this a bit and consider at each point $(x_0, y_0)$ of $C$, \textit{all} of the monomials of $P(x, y)$ that realise the maximum at $(x_0, y_0)$. 

Let us first go back to the tropical line $C$ defined by the equation 
$P(x,y)=\tg 
\frac{1}{2}+2x+(-5)y\td$ (see  
Figure \ref{droite}a). 
The point
 $(-\frac{3}{2},\frac{11}{2})$ is the vertex of the line $C$. This is where the three monomials 
 $\frac{1}{2}=\frac{1}{2}x^0y^0$, 
$2x=2x^1y^0$ and
$(-5)y=(-5)x^0y^1$ take the same value.
The exponents of those monomials, that it to say the points $(0, 0)$,
$(1, 0)$ and $(0, 1)$, define a triangle $\Delta_1$
(see Figure \ref{subd}a). Along the horizontal edge of $C$ the value of the polynomial $P(x, y)$ is given by the
monomials $0$ and $y$, 
in other words, the monomials 
with exponents $(0, 0)$ and $(0, 1)$.
Therefore, these two exponents define the vertical edge of the
triangle $\Delta_1$.  In the same way, the monomials giving the value
of $P(x, y)$ along the vertical edge of $C$ 
have exponents
$(0, 0)$ and $(1,
0)$,
which
define the horizontal edge of $\Delta_1$. 
Finally, along the edge of $C$ that has slope 1,  $P(x, y)$ is given
by the monomials 
with exponents
$(1, 0)$ and $(0,1)$,
which 
define the edge of $\Delta_1$ that has slope $-1$.

What can we learn from this digression? In looking at the monomials which give the value of the tropical polynomial $P(x, y)$ at a point of the tropical line $C$, we notice that the vertex of $C$ corresponds to the triangle $\Delta_1$ and that each edge $e$ of $C$ corresponds to an edge $\delta_e$ of $\Delta_1$, whose direction is perpendicular to that of $e$. 
\vspace{1ex}

Let us illustrate this with the  tropical conic defined by the polynomial $P(x, y) = \tg
3+2x+2y+3xy+x^2+y^2 \td$, and drawn in Figure \ref{droite}b. This curve has as its vertices 
four points  $(-1,1)$, $(-1,2)$, $(1,-1)$ and $(2,-1)$. At each of these vertices
 $(x_0,y_0)$, the value of the polynomial  $P(x,y)$ 
 is given by three monomials:
$$P(-1,1)=3=y_0+2=x_0+y_0+3
\ \ \ \ \ \ \ \ \ \ \ \ \ \ \ \ P(-1,2)=y_0+2=x_0+y_0+3=2y_0 \ \  $$
$$P(1,-1)=3=x_0+2=x_0+y_0+3
\ \ \ \ \ \ \ \ \ \ \ \ \ \ \ \ P(2,-1)=x_0+2=x_0+y_0+3=2x_0. \ \  $$
Thus for each vertex of $C$, the exponents of the three 
corresponding 
monomials 
define a triangle, and these four triangles are arranged as shown in 
 Figure \ref{subd}b. 
 Moreover, just as in the case of the line, for each edge $e$ of $C$, the exponents
 of the monomials giving the value of $P(x, y)$ along the edge $e$ define an edge 
 of one (or two) of these triangles. Once again the direction of the edge $e$ is
 perpendicular to the corresponding edge of the triangle.

\vspace{1ex}

To explain this phenomenon in full generality, let 
$P(x,y)=\tg \sum_{i,j} a_{i,j}x^iy^j\td$ be any tropical polynomial. 
 The  \textit{degree} of $P(x,y)$ is the maximum of the sums
 $i+j$ for all coefficients $a_{i, j}$ different from $-\infty$. For simplicity, 
 we will assume in this text that all polynomials of degree $d$ satisfy
 $a_{0,0}\ne -\infty$, $a_{d,0}\ne  -\infty$ and $a_{0,d}\ne
-\infty$. 
Thus, all the points $(i, j)$ such that $a_{i, j} \neq -\infty$ are contained in 
the triangle with vertices $(0, 0)$, $(0, d)$ and $(d, 0)$, which we call $\Delta_d$.
Given a finite set of points $A$ in $\RR^2$, the \emph{convex hull} of $A$ is the unique convex polygon  with vertices in $A$ and containing\footnote{Equivalently, it is the smallest convex polygon containing $A$.} $A$. From what we just said, the triangle $\Delta_d$ is precisely the convex hull of the points $(i, j)$ such that $a_{i, j} \neq -\infty$.

If $v = (x_0,y_0)$ is a vertex of the curve $C$ defined by $P(x, y)$, then the convex hull 
of the points $(i, j)$ in $\Delta_d \cap \Z^2$ such that 
 $P(x_0,y_0)=a_{i,j}+ix_0+jy_0$ is another polygon 
$\Delta_v$ 
which is contained in $\Delta_d$.
Similarly, if $(x_0, y_0)$ is a point in the interior of an edge $e$ of $C$, then 
the convex hull of the points $(i, j)$ in $\Delta_d\cap\ZZ^2$ such that
$P(x_0,y_0)=a_{i,j}+ix_0+jy_0$ is a segment 
$\delta_e$ contained in  $\Delta_d$.
The fact that 
 the tropical polynomial $P(x,y)$ is a convex piecewise 
 linear function implies that 
the collection of all  $\Delta_v$ form a \emph{subdivision} of
$\Delta_d$. 
In other words, the union of all of the polygons $\Delta_v$ is 
equal to the triangle $\Delta_d$,  and two polygons $\Delta_v$ and
$\Delta_{v'}$ have either an edge in common, a vertex in common, or do not intersect at all. 
Moreover, if $e$ is an edge of $C$ adjacent to the vertex $v$, then $\delta_e$ is an edge 
of the polygon $\Delta_v$, and  
$\delta_e$ is perpendicular to $e$. 
In particular, an edge $e$ of $C$ is infinite, i.e. is adjacent to only one vertex of $C$,  if and only if $\Delta_e$ is contained in an edge of $\Delta_d$. 
This subdivision of $\Delta_d$ is called the \textit{dual subdivision} of $C$.

For example, the dual subdivisions of the tropical curves in 
Figure \ref{droite} are drawn in 
 Figure \ref{subd} (the black points represent the points of $\R^2$ with integer
coordinates; notice they are not necessarily the vertices of the dual subdivision).

 The weight of an edge may be read off directly from the dual subdivision.
 \begin{prop}\label{prop:weight edge}
  An edge $e$ of a tropical curve has weight $w$ if and only if $Card(\Delta_e\cap \ZZ ^2 ) = w+1$.
 \end{prop}
 It follows  from Proposition \ref{prop:weight edge} that the degree
 of a tropical curve may be 
 determined 
 easily only from the curve itself: it is the sum
 of weights of all infinite the edges in the direction $(-1, 0)$ (we could equally consider the
 directions $(0, -1)$ or $(1, 1)$). 
 Moreover,  up to a translation and choice of 
 lengths of its edges,  a tropical curve is determined by its dual subdivision. 

\begin{figure}[h]
\begin{center}
\begin{tabular}{ccccc}
\includegraphics[width=1cm, angle=0]{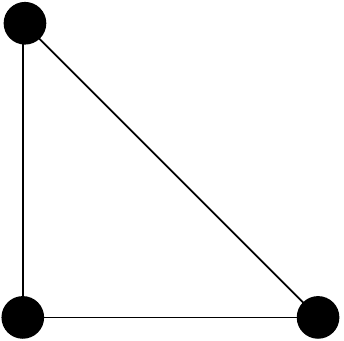}&\hspace{6ex} &
\includegraphics[width=2cm, angle=0]{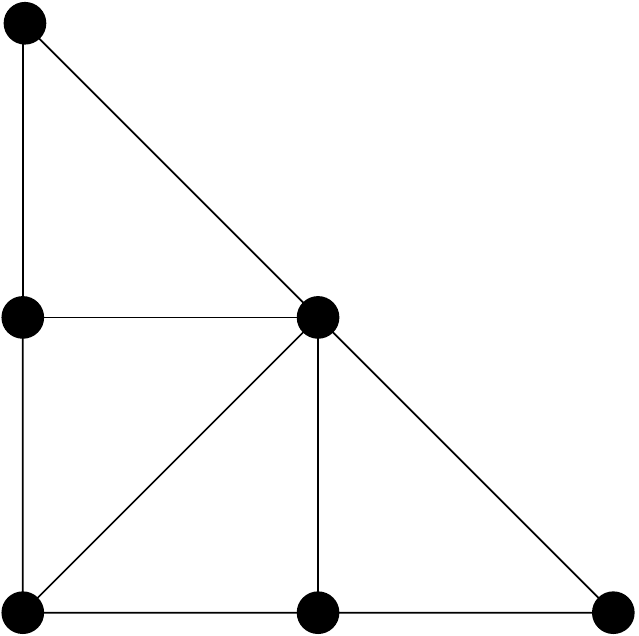}&\hspace{6ex} &
\includegraphics[width=2cm, angle=0]{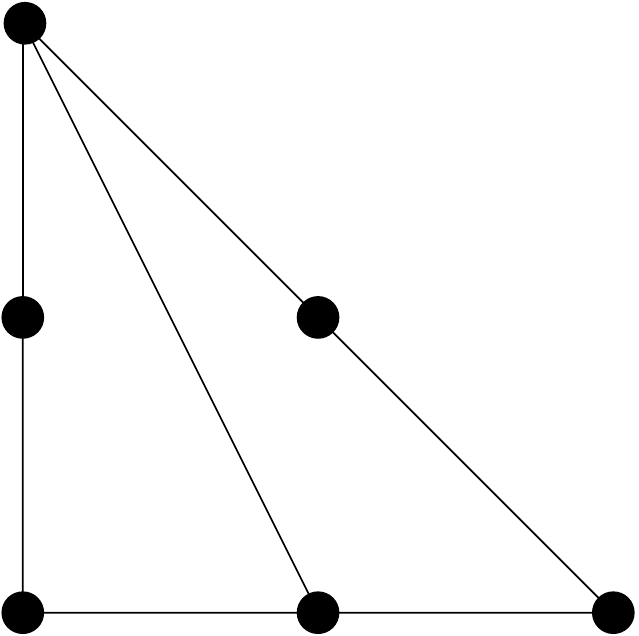} 
\\ \\a) && b) &&c)

\end{tabular}
\end{center}
\caption{Some dual subdivisions}
\label{subd}
\end{figure}

\subsection{Balanced graphs and tropical curves.}\label{sec:balancing}

The first consequence of 
the duality 
 from  the last section 
is that a certain relation, known
as the \textit{balancing condition}, is satisfied at each vertex of a tropical
curve. Suppose $v$ is a vertex of $C$ adjacent to the edges $e_1,
\dots, e_k$ with respective weights $w_1, \dots , w_k$. Recall that
every edge $e_i$ is 
contained in
 a line (in the usual sense) defined
by an equation with integer coefficients.
 Because of this there exists a unique integer vector $\vec v_i =
 (\alpha, \beta)$ in the direction of $e_i$ such that
 $\text{gcd}(\alpha, \beta) = 1$ (see Figure  \ref{equ}a).
We  orient the boundary of $\Delta_v$ in the counter-clockwise direction,
 so that each edge $\delta_{e_i}$ of $\Delta_v$ dual to $e_i$ is
 obtained from a vector $w_i \vec v_i$ by rotating by an angle of exactly $\pi /2$ (see Figure
\ref{equ}b). Then, following the previous section, the polygon $\Delta_v$ dual to $v$ yields
 immediately the vectors $w_1 \vec v_1, \dots , w_k\vec v_k$.

\begin{figure}[h]
\begin{center}
\begin{tabular}{ccc}
\includegraphics[width=4cm,
  angle=0]{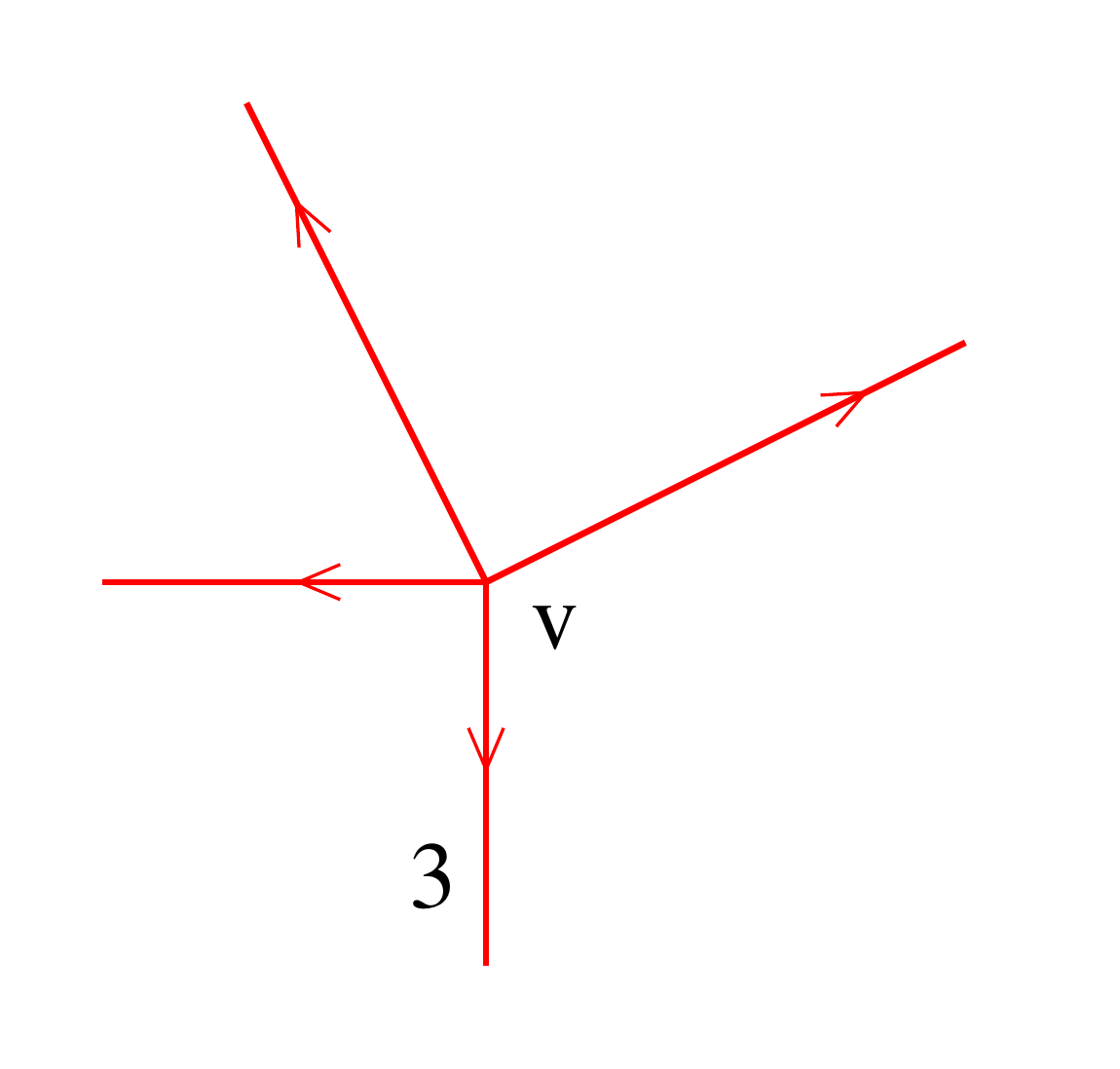}& \hspace{4ex} &
\includegraphics[width=4cm, angle=0]{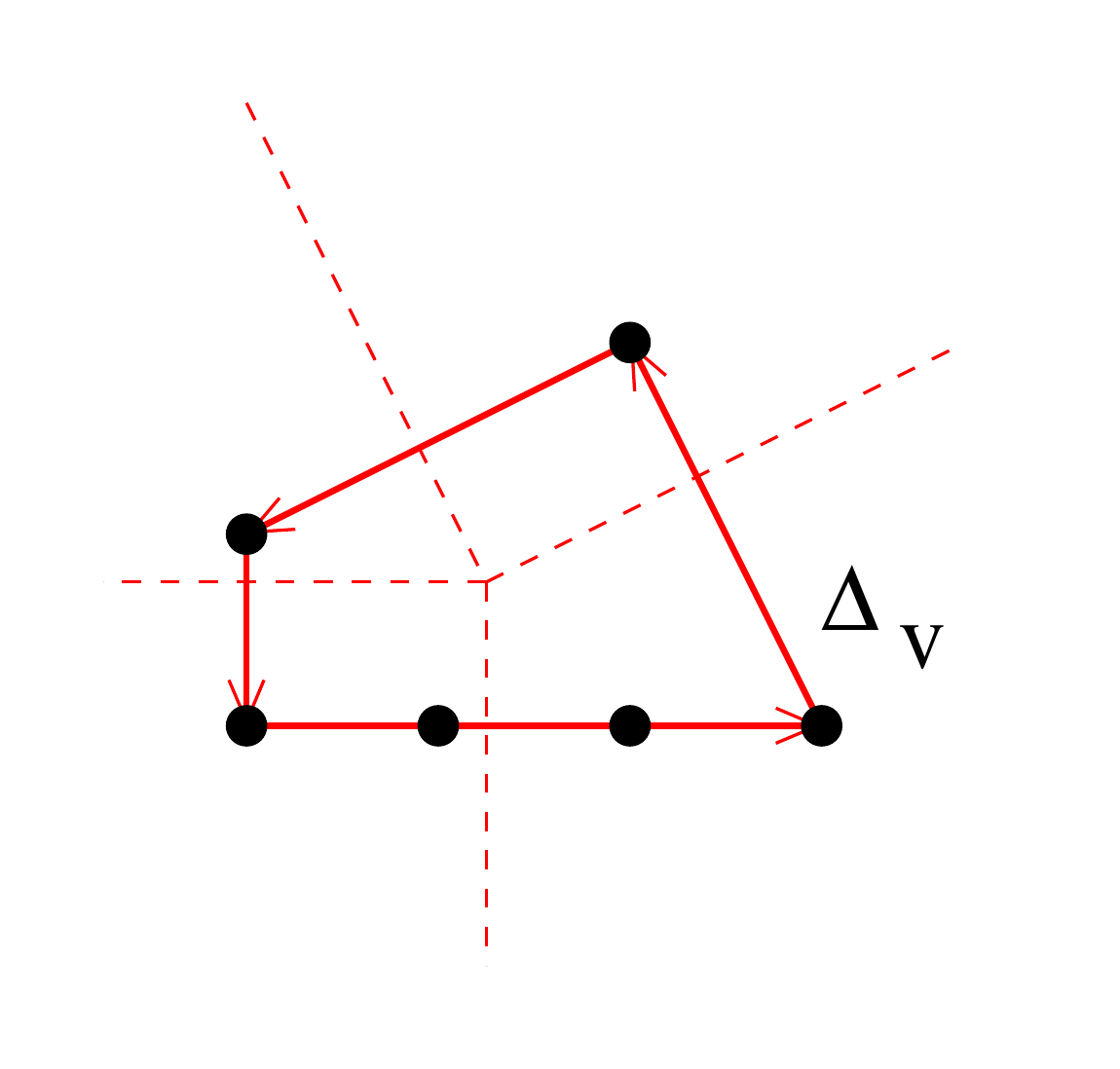}
\\ a) && b)
\end{tabular}
\end{center}
\caption{Balancing condition}
\label{equ}
\end{figure}

The fact that the polygon $\Delta_v$ is closed immediately implies the 
following balancing condition: 

$$\sum_{i=1}^k w_i\vec v_i=0.$$

A graph in $\R^2$ whose edges have rational slopes and
are equipped with positive 
integer weights is a \textit{balanced graph} if it satisfies the balancing condition
 at each one of its vertices. We have just 
 seen that every tropical curve is a balanced graph. In fact, the converse is 
 also true.

\begin{thm}[G. Mikhalkin]
Tropical curves in $\R^2$ are exactly the balanced graphs. 
\end{thm}

Thus, this theorem affirms that there exist tropical polynomials of degree 3 whose 
tropical curves are the weighted graphs in Figure \ref{equil}. 
We have also drawn for each curve, the associated dual subdivision of $\Delta_3$.

\begin{figure}[h]
\begin{center}
\begin{tabular}{ccccc}
\includegraphics[width=4cm, angle=0]{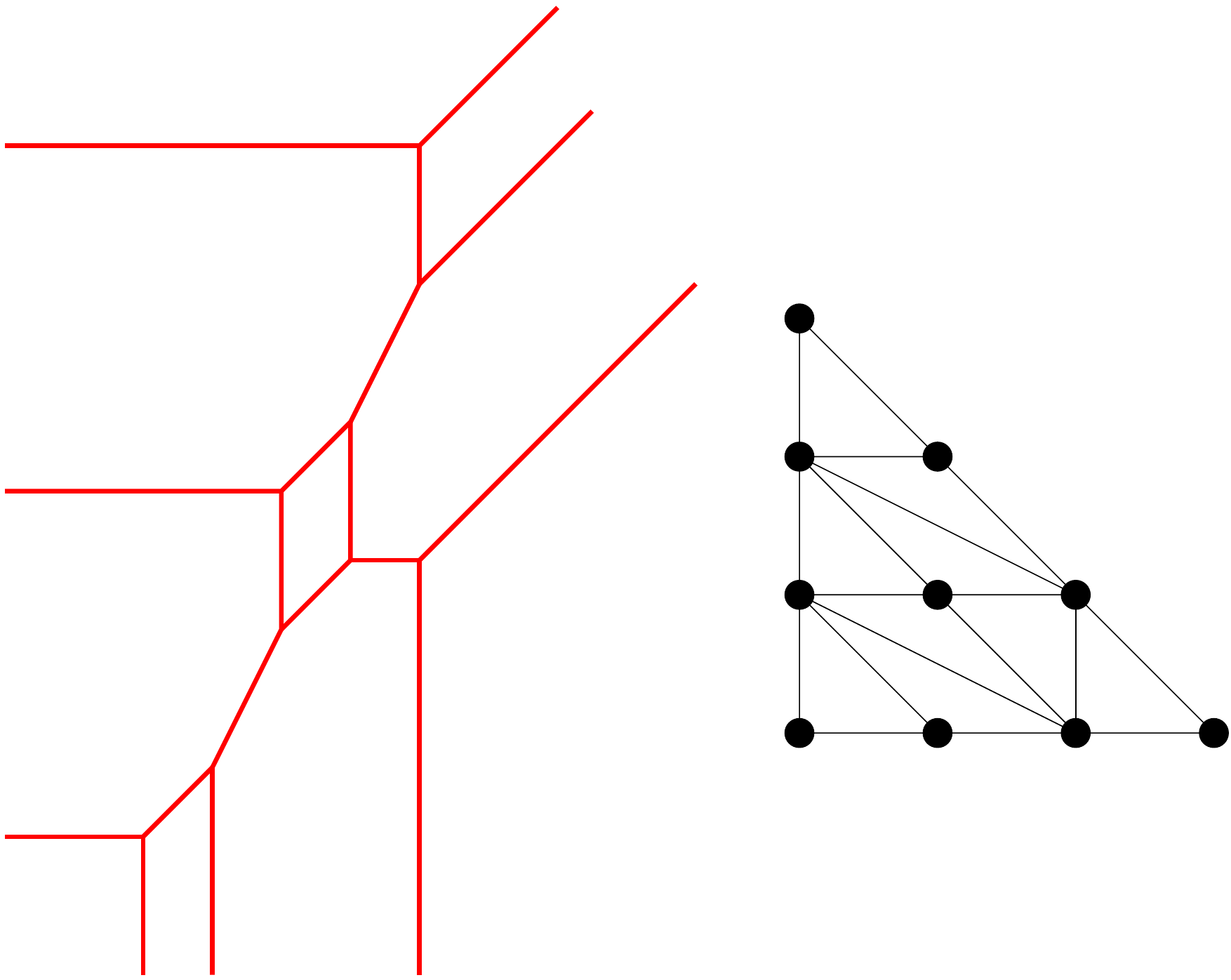}&\hspace{3ex} &
\includegraphics[width=4cm, angle=0]{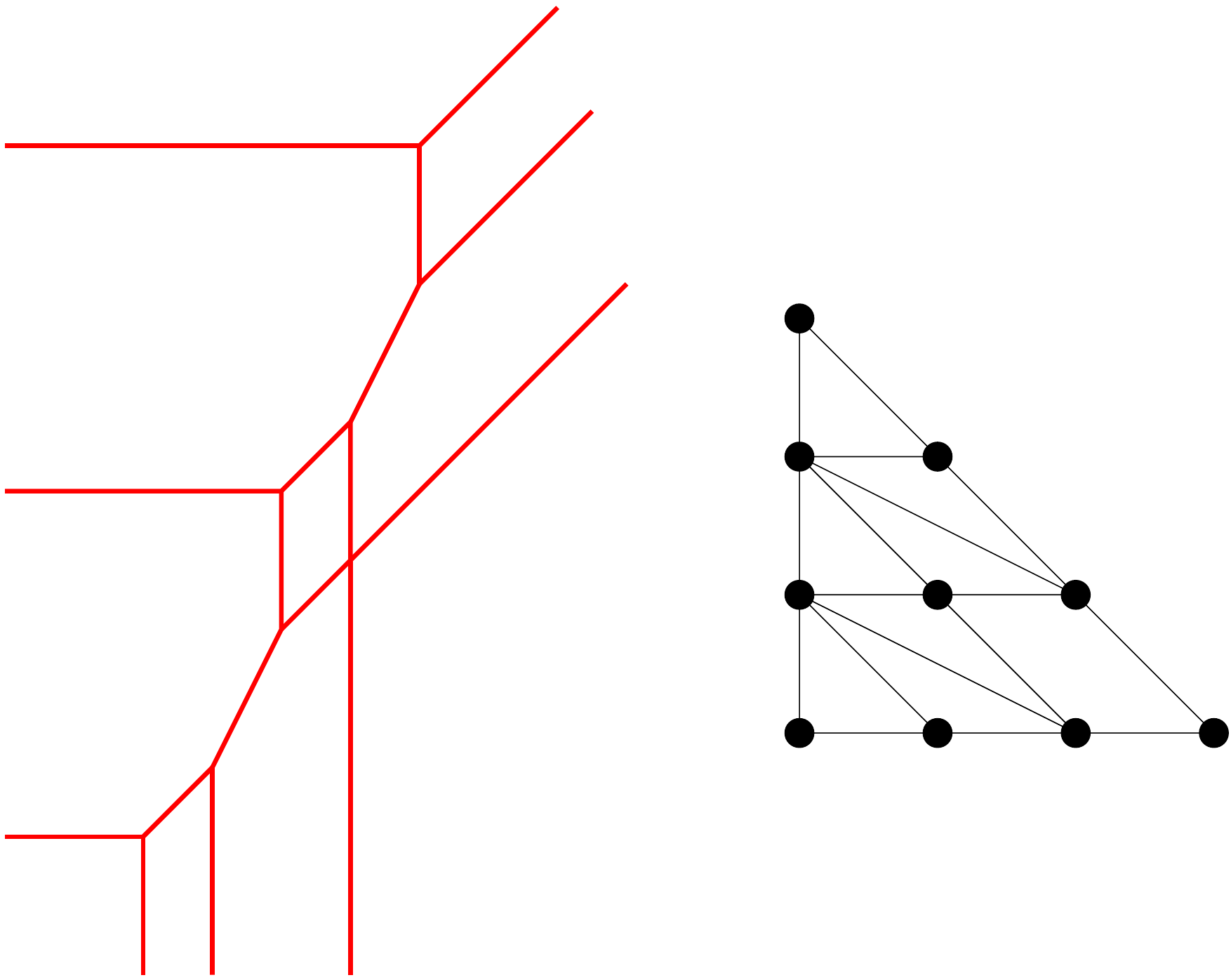}&\hspace{3ex} &
\includegraphics[width=4cm, angle=0]{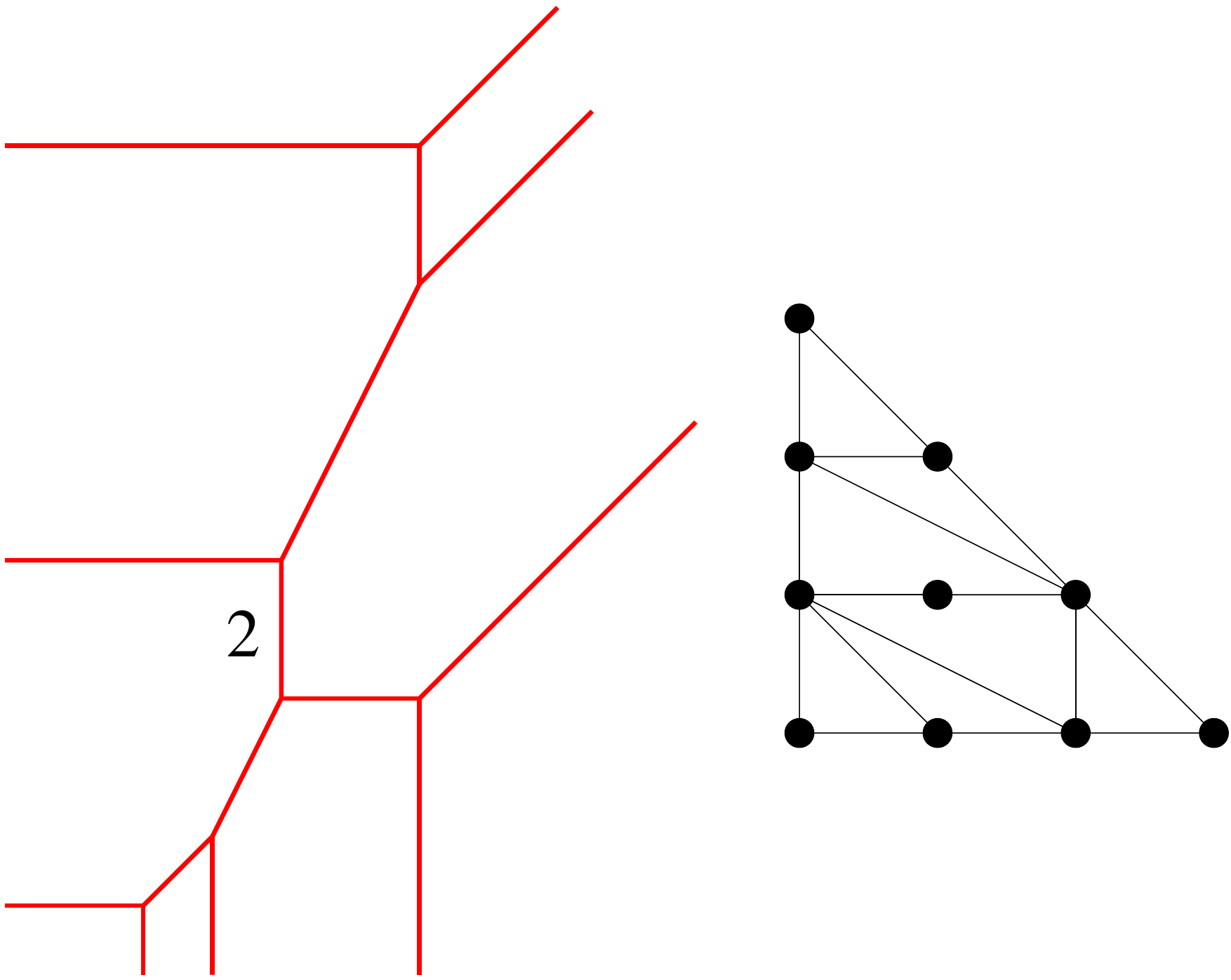}
\\ \\a) && b) &&c)

\end{tabular}
\end{center}
\caption{}
\label{equil}
\end{figure}

\subsection{Exercises}
\begin{exo}
\begin{enumerate}
\item Draw the tropical curves defined by the tropical polynomials 
  $P(x,y)=\tg 5 + 5x + 5y + 4xy+1y^2+x^2\td$ and $Q(x,y)=\tg  7+ 4x + y
  + 4xy+3y^2+(-3)x^2\td$, as well as their dual subdivisions. 
  
\item  A tropical triangle is a domain of $\R^2$ bounded by 
 three tropical lines. What are the possible forms of a tropical triangle?

\item Prove Proposition \ref{prop:weight edge}.
 
\item Show that a tropical curve of degree $d$ has at most $d^2$ vertices. 

\item Find an equation for each of the tropical curves in Figure \ref{equil}. 
The following reminder might be helpful: if $v$ is a vertex of a tropical curve defined by a 
tropical polynomial $P(x, y)$, then the value of $P(x, y)$ in a neighbourhood of $v$ is given
uniquely by the monomials corresponding to the polygon dual to $v$. 

\end{enumerate}
\end{exo}

\section{Tropical intersection theory}\label{sec:trop int}

\subsection{B\'ezout's theorem.}

One of the main interests of tropical geometry is to provide 
a simple model of algebraic geometry. For example, the basic 
theorems from intersection theory of tropical curves require much 
less mathematical background than their 
classical counterparts. The theorem which we have in mind is 
B\'ezout's theorem, which states that  two algebraic curves in the
plane of degrees $d_1$ and $d_2$ respectively,  intersect in $d_1d_2$ 
points\footnote{Warning, this is a theorem in projective geometry!
  For example, we may have two lines in the 
  classical  plane which are
  parallel: if they do not intersect in the plane, they intersect
  nevertheless \textit{at infinity}. 
Also 
one has to count intersection
  points with multiplicity. For example a tangency point between two
  curves will count as 2 intersection points.}.
Before we  tackle the general case, let us first consider tropical lines and
conics. 

As mentioned in the introduction, most of the time two tropical lines
 intersect in exactly one point 
(see  Figure \ref{inter}a), just as in classical geometry.
Now, do a tropical line and a tropical conic intersect each other in two points?
If we naively count the number of intersection points, the answer is sometimes yes (Figure \ref{inter}b)
and sometimes no  (Figure
\ref{inter}c)... 

\begin{figure}[h]
\begin{center}
\begin{tabular}{ccccc}
\includegraphics[width=3cm, angle=0]{Figures/InterDte.pdf}&\hspace{3ex} &
\includegraphics[width=4cm, angle=0]{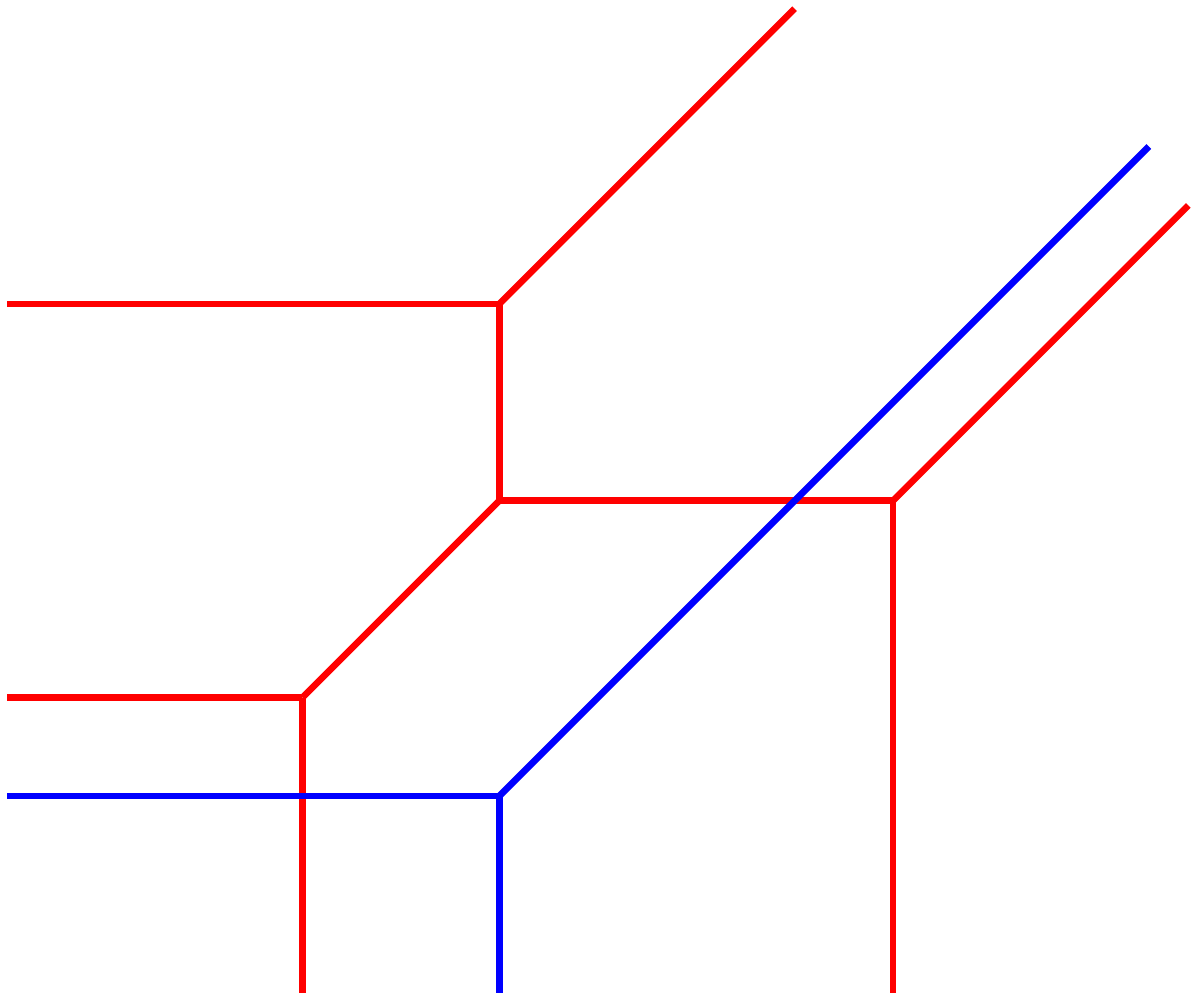}&\hspace{3ex} &
\includegraphics[width=4cm, angle=0]{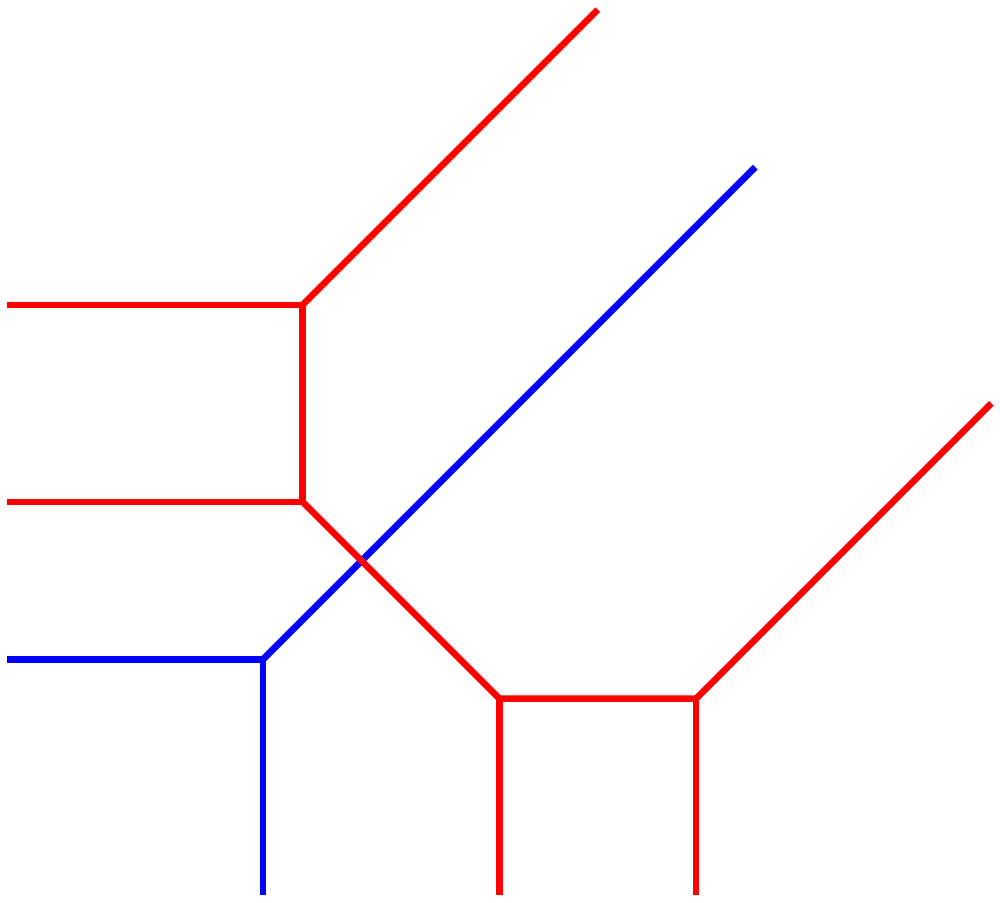} 
\\ \\a) && b) &&c)

\end{tabular}
\end{center}
\caption{Intersections of tropical lines and conics}
\label{inter}
\end{figure}

In fact, the unique intersection point of the conic and the tropical line in  Figure 
\ref{inter}c should be counted twice. But why twice in this case and not in the 
previous case? To find the answer we  look to the dual subdivisions. 

We start by restricting to the case when the curves $C_1, C_2$ intersect in a finite  collection of 
points away from the vertices of both curves. 
Remark that  the union of the two tropical curves $C_1$ and $C_2$
is again a tropical curve. In fact, we can easily verify that the union of two balanced 
graphs is again a balanced graph. Or, we could also easily check that if 
$C_1$ and $C_2$ are defined by tropical polynomials $P_1(x, y)$, $P_2(x,y)$ respectively, 
then $Q(x, y) = ``P_1(x,y)P_2(x,y)"$ defines precisely the curve $C_1 \cup C_2$. 
Moreover, the degree of $C_1 \cup C_2$ is the sum of the degrees of $C_1$ and $C_2$. 

The dual subdivisions of the unions of the two curves $C_1$ and $C_2$ in the 
three cases of Figure  \ref{inter} are represented in 
Figure  \ref{sub inter}. 
In every case, the 
set of vertices of $C_1 \cup C_2$ 
is the union of
the vertices of $C_1$, the vertices of
  $C_2$, and 
  of
  the intersection points of $C_1$ and $C_2$. Moreover, since each point of 
  intersection of $C_1$ and $C_2$ is contained in an edge of both $C_1$ and $C_2$, the 
  polygon dual to such a vertex of $C_1 \cup C_2$ is a parallelogram. To make 
   Figure \ref{sub inter} more transparent, we have drawn each edge of the dual subdivision in the
   same colour as its corresponding dual edge. We can conclude that in 
Figures \ref{sub inter}a and b, the corresponding parallelograms are of area one, whereas the 
corresponding  parallelogram of the dual subdivision in Figure 
\ref{sub inter}c has area two! Hence, it seems that we are counting each intersection point with 
the multiplicity we will describe below. 

\begin{figure}[h]
\begin{center}
\begin{tabular}{ccccc}
\includegraphics[width=2cm, angle=0]{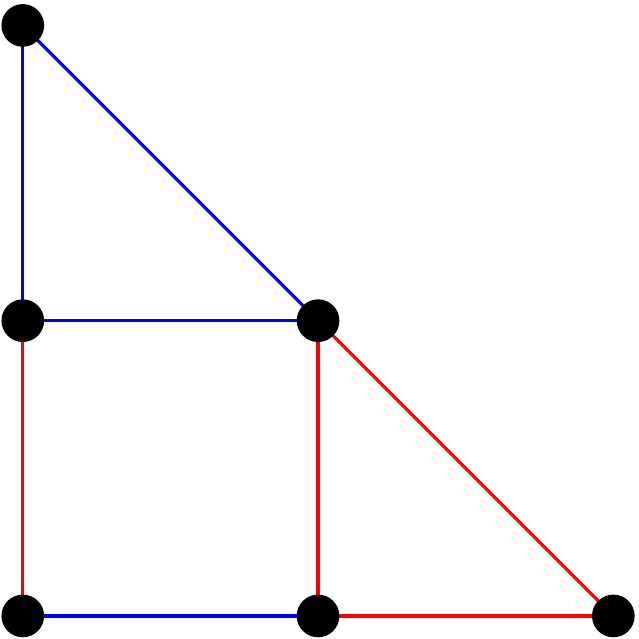}&\hspace{3ex} &
\includegraphics[width=3cm, angle=0]{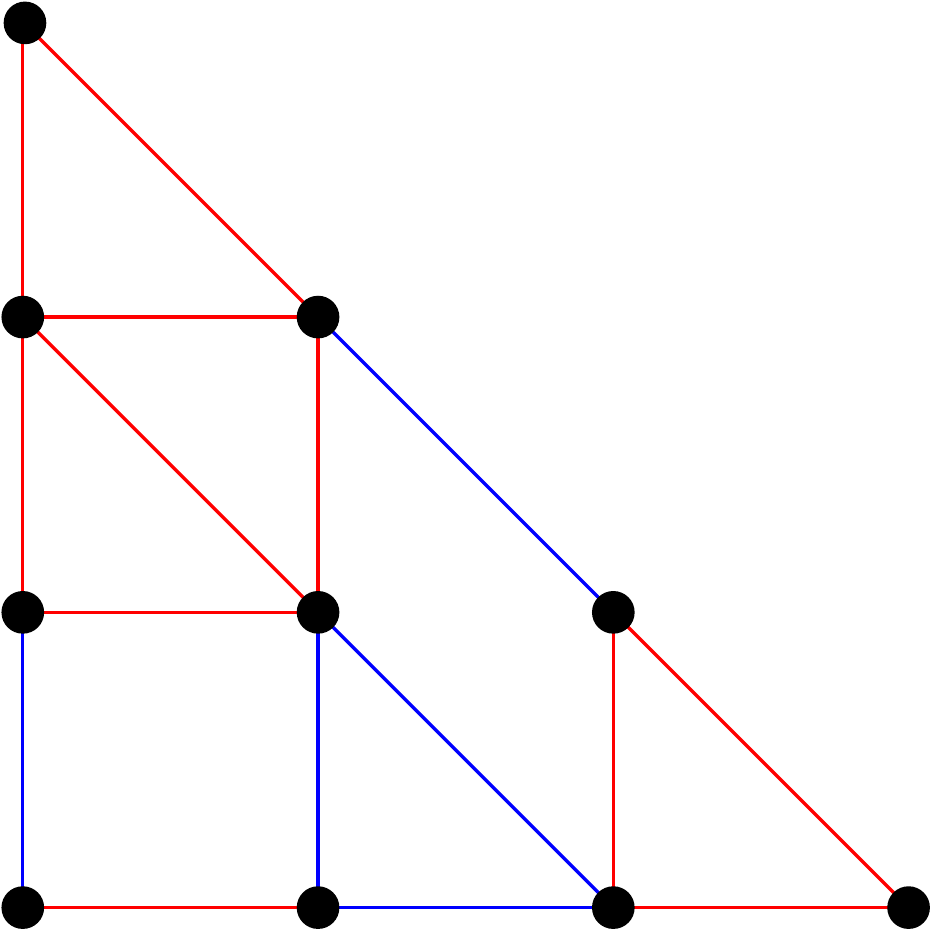}&\hspace{3ex} &
\includegraphics[width=3cm, angle=0]{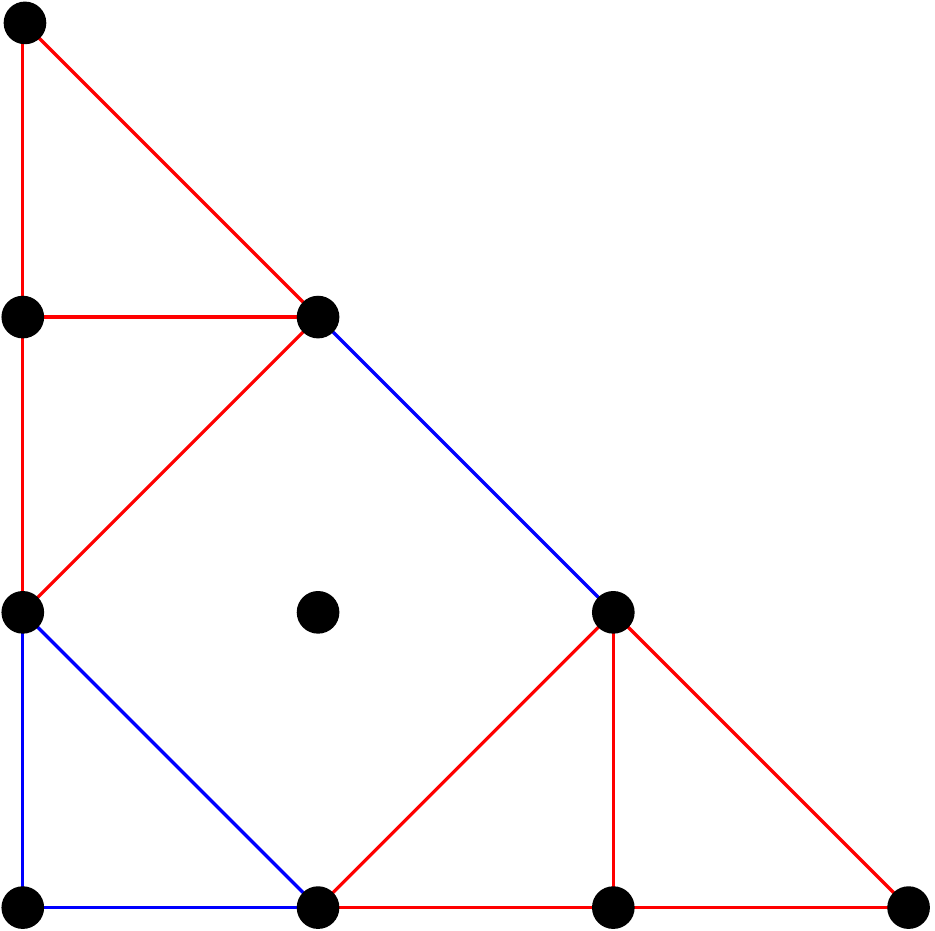} 
\\ \\a) && b) &&c)

\end{tabular}
\end{center}
\caption{The subdivisions dual to the union of the curves  in Figure \ref{inter}}
\label{sub inter}
\end{figure}

\begin{defi}
Let $C_1$ and $C_2$ be two tropical curves which intersect in a 
finite number of points and away from the vertices of the two curves. If
$p$ is a point of intersection of $C_1$ and $C_2$, the tropical multiplicity
of $p$ as an intersection point of $C_1$ and $C_2$ is the area of the parallelogram
dual to $p$ in the dual subdivision of $C_1 \cup C_2$.
\end{defi}

With this definition, proving the tropical B\'ezout's theorem is
a walk in the park!

\begin{thm}[B. Sturmfels]
Let $C_1$ and  $C_2$ be two tropical curves of degrees $d_1$ and $d_2$ respectively, 
 intersecting in a finite number of points away from the vertices of the two curves. 
Then the sum of the tropical multiplicities of all points in the intersection of 
 $C_1$ and $C_2$ is equal to $d_1d_2$.
\end{thm}

\begin{proof}
Let us call this sum of multiplicities  $s$. 
Notice that there are three types of polygons in the
subdivision 
dual
to the tropical curve  $C_1\cup C_2$: 
\begin{itemize}
\item those which are dual to a vertex of $C_1$. The sum of their areas is equal to
the area of  $\Delta_{d_1}$, in other words  $\frac{d_1^2}{2}$,
\item  those which are dual to a vertex of $C_2$. The sum of their areas is equal to
  $\frac{d_2^2}{2}$,
\item those dual to a intersection point of $C_1$ and $C_2$. The sum of their areas we
have called $s$.
\end{itemize}
Since the curve $C_1\cup C_2$ is of degree $d_1 + d_2$, the sum of the area of all 
of these polygons is equal to the area of 
$\Delta_{d_1+d_2}$, which is  $\frac{(d_1+d_2)^2}{2}$. Therefore, we obtain
$$s=\frac{(d_1+d_2)^2 -d_1^2 -d_2^2}{2}= d_1d_2, $$ which completes the proof. 
\end{proof}

\subsection{Stable intersection}

In the last section we considered only tropical curves which intersect ``nicely", 
meaning they intersect only in a finite number of points and away from the vertices of
the two curves. 
But what can we say in the two cases shown in the Figures
 \ref{inter st}a (two tropical lines that intersect in an edge) and  \ref{inter st}b (a tropical line passing 
 through a vertex of a conic)? 
 Thankfully, we have more than one tropical trick up our sleeve.

\begin{figure}[h]
\begin{center}
\begin{tabular}{ccccccc}
\includegraphics[width=3cm, angle=0]{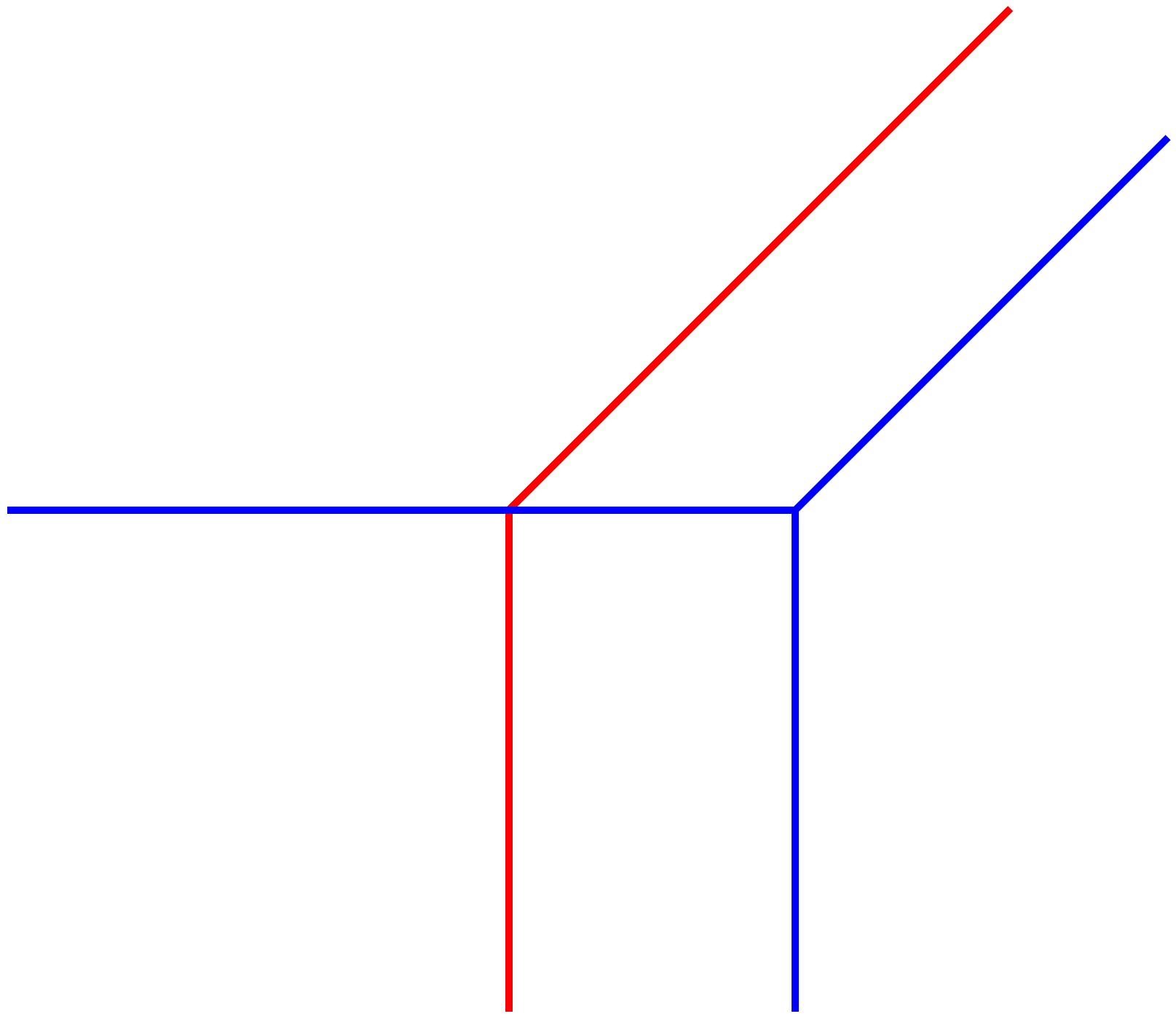}&\hspace{3ex} &
\includegraphics[width=3cm, angle=0]{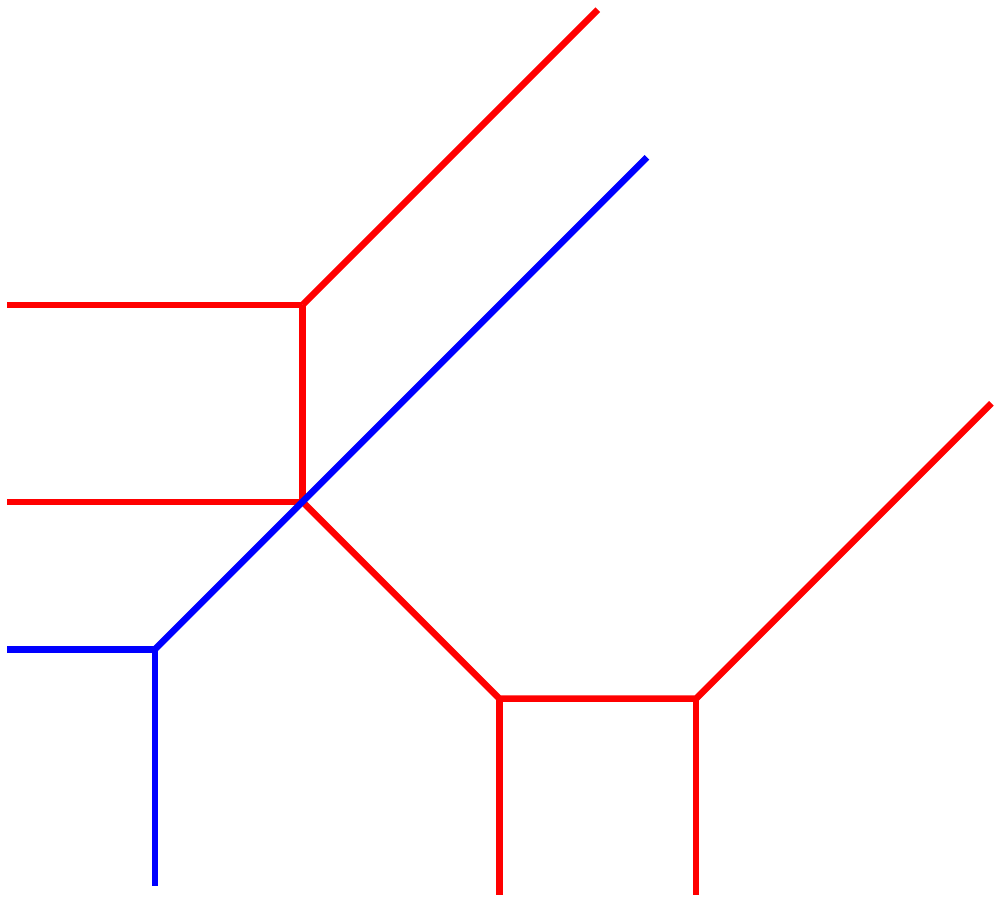}&\hspace{3ex} &
\includegraphics[width=3cm, angle=0]{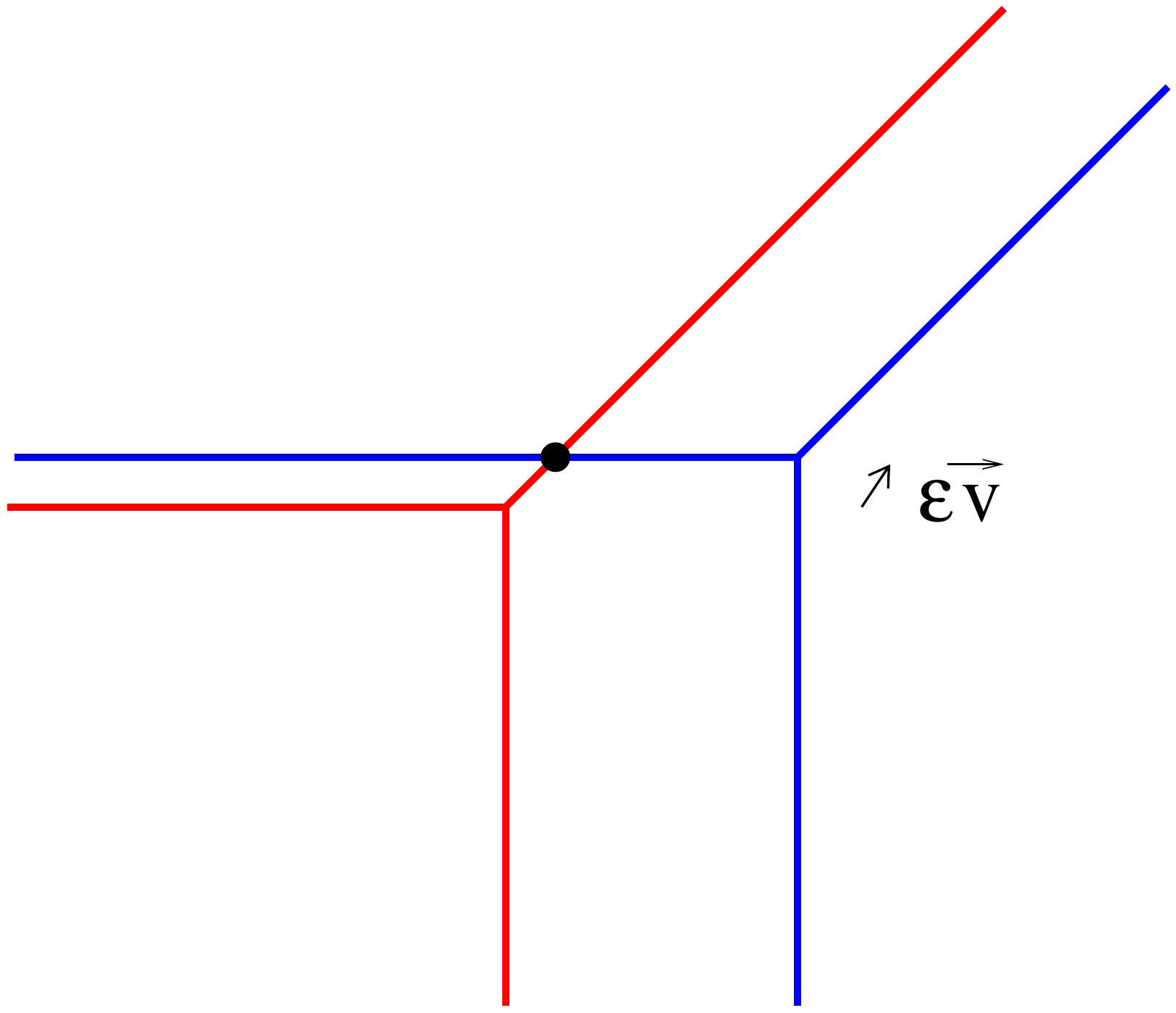}&\hspace{3ex} &
\includegraphics[width=3cm, angle=0]{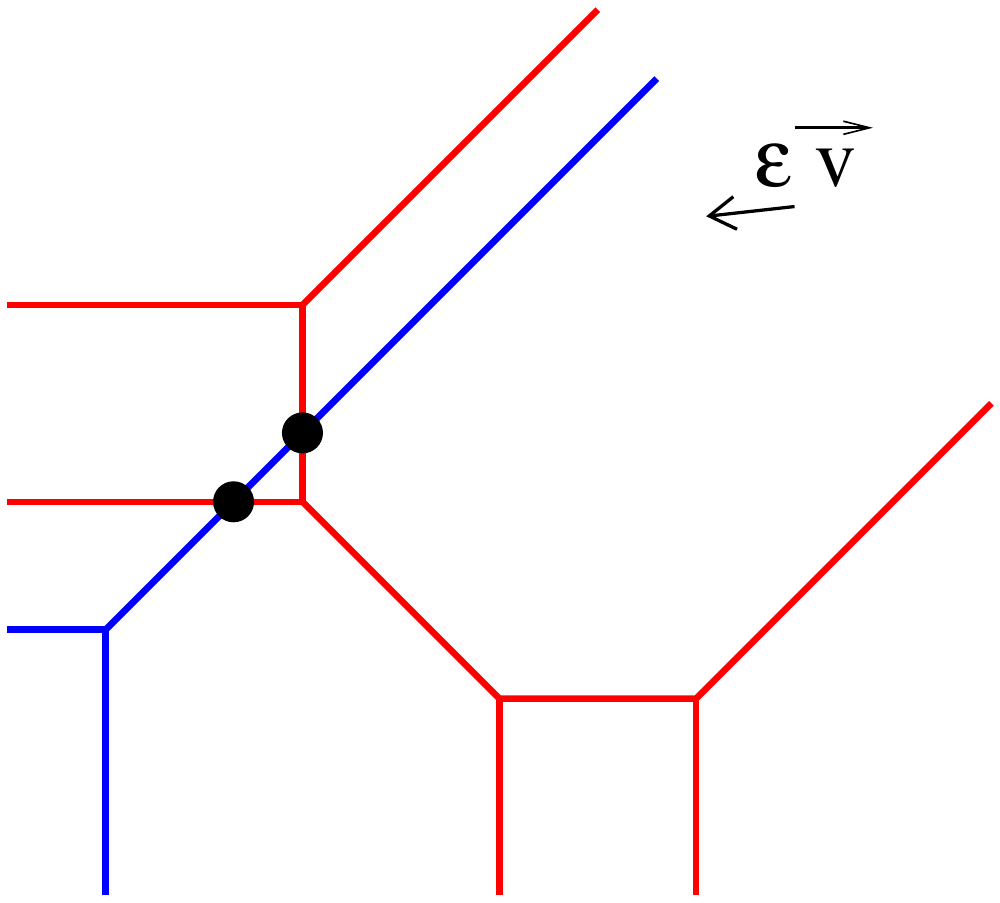} 
\\ \\a) && b)  &&c) &&d)

\end{tabular}
\end{center}
\caption{Non-transverse intersection and a translation}
\label{inter st}
\end{figure}

Let $\varepsilon$  be a very small positive real number and  $\vec v$  a
vector such that the quotient of its two coordinates is an irrational number. 
If we translate in each of the two cases one of the two curves in Figure \ref{inter st}a and b 
by the vector 
 $\varepsilon\vec v$, we find ourselves back in the case of ``nice" intersection 
 (see Figures  \ref{inter st}c and d). Of course,  the resulting intersection
 depends on the vector  $\varepsilon\vec v$. 
But on the other hand, the limit of these points, if we let  
$\varepsilon$ shrink to  0, does not depend on $\vec v$. The points in the limit are called the
\textit{stable intersection points} of two curves. 
The multiplicity of a point $p$  in the stable intersection is equal to the sum of the intersection
 multiplicities of all points which converge to $p$ when $\varepsilon$ tends to zero.

For example, there is only one stable intersection point of the two lines in
Figure
\ref{inter st}a. The point   is the vertex of the line on the left. Moreover, this point has multiplicity 1, so
again, our two tropical lines intersect in a single point. 
The point of stable intersection of the two curves in Figure
\ref{inter st}b is the vertex of the conic, and it has multiplicity 2.  

Notice that if  a point is in the stable intersection 
of two tropical curves 
it is either an isolated intersection point, 
or a  vertex of one of the two curves. 
Thanks to stable intersection, we can remove from our previous statement 
of the tropical B\'ezout theorem the hypothesis that the curves must intersect
``nicely".

\begin{thm}[B. Sturmfels]
Let $C_1$ and $C_2$ be two tropical curves of degree $d_1$ and
$d_2$. Then the sum of the multiplicities of the stable intersection points of 
 $C_1$ and $C_2$ is equal to $d_1d_2$.
\end{thm}

In passing, we may notice a surprising tropical 
phenomenon:  a tropical curve has a well defined 
\textit{self-intersection}\footnote{In classical algebraic geometry, only the
total number of points in the self-intersection of a planar curve is 
well-defined, not the positions of the points on the curve. We can still 
say that a line intersects itself in one point, but it is not at all clear which 
point\dots}! 
Indeed, 
all we have to do is to consider 
the stable intersection of a tropical curve with itself. Following the discussion
above, the self-intersection 
points
of a curve  
are
the 
vertices of the curve. 
(see Figure \ref{auto
  inter}).

\begin{figure}[h]
\begin{center}
\begin{tabular}{c}
\includegraphics[width=10cm, angle=0]{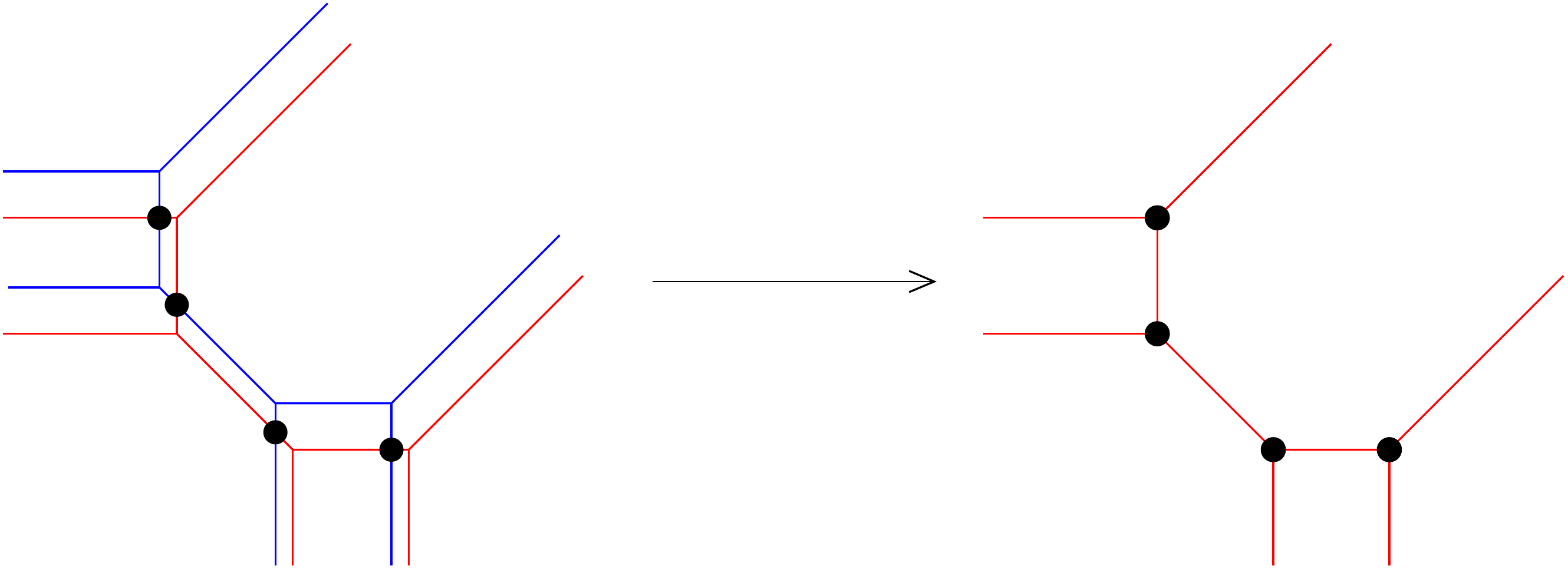}

\end{tabular}
\end{center}
\caption{The 4 points in the self-intersection of a conic}
\label{auto inter}
\end{figure}

\subsection{Exercises}
\begin{exo}
\begin{enumerate}
\item Determine the stable intersection points of the two tropical 
curves in Exercise 1 of Section 2, as well as their multiplicity. 
\item A double point of a tropical curve is a point where two edges 
intersect (i.e.~the 
dual polygon to the vertex of the curve is a parallelogram). Show that a tropical conic with a double point 
is the union of two tropical lines. Hint: consider a line passing through the 
double point of the conic and any other vertex of the conic. 
\item Show that a tropical curve of degree 3 with two double points is the union of 
a line and a tropical conic. Show that a tropical curve of degree 3 with 3 double 
points
 is the
union of 3 tropical lines. 

\end{enumerate}
\end{exo}

\section{A few explanations}\label{expl}

Let us pause for a while with our introduction of tropical geometry to explain 
briefly some connections between classical and tropical geometry. 
In particular, our goal is to illustrate the fact that tropical geometry is a limit of classical 
geometry. 
If we were to summarise roughly the content of this
section in one sentence it would be: tropical geometry is the
image of classical geometry under the logarithm with base $+\infty$.

\subsection{Maslov dequantisation.}

First of all let us explain how the tropical semi-field arises naturally as the 
limit of  some classical semi-fields.  This procedure, studied by Victor Maslov 
and his collaborators beginning in the 90's, is known as \textit{dequantisation of 
the real numbers}. 

A well-known semi-field  is the set of positive or zero real numbers 
together with the usual addition and multiplication,  denoted $(\R_+,+,\times)$.  
If $t$ is a strictly positive real number, then the logarithm of base $t$ 
provides a bijection between the sets $\R$ and $\T$. This bijection
induces a semi-field structure on $\T$ with the operations denoted by 
$\tg +_t\td$ and $\tg\times_t\td$, and given by:
$$\tg x +_t  y\td= \log_t(t^x +t^y) \  \ \ \text{ and } \ \ \  \tg x \times_t  y\td=
\log_t(t^x t^y) = x+y.$$

The equation on the right-hand side already shows classical addition appearing as an exotic kind of multiplication on $\T$. 
Notice that by construction, all of the semi-fields 
$(\T,\tg +_t\td,\tg
\times_t\td)$ are isomorphic to $(\R_+,+,\times)$. 
The trivial inequality 
$\max(x,y)\le x+y\le 2\max(x,y)$ on $\R_+$ together with the fact that the logarithm is
an increasing function gives us the following bounds for $\tg +_t\td$:
$$\forall t>1,  \ \  \max(x,y) \le \tg x +_t  y\td \le \max(x,y)
+\log_t 2.$$
If we let $t$ tend to infinity, then  $\log_t 2$ tends to $0$, 
and the operation
$\tg +_t\td$ therefore tends to 
the tropical addition $\tg +\td$! 
Hence
 the tropical semi-field comes naturally from degenerating the classical semi-field 
$(\R_+,+,\times)$.
From an alternative perspective, we can view the classical semi-field 
$(\R_+,+,\times)$ as a \emph{deformation} of the tropical semi-field. This explains the use of the 
term ``dequantisation",
coming from physics and referring
to the procedure of passing 
from quantum to classical mechanics.

\subsection{Dequantisation of a line in the plane.}\label{deqdte}

Now we will apply a similar reasoning to the line in the plane $\R^2$ defined by the equation
$x-y+1$ 
 (see  Figure \ref{amibe}a). 
To apply the  logarithm map  to the coordinates of $\R^2$, we must first take their absolute values.
Doing this results in  folding the 4 quadrants of $\R^2$ onto the positive quadrant (see  Figure
\ref{amibe}b). 
The image of the folded up line under the coordinate-wise logarithm with base $t$ 
 applied to $(\RR_+^*)^2$ is drawn in Figure
  \ref{amibe}c. By definition, taking the logarithm with base $t$ is the same thing as 
  taking the natural logarithm 
  and then rescaling the result  by a factor of $\frac{1}{\ln t}$. 
Thus, as $t$ increases, the image under the logarithm with base $t$ 
of the absolute value 
of our line becomes concentrated around a neighbourhood of the origin and  three 
asymptotic directions, as shown in Figures \ref{amibe}c, d and e. 
If we allow $t$ go all the way to infinity, then we see the appearance in   Figure \ref{amibe}f of... a tropical line!

\begin{figure}[h]
\begin{center}
\begin{tabular}{cccccc}
\includegraphics[width=2.5cm,  angle=0]{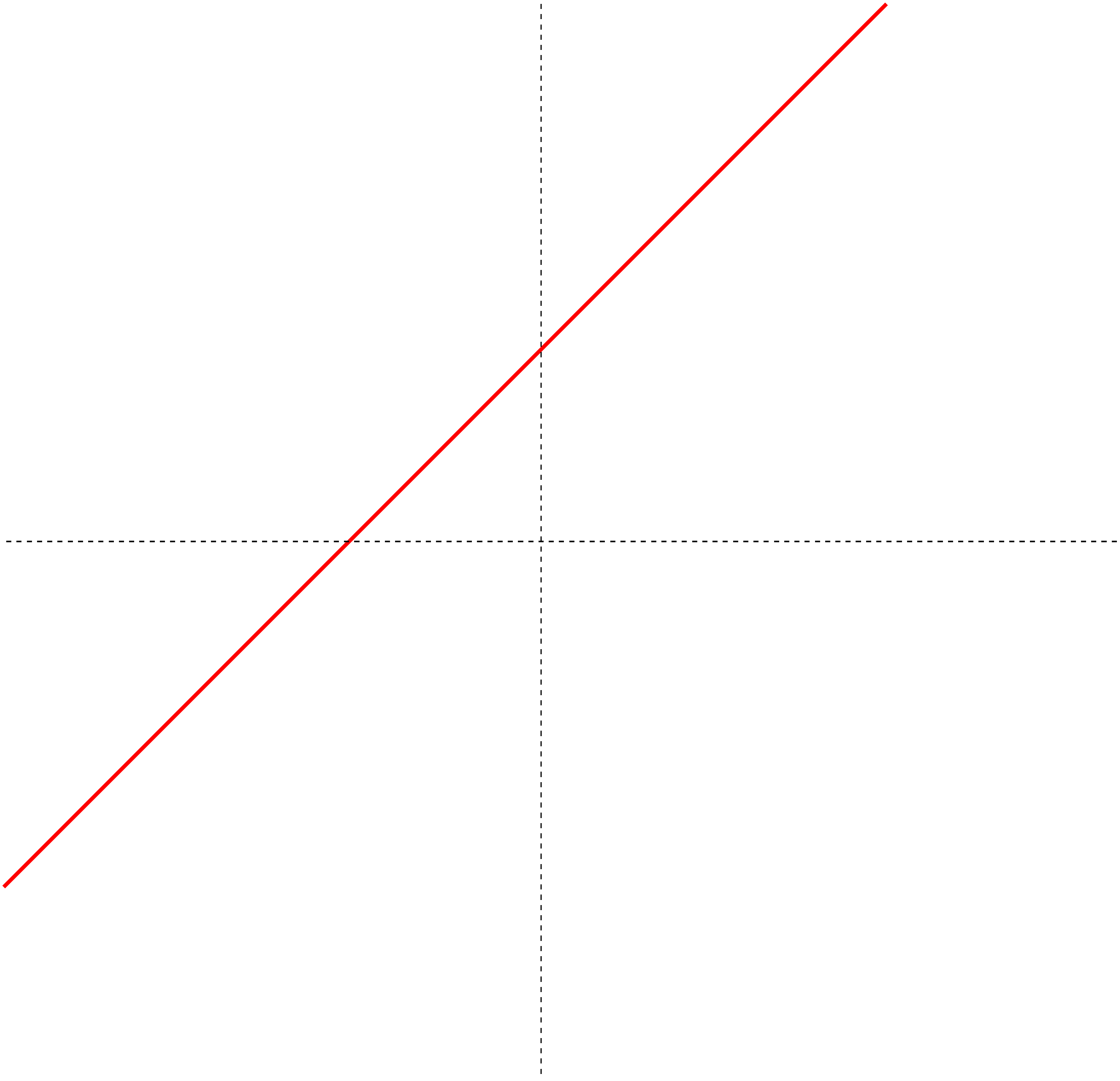}&
\includegraphics[width=2.5cm, angle=0]{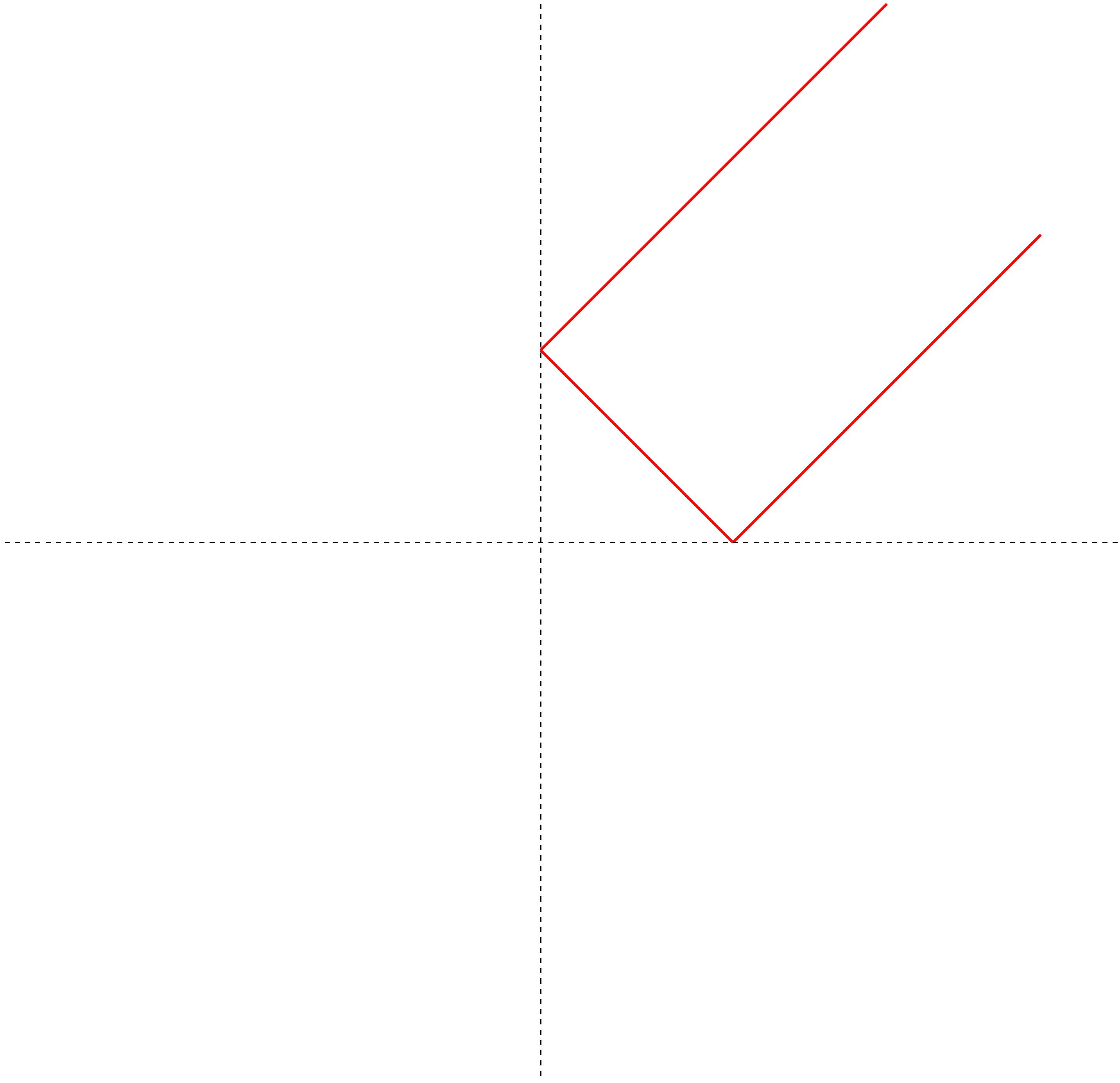}&
\includegraphics[width=2.5cm, angle=0]{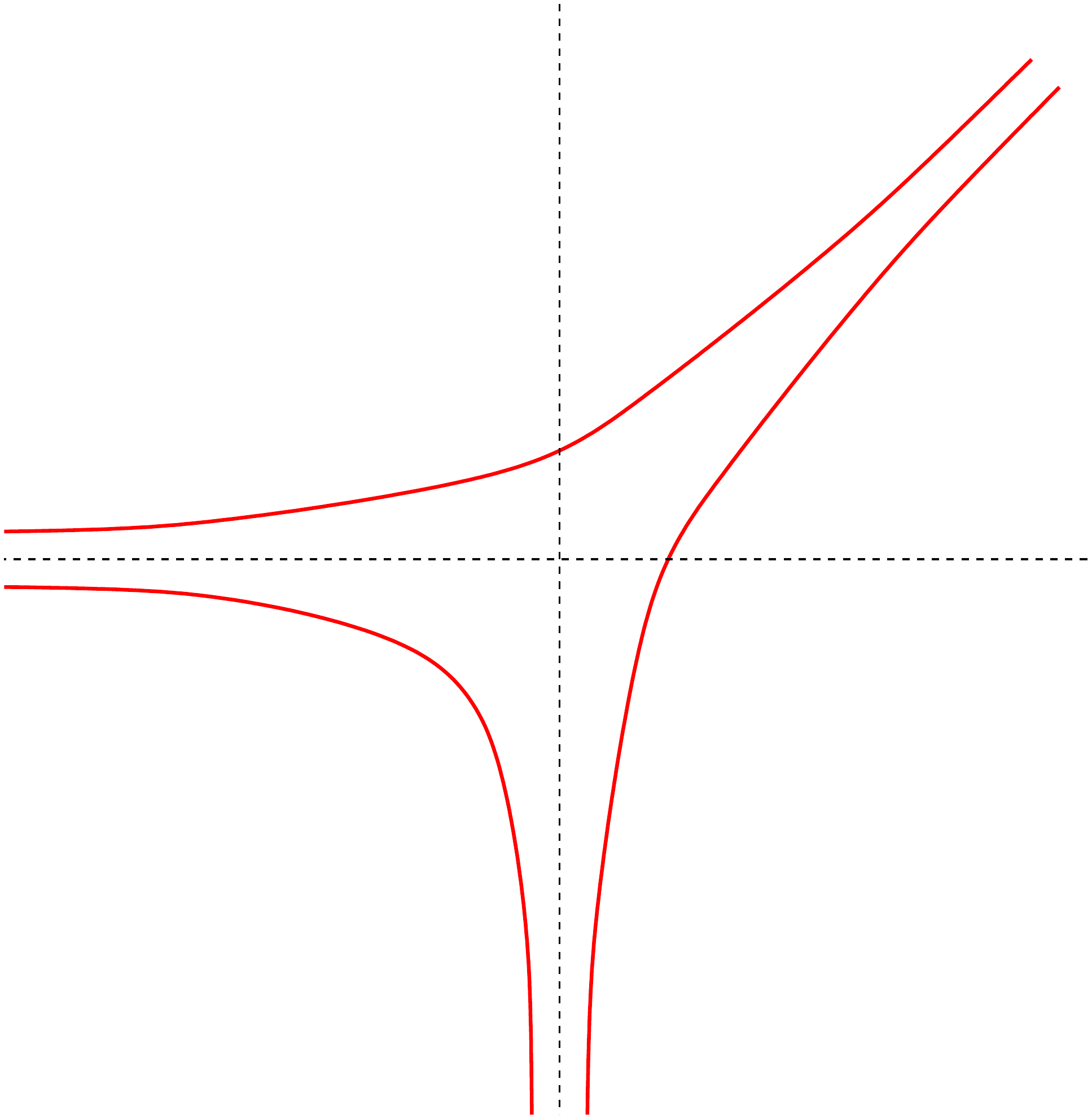}&
\includegraphics[width=2.5cm, angle=0]{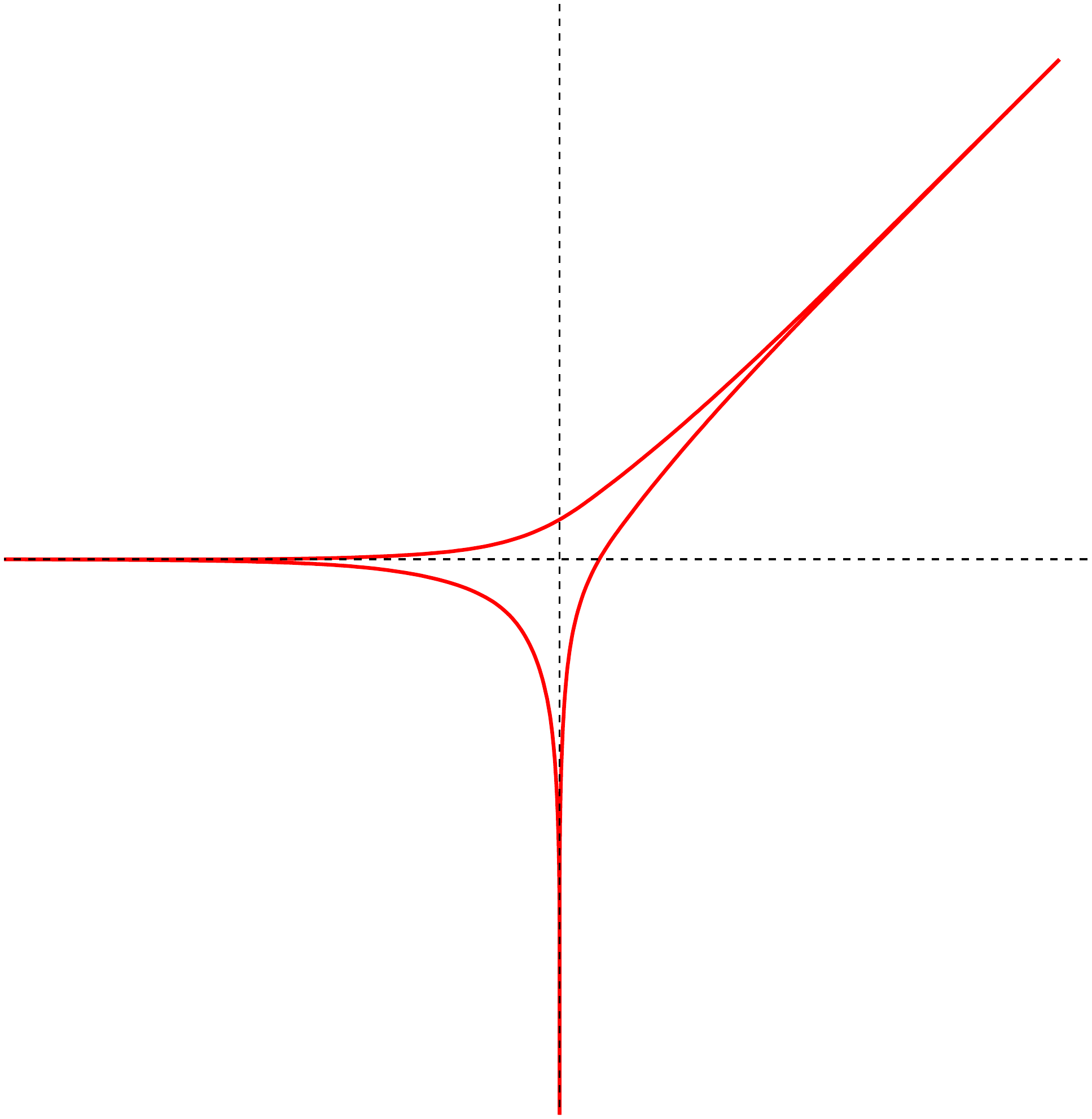}&
\includegraphics[width=2.5cm, angle=0]{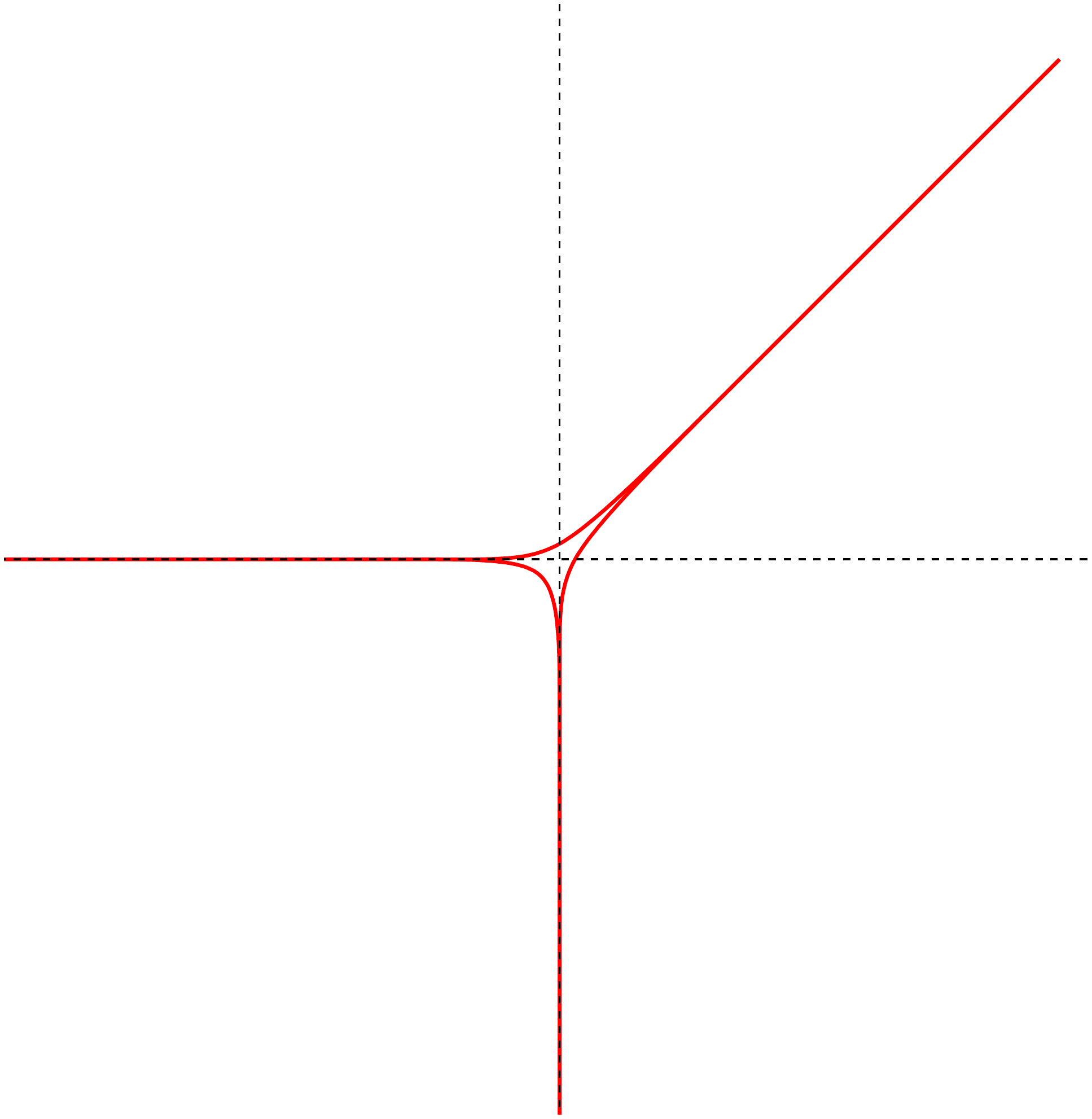}&
\includegraphics[width=2.5cm, angle=0]{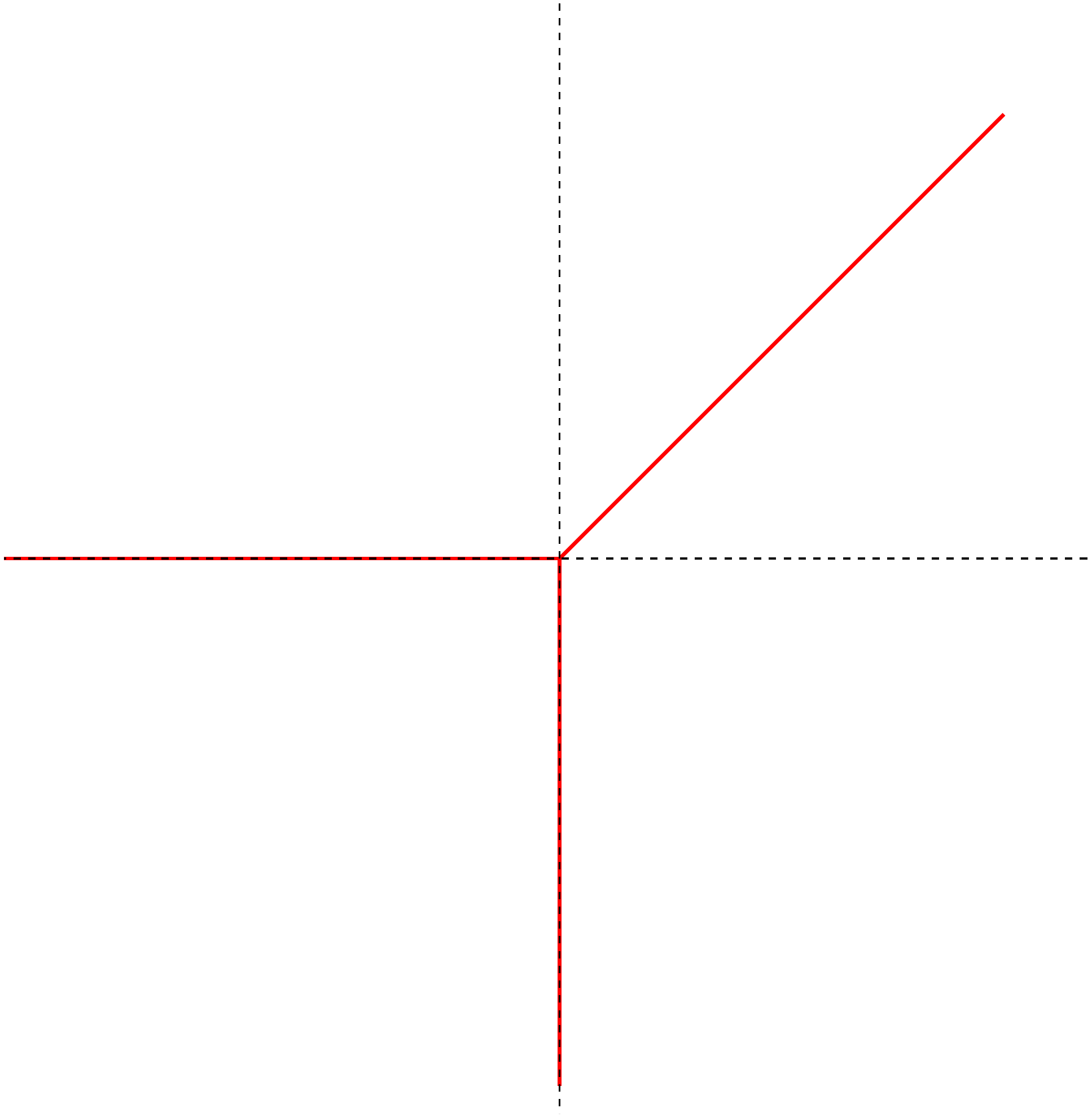}
\\ \\ a) & b) & c) &d)&e)& f)

\end{tabular}
\end{center}
\caption{Dequantisation of a line}
\label{amibe}
\end{figure}

\section{Patchworking}\label{patchwork}

In reading Figure  \ref{amibe} from left to right, we see how starting from a classical line in the plane we can arrive at a tropical line. 
Reading this figure from right to left is in fact much more interesting!
Indeed, we see as well how to construct a classical line given a tropical one. 
The technique known as \textit{patchworking} is a generalisation of this observation. 
In particular, it provides a purely combinatorial  procedure to construct real algebraic curves 
from a tropical curve. In order to explain this procedure in more detail, we first  go a bit back in time.

\subsection{Hilbert's 16th problem.}\label{sec:16}

A \textit{planar real algebraic curve} is a curve in the plane $\R^2$ defined by 
an equation of the form $P(x,y)=0$, where $P(x,y)$ is a polynomial whose
coefficients are real numbers. 
The real algebraic curves of degree 1 and 2 are simple and well-known; they
are lines and conics respectively. 
When the degree of  $P(x,y)$ increases, the form of the real algebraic curve can 
become more and more complicated. If you are not convinced, take a look at figure 
 \ref{quartic} which shows some of the possible drawings  
realised
by real algebraic curves of degree 4.

\begin{figure}[h]
\begin{center}
\begin{tabular}{ccccccc}
\includegraphics[width=3cm,  angle=0]{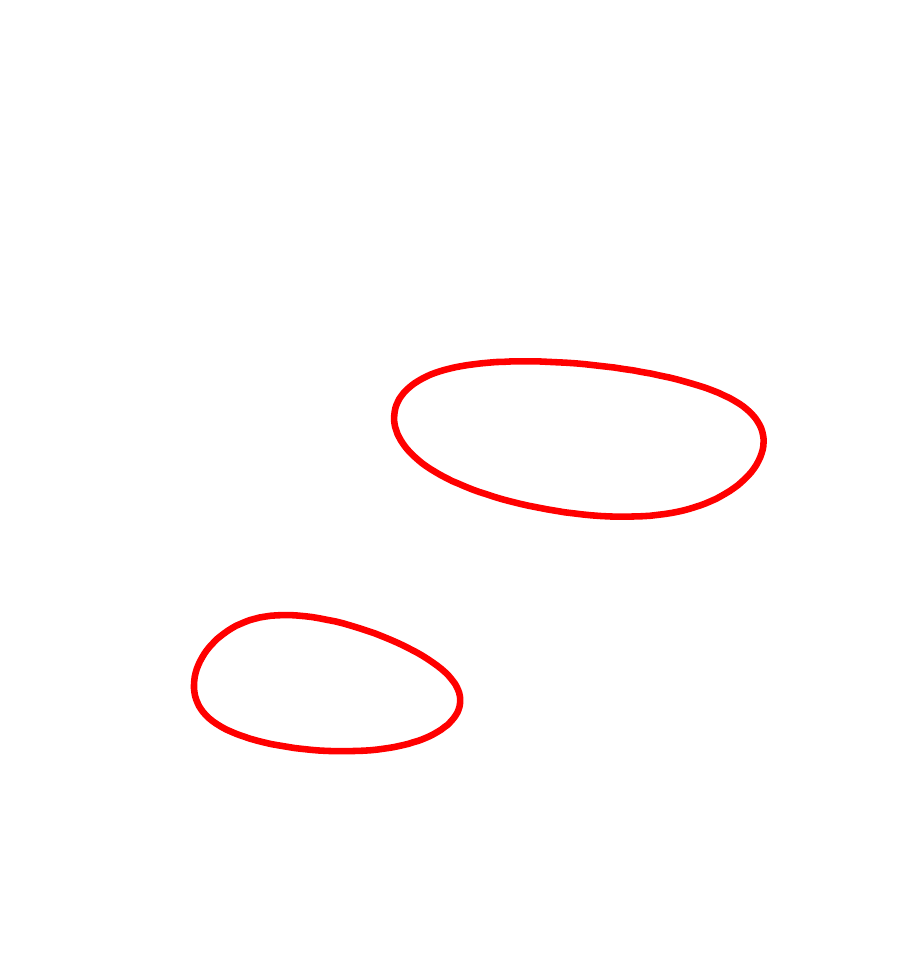}&\hspace{3ex} &
\includegraphics[width=3cm, angle=0]{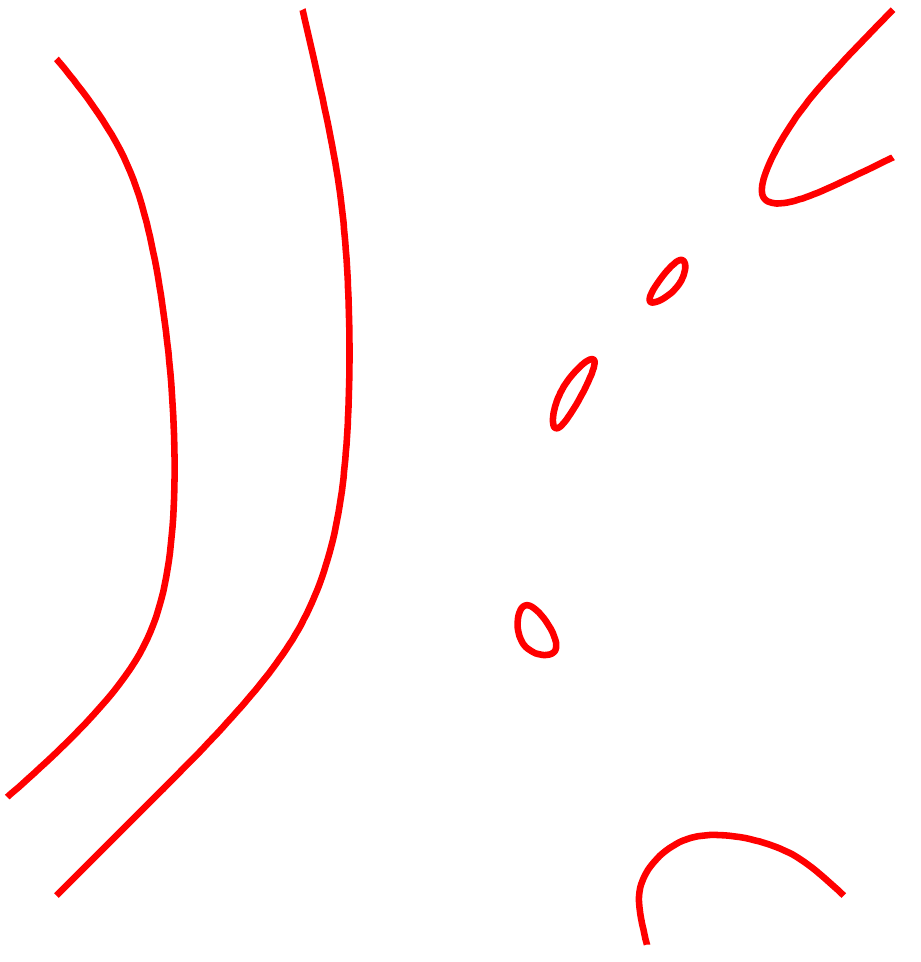}&\hspace{3ex} &
\includegraphics[width=3cm, angle=0]{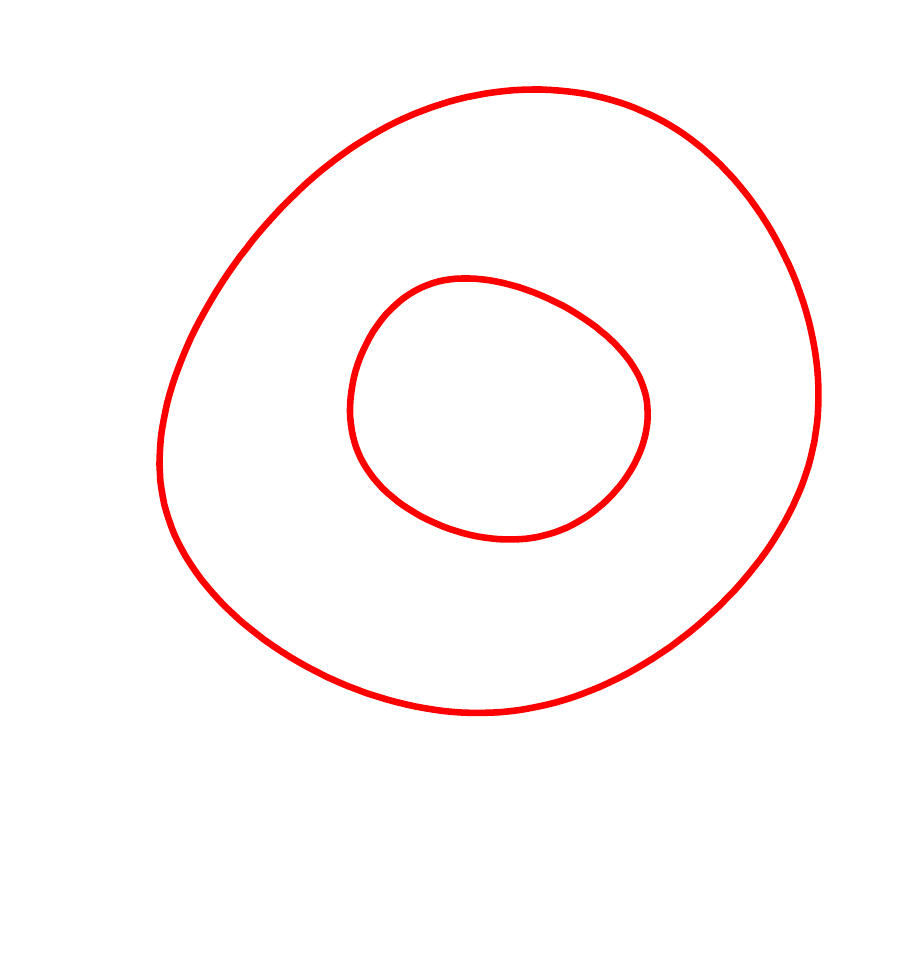}&\hspace{3ex} &
\includegraphics[width=3cm, angle=0]{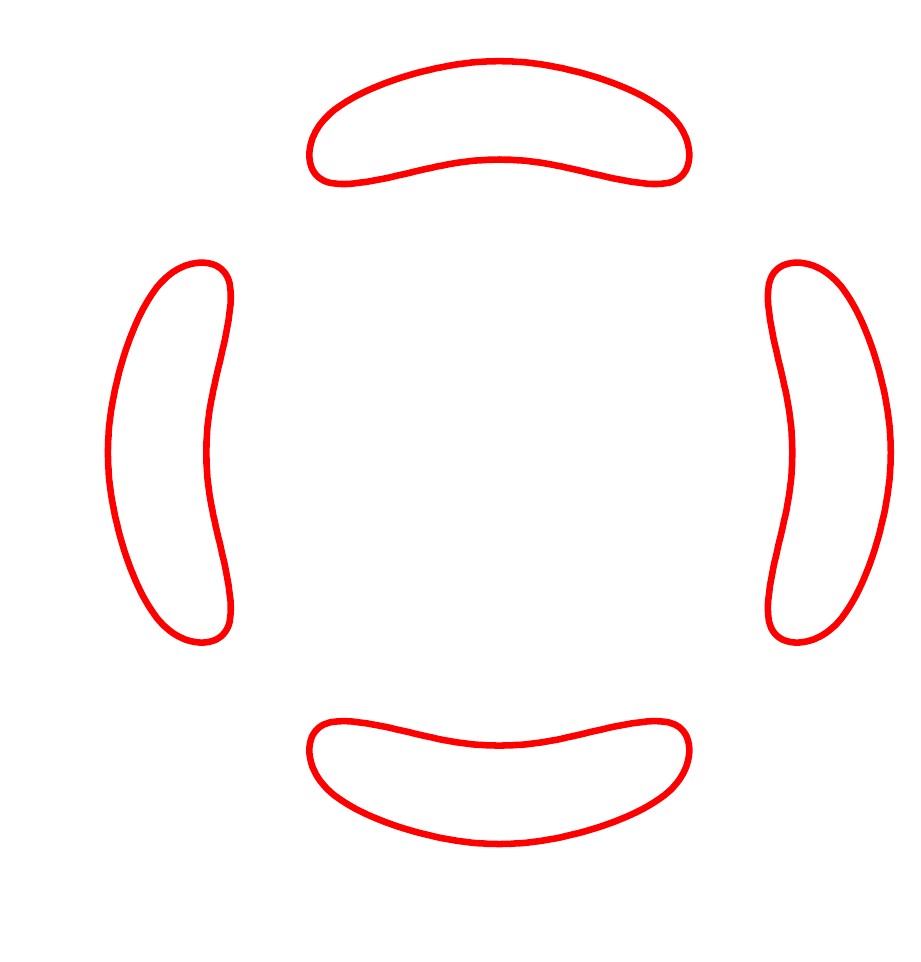}
\\ \\ a) && b) && c) &&d)

\end{tabular}
\end{center}
\caption{Some real algebraic curves of degree 4}
\label{quartic}
\end{figure}

A theorem due to 
 Axel Harnack at the end of the  XIXth century
states that a planar real algebraic curve of degree $d$ has a maximum 
of  $\frac{d(d-1)+2}{2}$ connected components. But how can these 
components be arranged with respect to each other?
We call the relative position of the connected components of a planar real algebraic curve in the plane 
 its \textit{arrangement}. 
In other words, we are not interested in the exact position of the 
curve in the plane, but simply 
in
the configuration that it 
realises. For example, if one curve has two bounded connected components, 
we are only interested in whether one of these components is contained in the other 
(Figure \ref{quartic}c) or not (Figure \ref{quartic}a). 
 At the second International Congress in Mathematics in Paris in 1900, David Hilbert 
 announced his famous list of 23 problems for the  XXth century. The
 first part of his 16th problem can be very widely understood as the following:

\begin{center}
\textit{Given a positive integer $d$, establish a list of possible arrangements 
of real algebraic curves of degree $d$.}
\end{center}

At Hilbert's time, the answer\footnote{A more reasonable
and natural problem is to consider arrangements of connected
components of 
non-singular real algebraic curves in the projective plane rather than  
in $\R^2$. 
In this
more restrictive case, the answer was known up to degree 5 at
Hilbert's time, and is now known up to degree 7. 
 The list in degree 7 is given in a theorem by Oleg Viro and patchworking is an essential tool in its proof.} was known for curves of degree at most
4.
 There have been 
spectacular advances in this problem in the XXth century due mostly in part  to mathematicians
from the Russian school. Despite this,
 there remain numerous open questions...

\subsection{Real and tropical curves}

In general, it is a difficult problem to construct a real algebraic curve
of a fixed degree  and realising a given arrangement. 
For a century, mathematicians have proposed many ingenious methods for
doing this. 
The patchworking method invented by Oleg Viro in the 70's is actually one of 
the most powerful. 
At this time, tropical geometry was not yet in existence, and Viro announced
his theorem in a language  different from what we use here. 
However, he realised by the end
of the 90's that his patchworking could be interpreted as 
a \textit{quantisation} of tropical curves. Patchworking is in fact  
the process of 
reading  Figure \ref{amibe} from right to left instead of left to right. 
Thanks to the  interpretation of patchworking in terms of tropical 
curves, shortly afterwards, Grigory Mikhalkin 
generalised Viro's original method. Here we will present  a simplified 
version of patchworking. The interested reader may find a more 
complete version in the references indicated at the end of the text in Section \ref{ref}.

In what follows, for $a$, $b$ two integer numbers, we denote by
$s_{a,b}:\RR^2\to\RR^2$ the composition of $a$ reflections in 
the $x$-axis with $b$ reflections in 
the $y$-axis. Therefore, the map 
$s_{a, b}$ depends only on  the parity of $a$ and $b$. To be precise, 
$s_{0, 0}$ is the identity, $s_{1, 0}$ is the reflection in the $x$-axis, 
$s_{0, 1}$ reflection in the $y$-axis, and $s_{1, 1}$ is reflection 
in the origin (equivalently,  rotation by 180 degrees).

We will now explain in detail the procedure of patchworking. We 
start with a tropical curve $C$ of degree $d$ which has only 
edges of odd weight,
 and such that the polygons dual to all vertices
are triangles. 
For example, take the tropical line from Figure 
 \ref{patch dte}a. For each edge $e$ of $C$, choose a vector
  $\vec v_e=(\alpha_e,\beta_e)$ in the direction of $e$ such that $\alpha_e$
and $\beta_e$ are
 relatively prime integers (note that we may
  choose either
   $\vec v_e$ or $-\vec v_e$, but this does not matter
  in what follows).
For the tropical line, we may take 
the vectors  $(1,0)$, 
$(0,1)$ and $(1,1)$. 
Now let us think
of  the plane $\R^2$ where 
our tropical curve 
sits as being 
the positive quadrant $(\RR^*_+)^2$ of $\RR^2$, 
and take the union of the curve with its 3 symmetric copies obtained by 
 reflections in the axes.  For the tropical line, we obtain  Figure \ref{patch 
  dte}b.
  For each edge $e$ of our curve, we will erase two 
  out of the four symmetric copies of $e$, denoted  $e^{\prime}$ and $e^{\prime \prime}$. 
Our choice must satisfy  the following rules:

\begin{itemize}
 \item $e'=s_{\alpha_e,\beta_e}(e'')$,
\item for each vertex $v$ of $C$ adjacent 
to the edges $e_1$, $e_2$
  and $e_3$ and for each pair
  $(\varepsilon_1,\varepsilon_2)$ in  $\{0,1\}^2$, 
exactly one or three of
the copies 
of
  $s_{\varepsilon_1,\varepsilon_2}(e_1)$,
  $s_{\varepsilon_1,\varepsilon_2}(e_2)$, and
  $s_{\varepsilon_1,\varepsilon_2}(e_3)$ are erased. 
\end{itemize}

\vspace{1ex}
We call the result after erasing the edges a
\textit{real tropical curve}. For example, if $C$ is a tropical line, 
using the rules above, it is possible to erase 6 of the edges 
from the symmetric copies of $C$ to obtain the real tropical line 
represented in Figure \ref{patch
  dte}c.
  It is true that this real tropical curve
   is not a regular line in $\R^2$, but 
 they are arranged in the plane in  same way 
  (see Figure \ref{patch dte}d)!

This is not just a coincidence, it is in fact a theorem.

\begin{figure}[h]
\begin{center}
\begin{tabular}{ccccccc}
\includegraphics[width=3cm,  angle=0]{Figures/Droite.pdf}&\hspace{3ex} &
\includegraphics[width=3cm, angle=0]{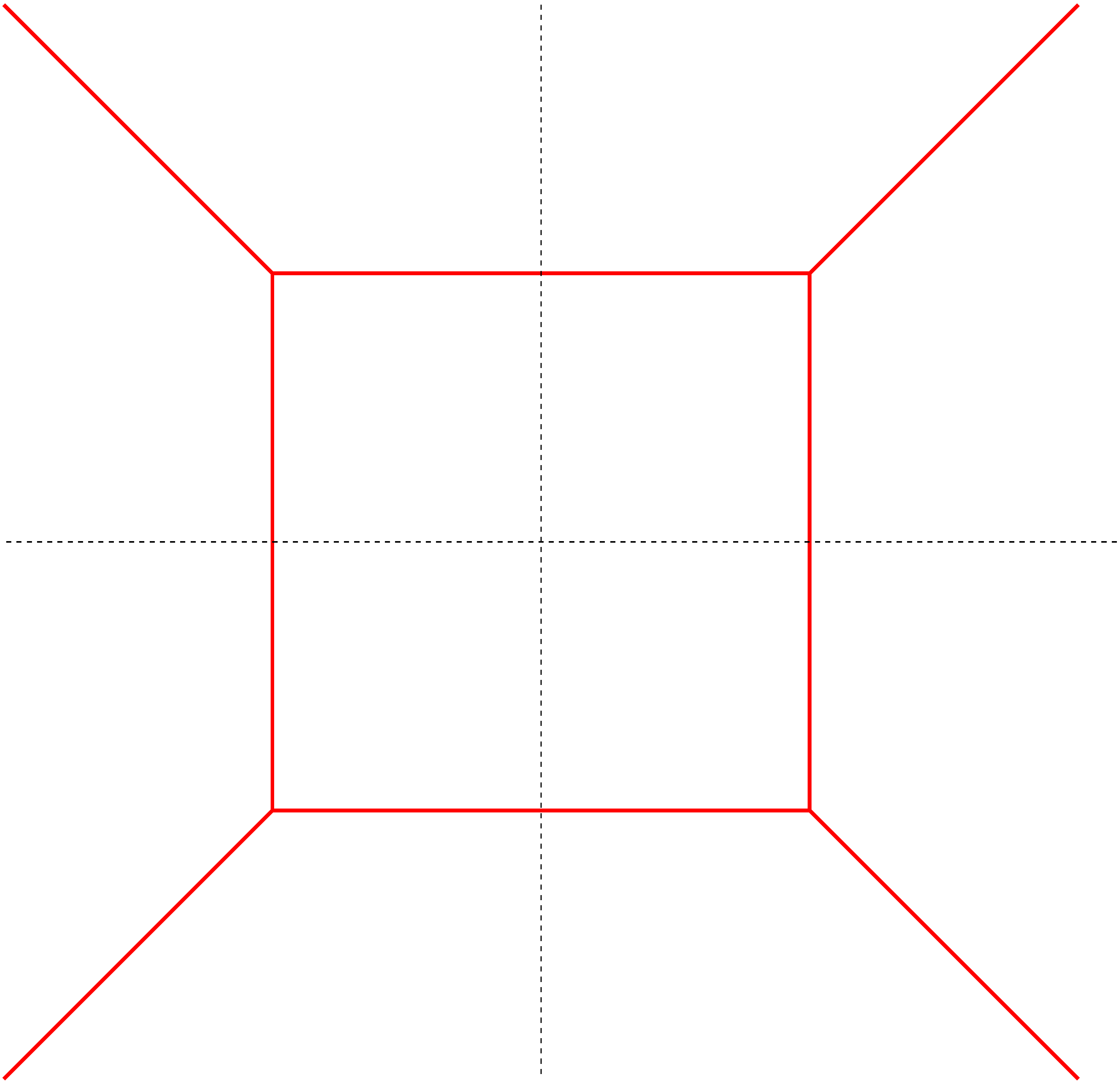}&\hspace{3ex} &
\includegraphics[width=3cm, angle=0]{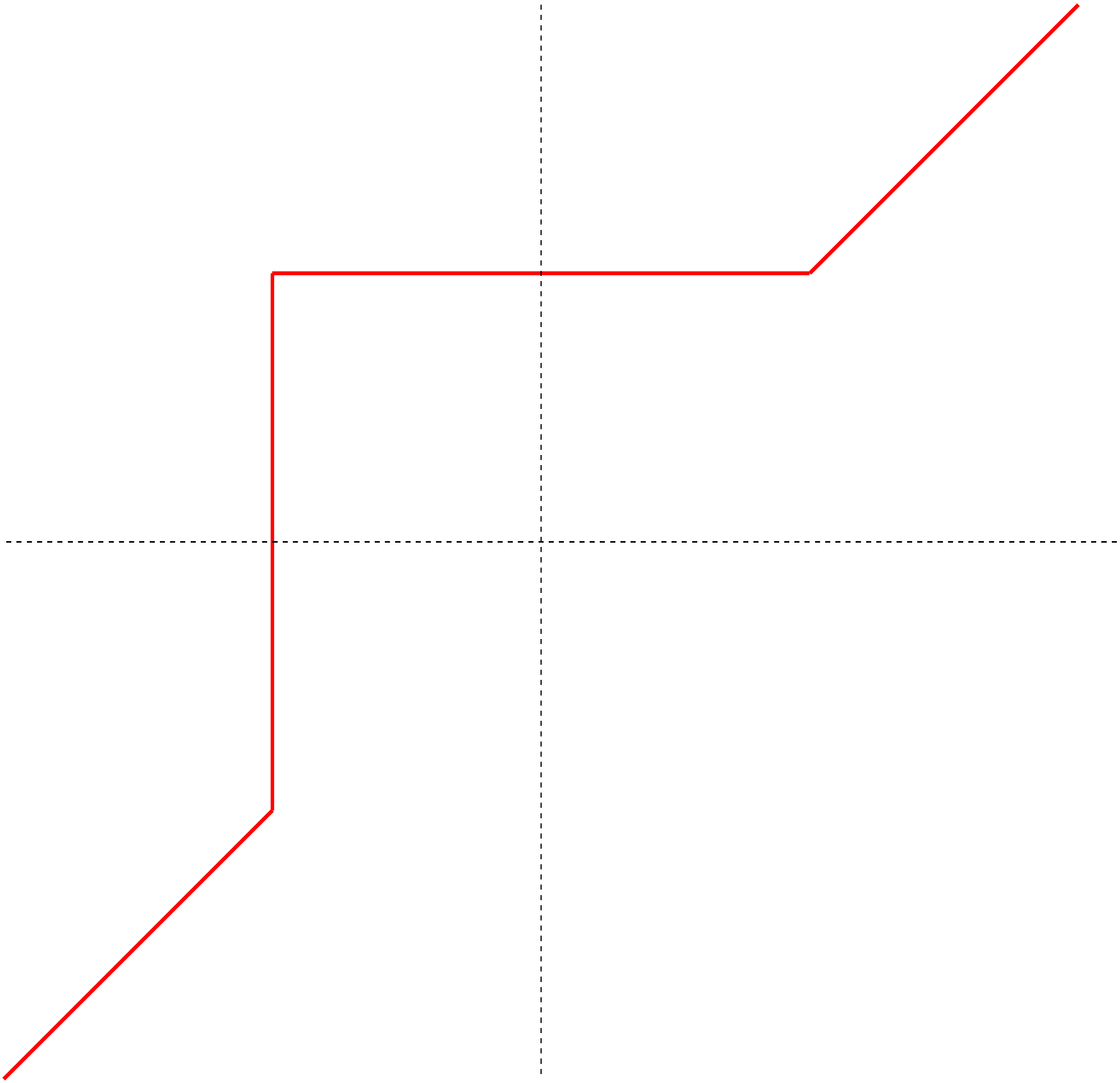}&\hspace{3ex} &
\includegraphics[width=3cm, angle=0]{Figures/PatchDte4.pdf}
\\ \\ a) && b) && c) &&d)

\end{tabular}
\end{center}
\caption{Patchworking of a line}
\label{patch dte}
\end{figure}

\begin{thm}[O. Viro]\label{viro}
Given 
any real tropical curve of degree $d$, there exists
a real algebraic  
curve of degree $d$ with the same arrangement. 
\end{thm}

We should take a second to realise the depth and elegance of the above statement. 
A real tropical curve is constructed by following the rules of a  purely combinatorial 
game. It seems like magic to assert that there is a relationship between these 
combinatorial objects and actual real algebraic curves!
We will not get into details here, but Viro's method even allows us to determine
the 
equation of 
a
real algebraic curve 
with the same arrangement than
a given
real tropical curve.

Of course, much skill and intuition are still required to construct real algebraic curves with complicated arrangements and  both of these are acquired with 
practice. Therefore, we invite the reader to try the exercises at the end of this section. 
 Nevertheless, patchworking
 provides a much more flexible technique to construct arrangements than dealing directly with polynomials, 
in particular it requires a priori no knowledge in algebraic geometry.
To illustrate 
the use of
Viro's patchworking,
let us now
construct the arrangements of 
two real algebraic curves, one of degree 3 and the other of degree 6.

First of all, consider the tropical curve of degree 3 shown in Figure
 \ref{Cub}a. 
For  suitable choices of edges to erase, the Figures 
 \ref{Cub}b and c represent the 
two
stages of the patchworking procedure.  
With this we have proved the existence of a real algebraic curve of degree 3 
resembling the configuration of 
 Figure  \ref{Cub}d.

\begin{figure}[h]
\begin{center}
\begin{tabular}{ccccccc}
\includegraphics[width=3cm, angle=0]{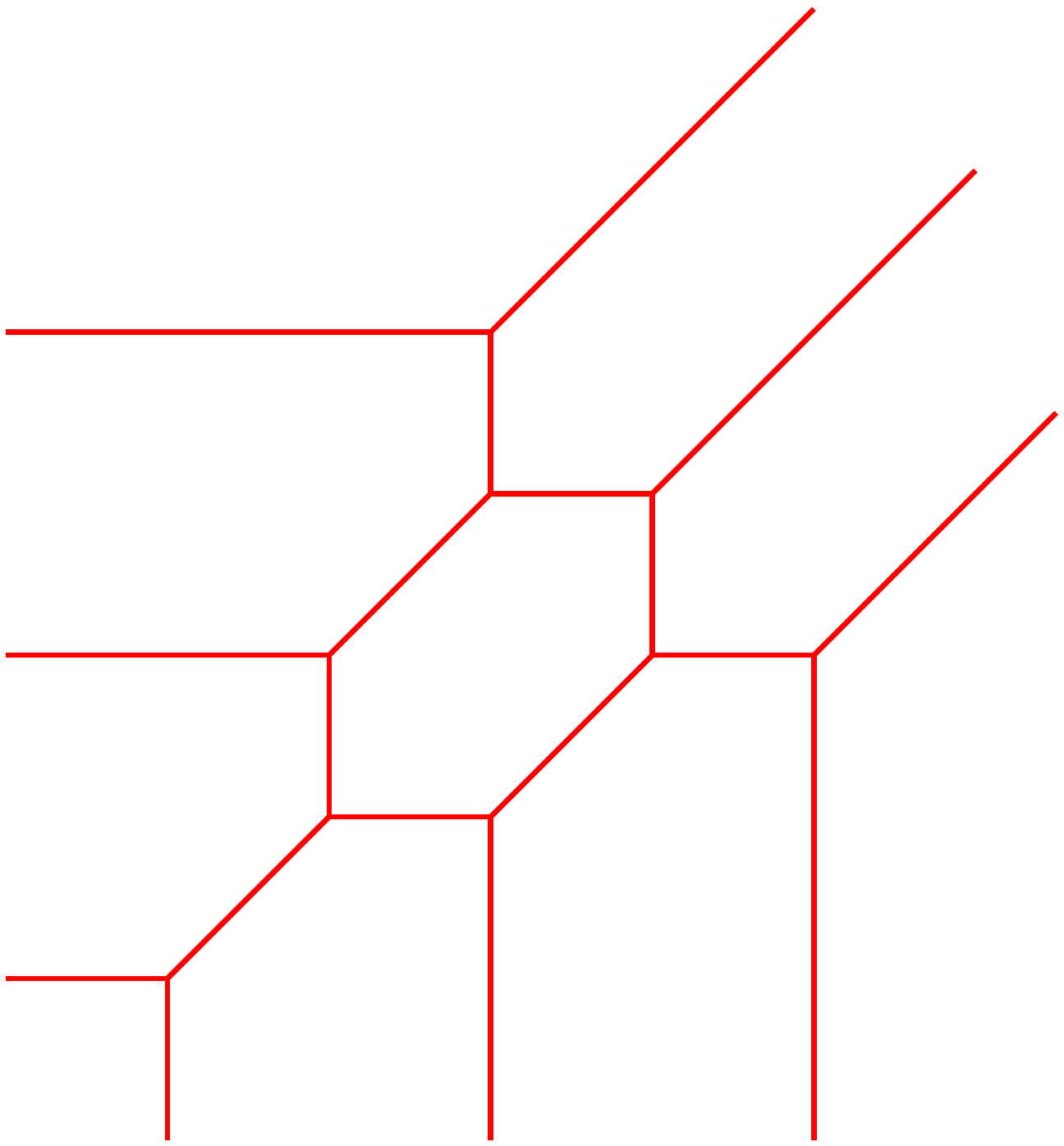}&\hspace{3ex} &
\includegraphics[width=3cm, angle=0]{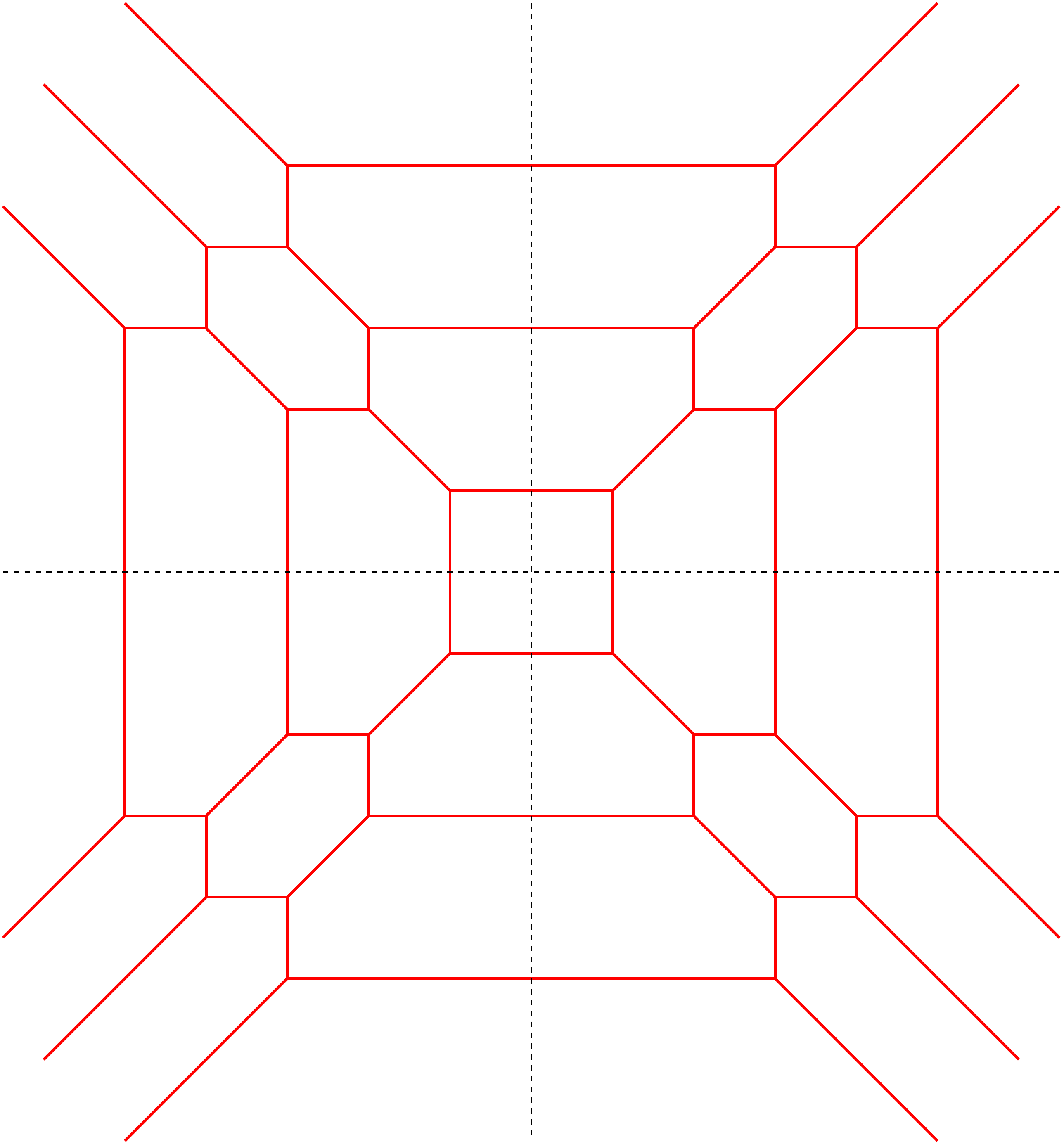}&\hspace{3ex} &
\includegraphics[width=3cm, angle=0]{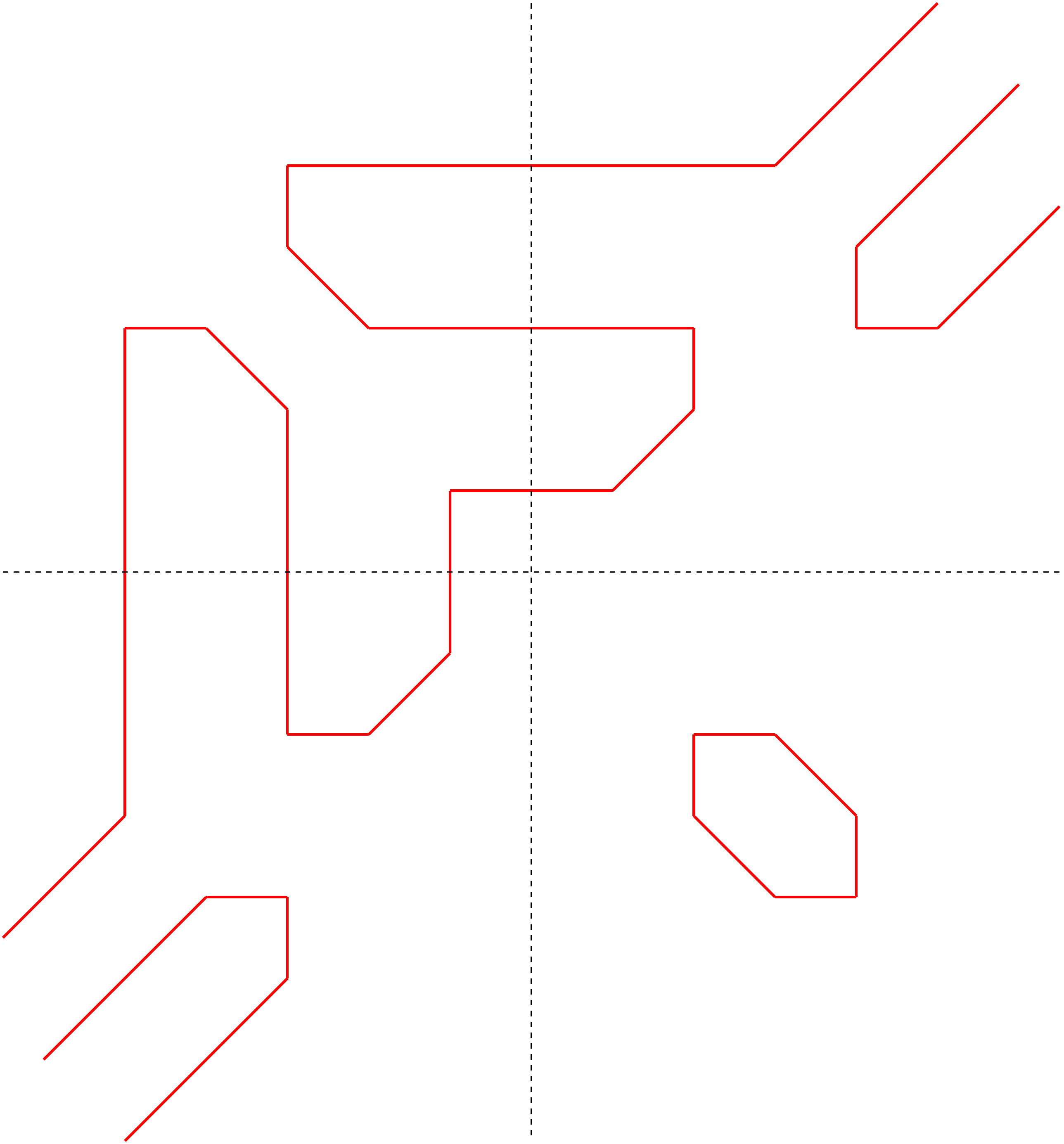}&\hspace{3ex} &
\includegraphics[width=3cm, angle=0]{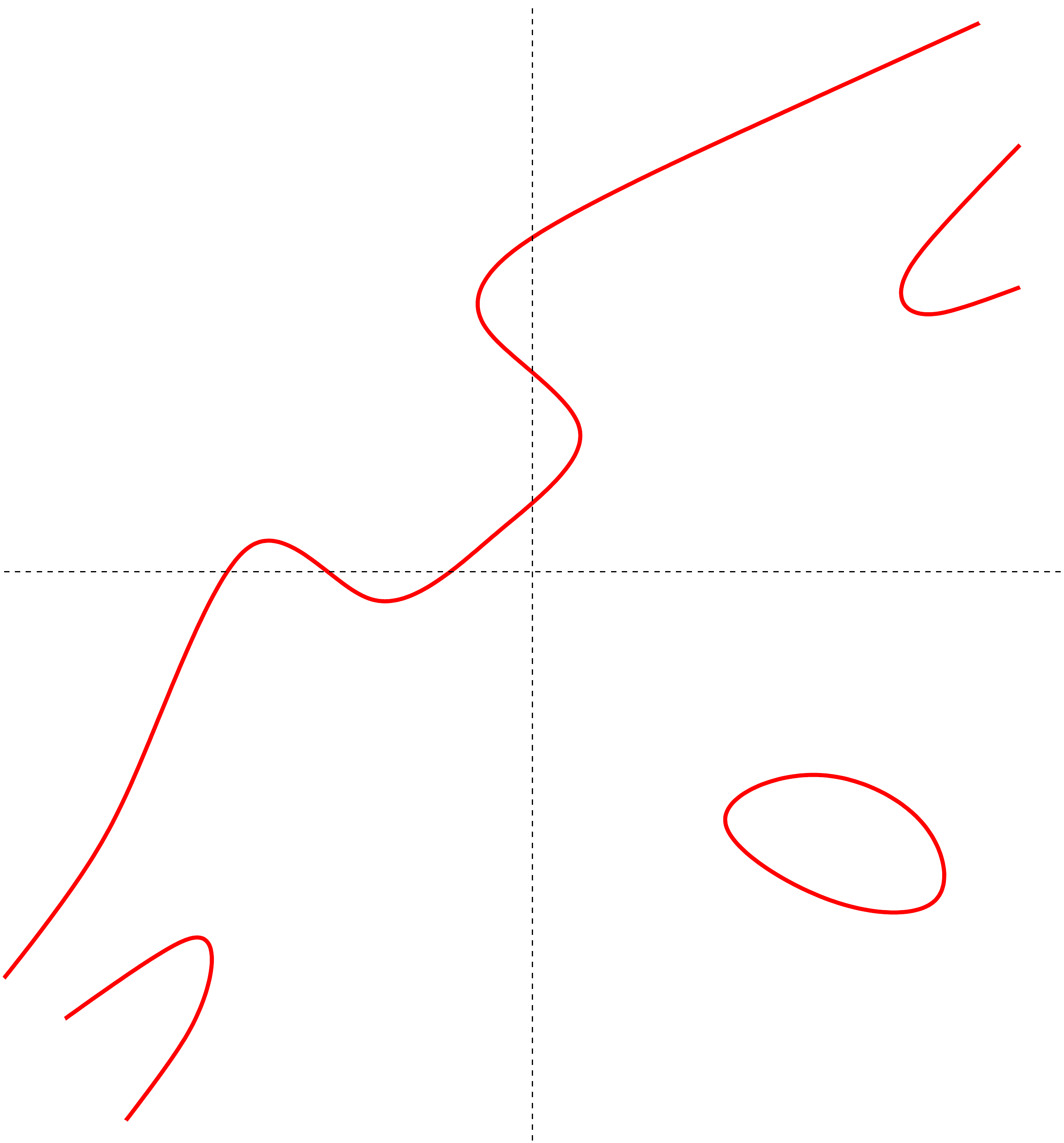}
\\ \\ a) && b) && c) &&d)
\end{tabular}
\end{center}
\caption{Patchworking of a cubic}
\label{Cub}
\end{figure}

To finish, consider the tropical curve of degree 6 drawn in Figure 
 \ref{Gud}a.
Once again, for a suitable choice of edges to erase, the patchworking 
procedure gives the curve in Figure \ref{Gud}c.
A real algebraic curve of degree 6 which realises the same arrangement
as the real tropical curve was first constructed using much more complicated 
techniques in the 60's by Gudkov.
An interesting piece of trivia: Hilbert announced in 1900 that such a curve could not exist!

\subsection{Amoebas}\label{sec:amoeba}

Even though dequantisation of a line is the main idea underlying patchworking 
in its full generality, the proof of Viro's theorem is quite a bit more technical to explain
rigorously. Here we will give a sketch of what is involved in the proof.  

\begin{figure}[h]
\begin{center}
\begin{tabular}{ccc}
\includegraphics[width=5cm, angle=0]{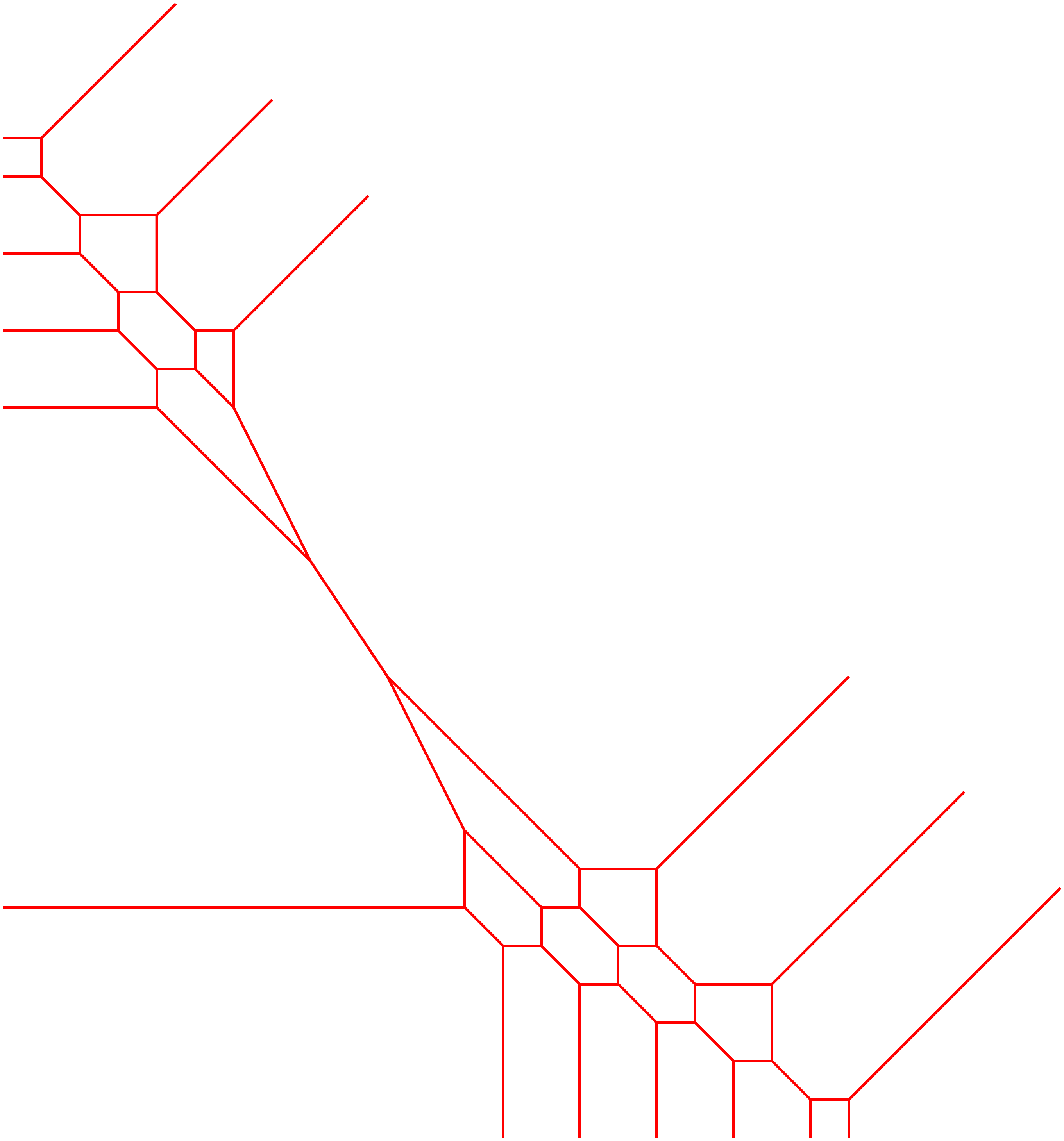}&
\includegraphics[width=5cm, angle=0]{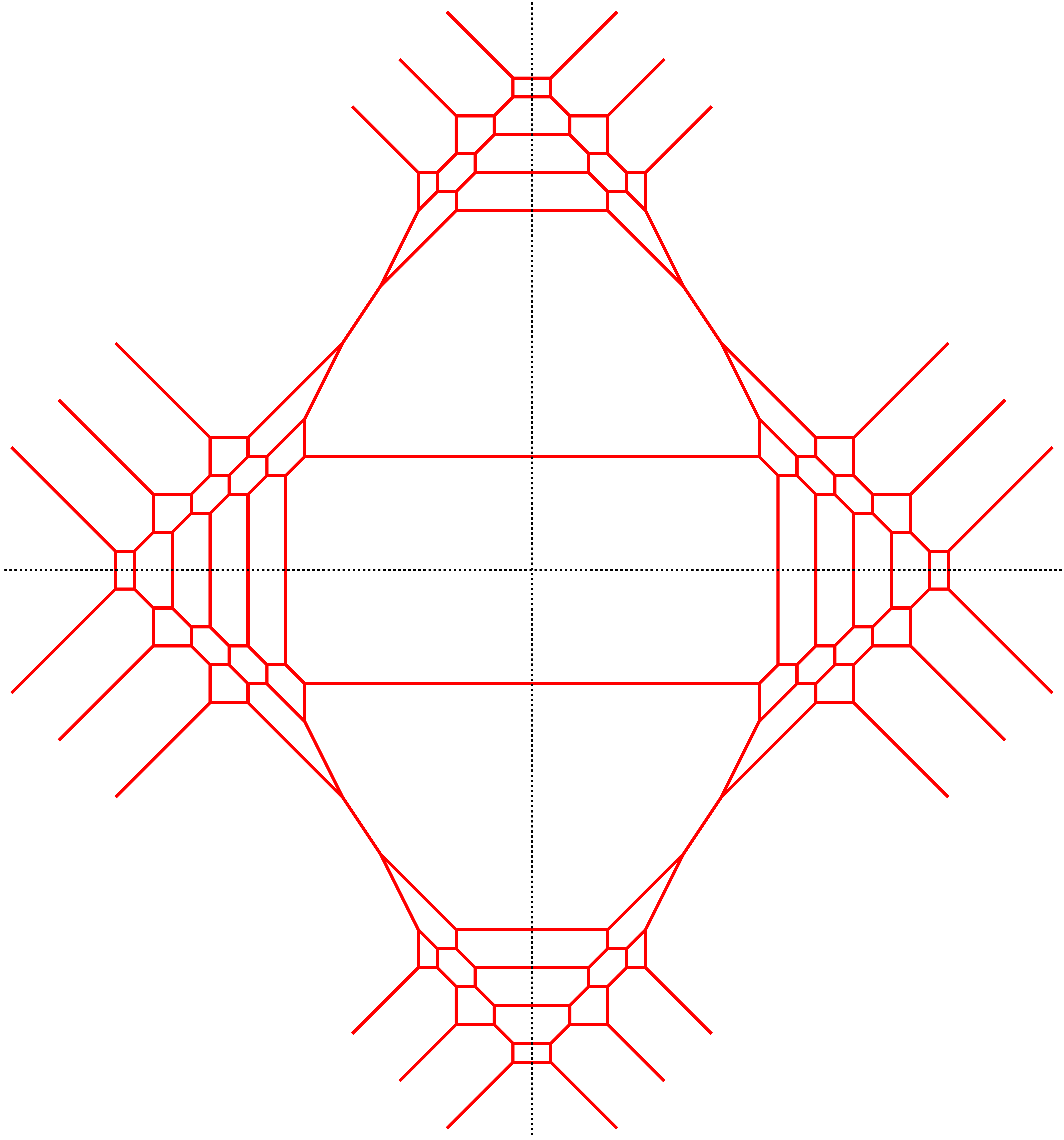}&
\includegraphics[width=5cm, angle=0]{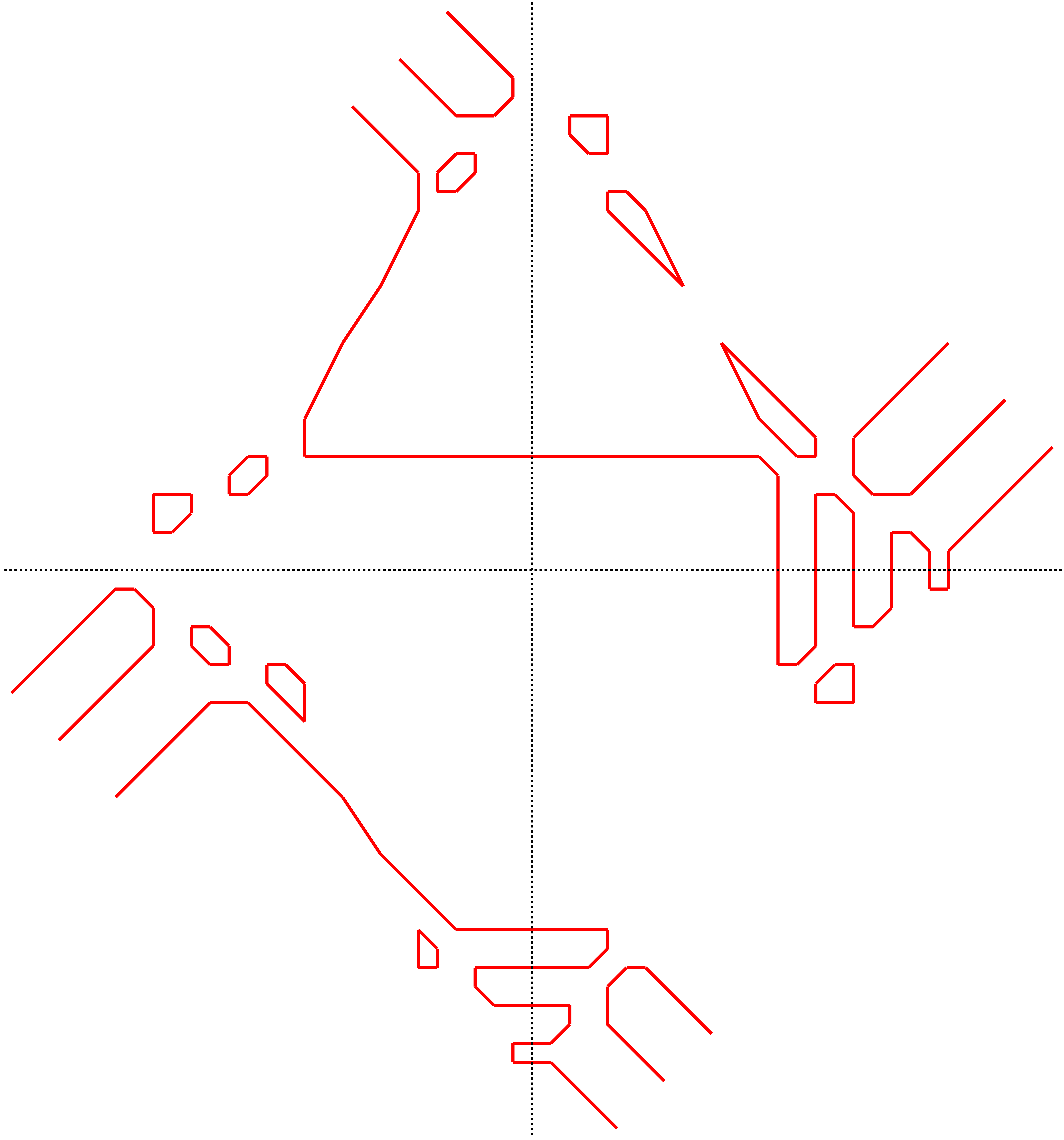}
\\ \\ a) & b) & c)

\end{tabular}
\end{center}
\caption{Gudkov's curve}
\label{Gud}
\end{figure}

First of all, since the field $\R$ is not algebraically closed, we will not work
with real algebraic curves but rather \textit{complex algebraic curves}. In other words, 
subsets of the space 
$(\CC^*)^2$
 defined by equations of the form   $P(x,y)=0$, 
where $\CC^*=\CC\setminus\{0\}$, and 
$P(x,y)$ is a polynomial with complex coefficients (which can therefore
be real). 
For $t$ a positive real number, we define the map
$\text{Log}_t$ on $(\CC^*)^2$ by: 
$$\begin{array}{cccc}
\text{Log}_t & (\CC^*)^2& \longrightarrow & \RR^2
\\ & (x,y) &\longmapsto  & (\log_t |x|,\log_t |y|).
\end{array} $$

We denote the image of a curve given by the equation $P(x,y)=0$   under the map
$\text{Log}_t$ by $\A_t(P)$, and call it the 
\textit{amoeba of base 
  $t$} of the curve. 
  Exactly as in Section \ref{deqdte}, the limit of $\A_t(P)$ when $t$ goes to $+\infty$ 
  is a tropical curve having a single vertex,  with infinite rays corresponding to asymptotic directions of the amoebas $\A_t(P)$.
  However, one can obtain more interesting limiting objects by considering \emph{families} of algebraic curves. 
  That is to say, for each $t$, not only the base of the logarithm changes, but also the \emph{curve whose amoeba we  are considering}.
  Such a family of complex algebraic curves
  takes the form of a polynomial $P_t(x,y)$ whose coefficients are given as functions
  of the real number $t$. Therefore, for each choice of $t>1$ we obtain an 
  honest polynomial in $x$ and $y$ which defines a curve in the plane.   
  
  The following theorem says that all tropical curves arise as the limit of amoebas of 
  families of 
  complex algebraic 
  curves, and  
  provides a fundamental link between classical algebraic geometry 
  and tropical geometry.

\begin{thm}[G. Mikhalkin, H. Rullg\aa rd]\label{thm:approx}
Let $P_{\infty}(x,y)=\tg \sum_{i,j}a_{i,j}x^iy^j \td$ be a tropical
polynomial.
 For each coefficient $a_{i,j}$ different from
    $-\infty$, choose a finite set $I_{i,j}\subset\RR$ such that
    $a_{i,j}=\max I_{i,j}$, and choose
$\alpha_{i,j}(t)=\sum_{r\in I_{i,j}} \beta_{i,j,r}t^r$ with $
    \beta_{i,j,r}\ne 0$.
For every $t>0$, we define 
the complex polynomial
 $P_t(x,y)=\sum_{i,j}\alpha_{i,j}(t)x^iy^j$. Then the amoeba 
 $\A_t(P_t)$ converges to the tropical curve defined by 
$P_\infty(x,y)$ when $t$ tends to $+\infty$. 
\end{thm}
The dequantisation of the curve seen in Section
 \ref{deqdte} is a particular case of the above theorem. The amoeba of base $t$
 of the line with equation $t^0x-t^0y+t^01=0$ converges to the tropical line 
 defined by $\tg
0x+0y+0\td$. 
We can deduce Viro's theorem from the preceding theorem by remarking, 
among other details, that if the 
 $\beta_{i,j,r}$ are real numbers, then the curves defined by the 
polynomials  $P_t(x,y)$  are real algebraic curves.

\subsection{Exercises}
\begin{exo}
\begin{enumerate}
\item Construct a real tropical curve of degree 2 that realises the same arrangement as a 
hyperbola in $\R^2$. Do the same for a parabola. Can we construct a real  tropical curve 
realising the same arrangement as an ellipse?

\item By using patchworking, show that there exists a real algebraic curve of 
degree 4 realising the arrangement shown in Figure 
  \ref{quartic}b. Use the construction illustrated in Figure 
  \ref{Cub} as a guide.
\item Show that for every degree $d$, there is a real algebraic curve
with $\frac{d(d-1)+2}{2}$ connected components. 
\end{enumerate}
\end{exo}

\renewcommand{\L}{{\mathcal L}}
\section{Further directions}\label{sec:further}

So far we only considered tropical polynomials in one
  variable, their tropical roots, and tropical curves in the plane. 
Tropical geometry however extends 
 far beyond these basic objects, and we would like to
showcase briefly some  horizons that fit
with the topics of  the material discussed in this text.

\subsection{Hyperfields}

Recall that we started our tour of tropical geometry by introducing
the tropical semi-field $\T$ which we obtained from the real 
numbers by Maslov's dequantisation procedure. We then proceeded to
consider the solutions to polynomial equations 
  over this semi-field by choosing our
definitions wisely.  
Since the appearance of tropical algebra, there have been several proposals to enrich this  simple tropical semi-field to other algebraic 
structures in order to better understand  relations between tropical algebra and geometry. 
Here we present an example of such an enrichment, suggested by Viro, known as \emph{hyperfields}.

Similarly to a field, a hyperfield is a set
equipped with  two operations called multiplication
and  addition, which satisfy some
axioms. 
The main difference between
  hyperfields and (semi-)fields is that addition is allowed to
 be \textit{multivalued}, this means that the sum
of two numbers is not necessarily just one number as we are used to,
but it can actually be a set  of numbers. 
The simplest hyperfield is the sign hyperfield, 
which consists of only three elements $0$, $+1$
and $-1$. The multiplication is as usual, 
i.e. $(-1) \times (-1) = 1$, but
the multivalued addition is defined by the following
table,
\begin{center}
  \begin{tabular}{  c || c | c | c }
    $+$  & 0 & $+1$ & $- 1$\\ \hline \hline
    0  & \{0\}  & $\{+1\}$ & $\{-1\}$ \\ \hline
    $+ 1$ & $\{+1\}$  & $\{+ 1\}$ & $\{0, \pm1\}$  \\ \hline
      $- 1$ & $\{-1\}$ & $\{0, \pm 1\}$  & $\{-1\}$ \\
  \end{tabular}
\end{center}
In fact, this hyperfield is  quite natural.  Suppose you were to add or multiply two 
real numbers but you only knew their signs, then the 
corresponding operations of the signed hyperfield
give all the possible signs of the output.

There exist several hyperfields related to tropical
  geometry. Here we
will show just one, which resembles our tropical semi-field quite
a bit. It consists of the same underlying
set $\T = \R \cup \{-\infty\}$, and  multiplication is again the usual addition. However the
multivalued addition is now the following:  
$$x \downY y = \begin{cases} \{\max(x, y) \}, & \mbox{if } x \neq y \\ 
\{ z \in \mathbb{T} \ | \ z\leq x\}, & \mbox{if } x = y \end{cases}$$ 
Notice that $-\infty$ can still be considered as the tropical zero,
since $x \downY -\infty=\{x\}$ for any $x\in\TT$.

Let us take a look at how this foreign concept of multivalued addition can simplify some definitions and resolve some peculiarities in tropical geometry. 
In
the tropical hyperfield $\TT$, we 
declare $x_0$ 
to be a  root of a
tropical polynomial $P(x)$ if the set $P(x_0)$ contains the tropical
  zero $-\infty$. Notice that 
both definitions of the roots of a
polynomial coincide whether it is considered over the tropical
hyperfield or the tropical semi-field. 

In  the case of  curves, every classical
line
in the plane is given by the equation
$y= ax + b$. Points 
in $\TT^2$ satisfying the same equation $y =
``ax + b" = \max (a + x, b) $
over   the tropical semi-field   form a piecewise linear graph, as
  depicted in Figure \ref{graphes}a. However this is not at all a
tropical curve, as the balancing condition is not satisfied at
$(b-a,b)$,
 where the 
linearity breaks. However, if we instead replace our semi-field
addition with the new mutlivalued  hyperfield addition, then 
there is one point where the
function $P(x) = (a+x) \downY b$ is truly multivalued: $$P(b-a) = \{ x \in \R \cup
\{-\infty\} \ | \ x \leq b\}.$$ 
Hence, the previous graph grows a vertical tail at $x= b-a$ when the function is considered over the tropical hyperfield, this is the dotted line in Figure \ref{graphes}a. Therefore,   we obtain the tropical line which we first
encountered in Figure \ref{intro}.
  
In conclusion, we see that using the tropical
  hyperfield already simplifies some basic aspects of tropical
  geometry: the definition of the zeroes of polynomials seems more
  natural, and sets in $\RR^2$ defined by a simple equation of the
  form $y=P(x)$ are tropical curves, which was not the case over the tropical semi-field.

\subsection{Tropical modifications}\label{sec:modification}
If we compare
the above example again to the classical situation we may notice
another phenomenon particular to tropical
geometry. 
Classically the shape of a line does not change depending on where it lives. 
For example, a non-vertical line in
$\RR^2$  projects bijectively to the 
$x$-axis; 
more generally any line in
$\RR^n$ 
is in bijection with a coordinate axis via a 
 projection.
This dramatically fails in tropical geometry:
  the projection  to 
  the
 $x$-axis
of a generic tropical line in $\T^2$ no longer gives
  a bijection to $\T$, 
since the vertical edge of  the line is contracted to a
point. 
In general, 
the possible shapes of a  tropical line 
depend on the
space $\T^n$ where it sits.
 We can nevertheless  describe them all:  a generic line
in  $\T^{n}$ has 
the
shape of a line in $\T^{n-1}$ with an additional tail grown at an
interior point. 
We depicted in
Figure \ref{fig:lines} some possible shapes of tropical lines in
$\T^3$ and $\T^5$.
Notice that according to this description, a generic 
tropical line in $\T^n$ has 
$n+1$ ends, and is 
always a tree, meaning a graph with no cycles.\footnote{Alternatively, a tree is a graph in which there is exactly one path joining any two vertices.}

This phenomenon is not  specific to
tropical lines, in general there are infinitely many possible
  ways of modelling an object from classical geometry by a tropical
  object. This presents tropical geometers with some
interesting problems such as:  
What do the different tropical models of the same classical object
have in common? How are two different models of the same object 
related? Does there exists a model better
 than the others? 

\begin{figure}
\begin{center}
\begin{tabular}{ccccccc}
\includegraphics[width=3cm,  angle=0]{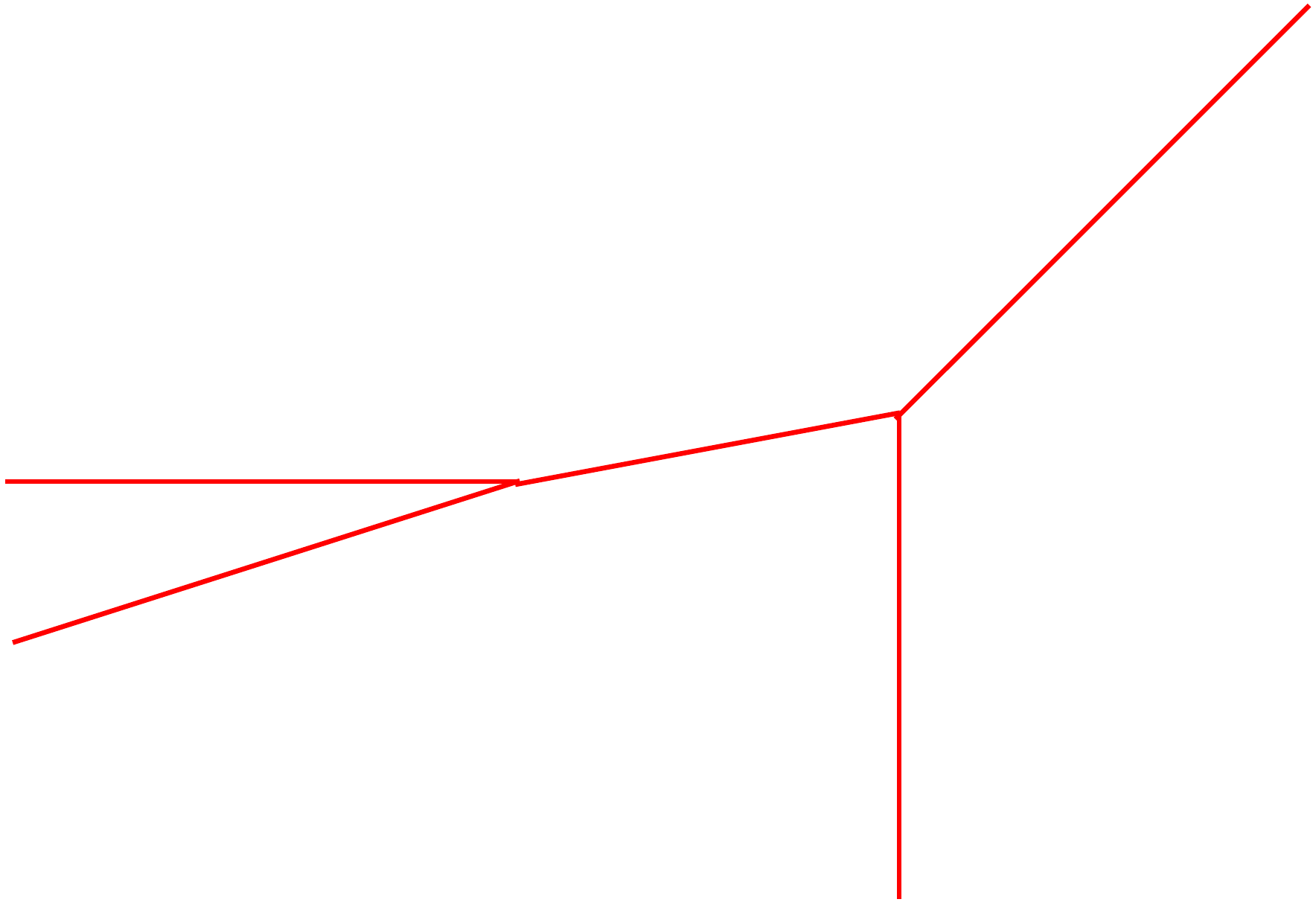} &\hspace{0cm} &
\includegraphics[width=2.5cm,  angle=0]{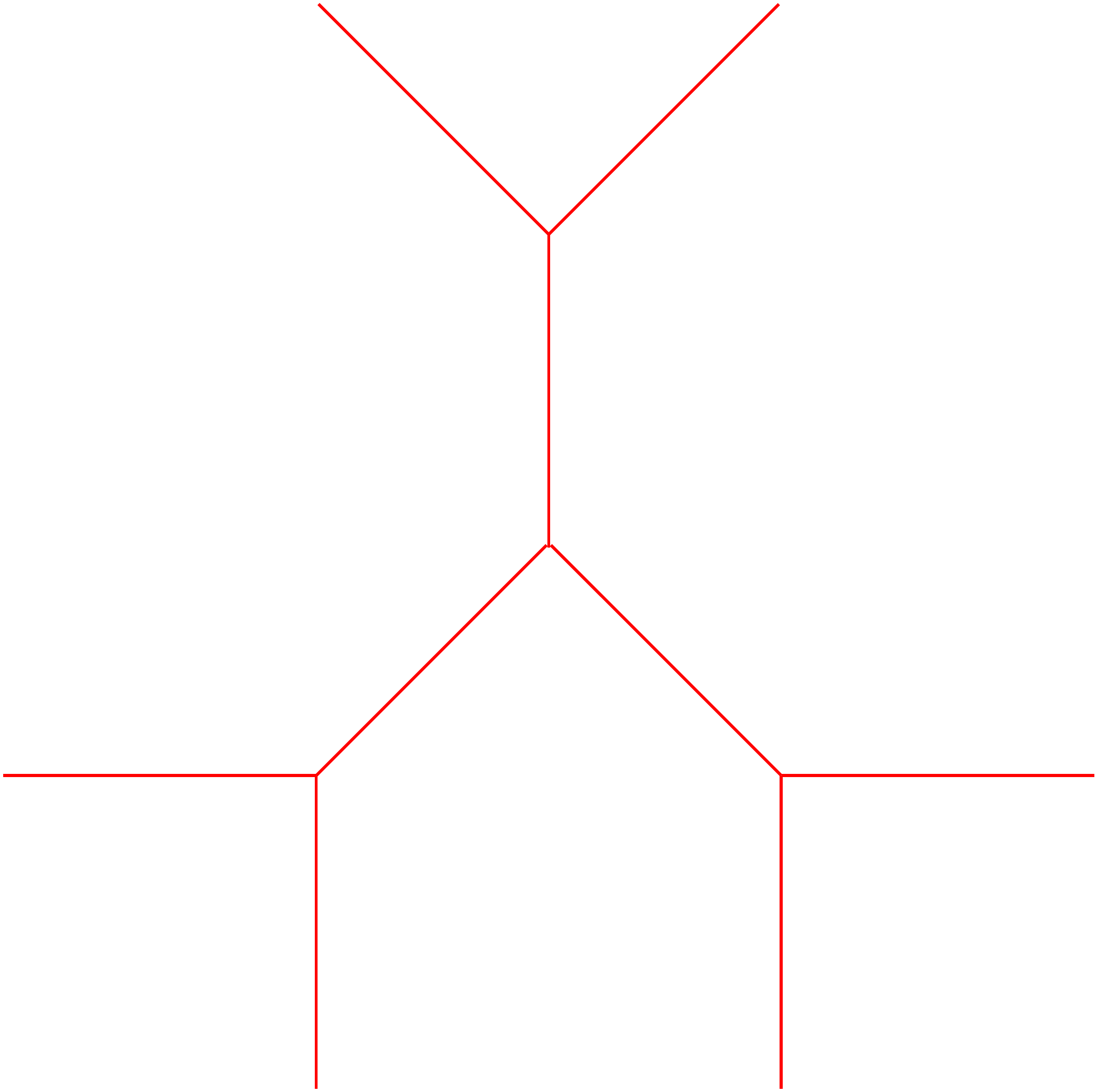}&\hspace{0cm} &
\includegraphics[width=3cm,  angle=0]{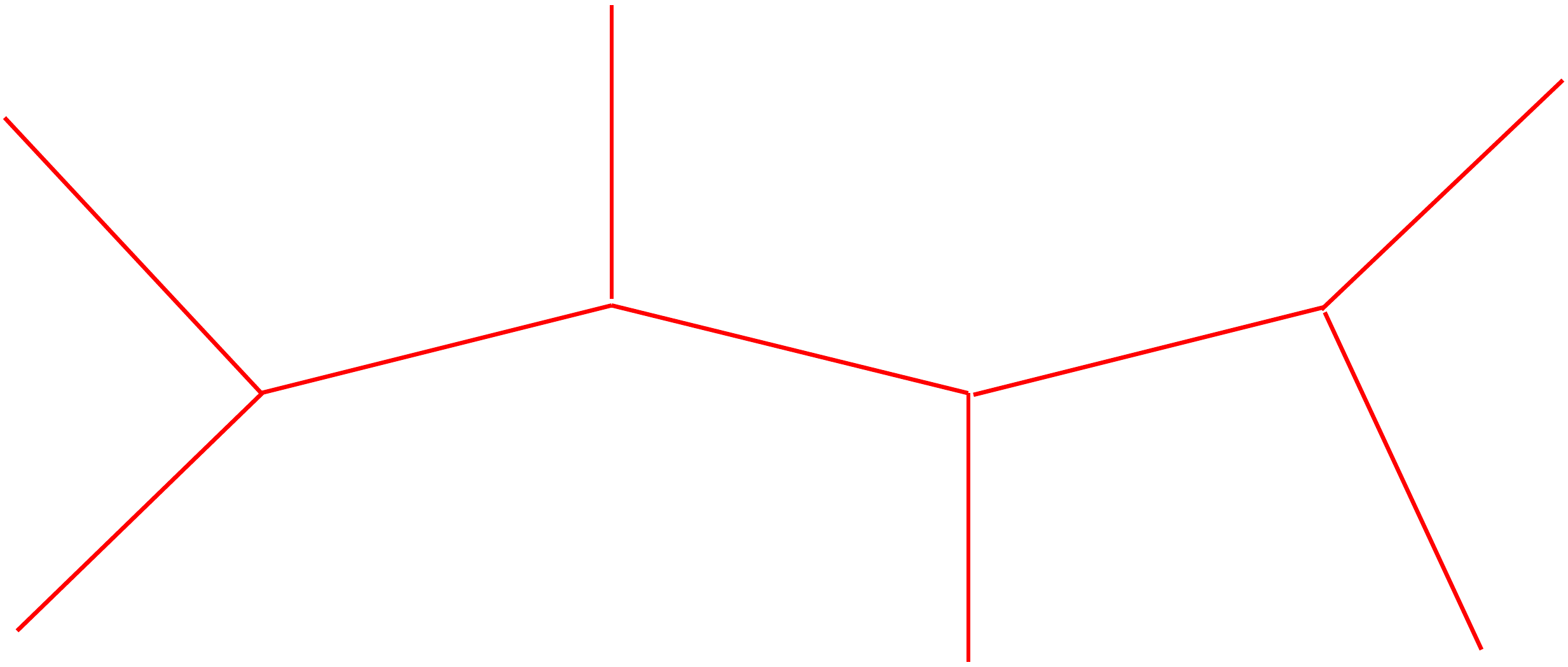}
\\ \\ \small{a) A tropical line in $\T^3$} && \small{b) A ``snowflake" line in $\T^5$} && \small{c) A ``caterpillar" line in $\T^5$}
\end{tabular}
\end{center}
\caption{Different tropical lines}
\label{fig:lines}
\end{figure}

It turns out  that different tropical  models  
are related by an operation known as \emph{tropical modification}.  In fact, the example which we just saw with
a tropical line in $\TT^2$ along with the projection to $\T$  in the
vertical direction is precisely 
a  tropical modification of  $\T$. Let us take a look at another
example in two dimensions.
Consider the tropical  polynomial in two
variables $P(x, y) = x \downY y \downY
0$ over the tropical hyperfield.  
The subset  $\Pi$ of $\T^3$ defined\footnote{Here we use the tropical hyperfield  for consistency. Alternatively, definitions of Sections \ref{sec:algebra}
    and
 \ref{sec:curves} admit natural
 generalisations to polynomials with any number of variables, and
 $\Pi$ is the tropical surface defined
    by the tropical polynomial $\tg 0+x+y+z \td$.} by $z$ $\in$ $P(x, y)
  $ is drawn in Figure 
\ref{Plane}. It is a piecewise linear surface with
 $2$-dimensional  faces; three of them are in
the downwards vertical direction and arise due to $P(x,y)$ being
multivalued. These six faces are glued along any pair of the four rays
starting at $(0,0,0)$ and in the directions $$u_1 = (-1, 0, 0), \qquad u_2 = (0, -1,
0),  \qquad u_3  = (0, 0, -1), \qquad \text{and} \qquad u_0 = (1,
1,1).$$

Just as in Section
\ref{sec:amoeba}, we may consider the limit of  amoebas
$\Log_t(\P)$ of the plane $\P\subset\RR^3$
 defined by the equation $x+y+z+1 = 0$. 
This 
limit is
again
the
 piecewise linear surface $\Pi$, and
therefore this latter   provides another tropical model of a classical
plane\footnote{The same
    phenomenon as  in the case of
    curves happens here: we may consider for simplicity  a plane in $\RR^3$
    because it is defined by a linear equation. However if one
wants to study tropical surfaces of higher degree and their
approximations by amoebas of classical surfaces as in Theorem
\ref{thm:approx}, one has to leave $\RR$ for its algebraic closure $\CC$.}! 
Notice that once again the projection  which forgets the third
coordinate establishes a bijection between $\P$ and $\RR^2$, but is
not a bijection between $\Pi$ and $\T^2$.
This projection map $\Pi \longrightarrow \T^2$  is  a
tropical modification of the tropical plane $\TT^2$.

\begin{figure}
\begin{center}
\includegraphics[scale=1.3]{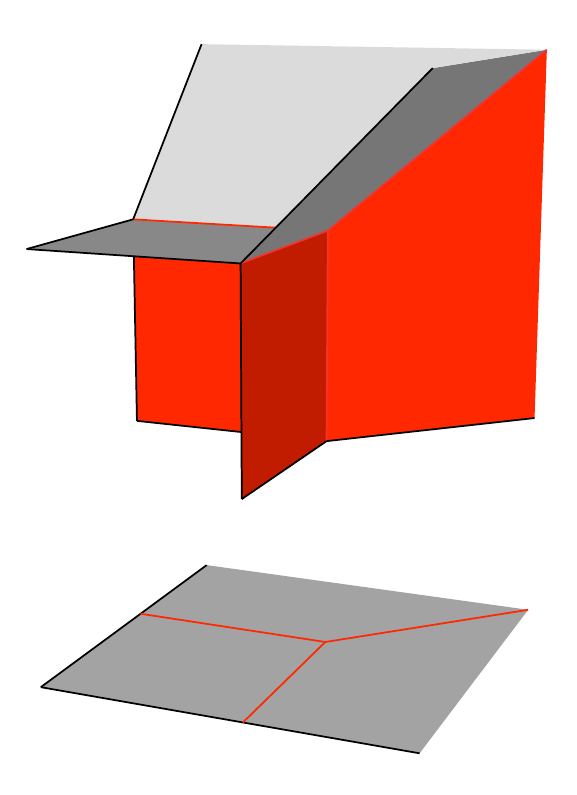}
\put(-150,45){0}
\put(-110, 70){$y$}
\put(-80, 35){$x$}
\put(-10, 160){$\Pi$}
\put(-30, 35){$\T^2$}
\put(-200, 220){\tiny{$(0,-t,0)$}}
\put(-10, 280){\tiny{$(t,t,t)$}}
\put(-160, 190){\tiny{$(-t,0,0)$}}
\put(-90, 160){\tiny{$(0,0, -t)$}}
\end{center}
\vspace{0.5cm}
\caption{A modification $\Pi$ of  the tropical affine plane $\T^2$}
\label{Plane}
\end{figure}
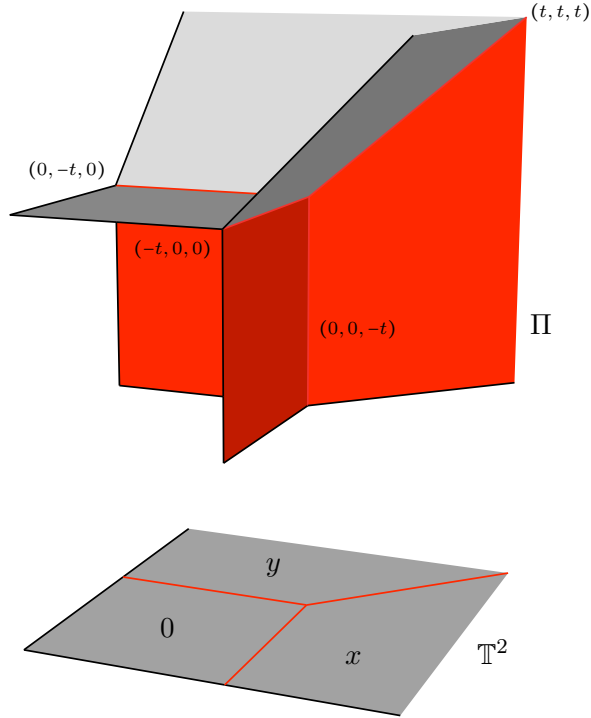

\medskip
You might ask why one would want to use this more complicated model
$\Pi$ of a plane rather than just $\T^2$. Let us show an example.
Consider the real  line, 
$$\L = \{(x, y) \in \RR^2
  \ | \ x+y+1=0 \} \subset \RR^2,$$  
and let
$L\subset\T^2$ be the corresponding tropical line, i.e.
$L=\lim_{t \to +\infty} \Log_t(\L)$. Let $\L'_t$ be a family
  of lines  in $\RR^2$ whose amoebas converge to some tropical line
  $L'\subset\T^2$ as in 
Theorem \ref{thm:approx}, and define $p_t=\L\cap\L'_t$. 
Now let us ask the
following question: can one
determine
$\lim_{t \to +\infty} \Log_t(p_t)$ by only looking at the tropical
picture?

If the set-theoretic intersection of $L$ and $L'$ consist of a single
point $p$, then $\Log_t(p_t)$ has no choice but to converge to
$p$. But what happens when this set-theoretic intersection is infinite?

Suppose that $L'$ is given by the tropical polynomial 
 $``1\cdot x+y+0"$,
 so that the position
 of $L$ and $L'$ is as in
Figure \ref{inter st}a with $L$ in blue and $L'$ in red.  Then the
stable intersection of $L$ and $L'$ is the vertex of $L'$, that is to
say the point $(-1,0)$. Note that this is independent of the family
$\L'_t$, as long as $\Log_t(\L'_t)$ converges to $L'$.
However depending on the family $\L'_t$, the point $\lim_{t\to +\infty}
\Log_t(p_t)$ 
 may be located anywhere on the half-line
$\{(x,0) , -\infty\le x\le-1\}$.
  Indeed,
  given the family
$$\L'_t=\{(x, y) \in \RR^2
  \ | \ (t+1)x+y+(1-t^{b+1})=0\} \textrm{ with } b\le -1,$$ 
then according to Theorem \ref{thm:approx},
 $\lim_{t\to +\infty}\Log_t(\L'_t)=L'$. 
Yet 
we have
$p_t=(t^b,-1-t^b)$, 
and  so
$\lim_{t\to +\infty}\Log_t(p_t)=(b,0)$.
Similarly, we get 
 $\lim_{t\to +\infty}\Log_t(p_t)=(-\infty,0)$ by taking
$$\L'_t=\{(x, y) \in \RR^2
  \ | \ tx+y+1=0\}.$$ 
This shows that using the tropical model
  $L,L'\in\T^2$, it is impossible to extract information about the
  location of the intersection point of $\L$ and $\L'_t$.  
It turns out that changing the model $\T^2$ via a tropical
modification unveils the location of the limit of the intersection
point $p_t$. 
To obtain our new model let us consider again the plane $\P\subset\RR^3$ with
equation 
$x+y+z+1=0$, and the projection
$$
\begin{array}{cccc}
\pi : &\P&\longrightarrow &\RR^2
\\ & (x,y,z) &\longmapsto &(x,y)
\end{array}.
$$ 
The map $\pi$ is clearly  a bijection, and  $\pi^{-1}(x,y)=(x,y,-x-y-1)
$. 
The intersection of $\P$ with the plane $z = 0$ in $\R^3$
 is precisely the line $\L$.
This
implies that  $\Log_t(\pi^{-1}(\L))\subset \{z=-\infty\}$ and 
$$\lim_{t \to +\infty} \Log_t(\pi^{-1}(\L))=\Pi\cap \{z=-\infty\}\subset\T^3.$$

If $\L'_t$ is a family of lines  in $\RR^2$ as above, then  $\pi^{-1}(\L'_t)$
is a family of lines in $\P$, and moreover 
the
intersection point  $p_t$ must satisfy
$\pi^{-1}(p_t)\in \{z=0\}.$  In particular 
we
have $$\lim_{t\to +\infty} \Log_t(\pi^{-1}(p_t))\in \{z=-\infty\}.$$ 
Now it turns out that 
$L''=\lim_{t\to +\infty} \Log_t(\pi^{-1}(\L'_t))$ is a tropical line in $\Pi$, which
intersects the plane $\{z=-\infty\}$ in a single point! Therefore, 
upon projection back to $\R^2$
 this
point is nothing else but $\lim_{t\to +\infty} \Log_t(p_t)$, in other words the location of 
$\lim_{t\to +\infty} \Log_t(p_t)$ 
 is now revealed by considering the tropical picture in
$\T^3$.
In 
Figure \ref{fig:planelines}  we depict the case when $\L'_t$ is given by $(t+1)x+y+(1-t^{b+1})=0$.

\begin{figure}
\begin{center}
\includegraphics[width=8cm, angle=0]{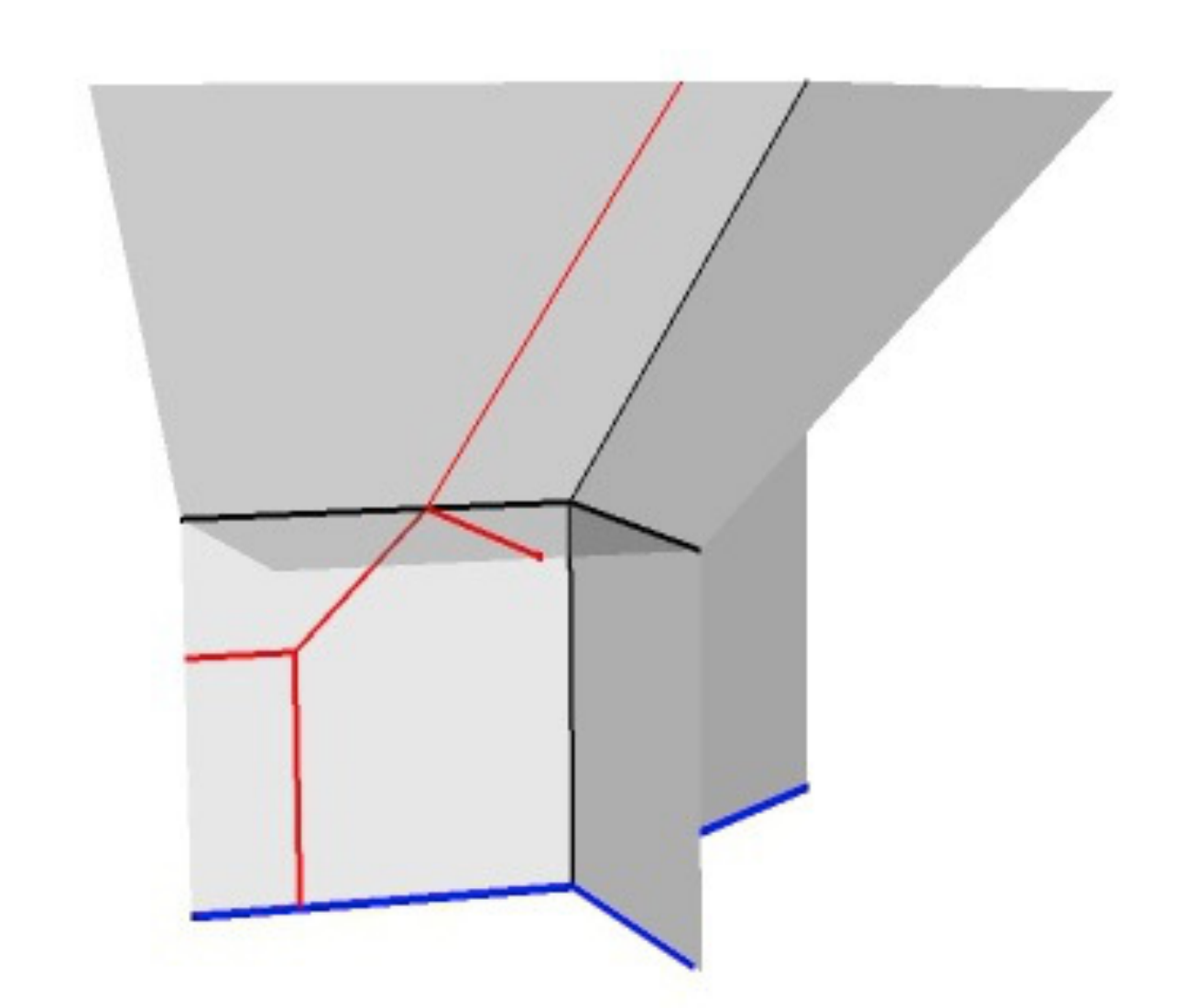}
\put(-190,5){$(b, 0, -\infty)$}
\put(-70, 30){$L$}
\put(-123, 160){$L^{\prime \prime}$}
\end{center}
\caption{The tropical lines $L$ and $L^{\prime \prime}$ in the modified plane $\Pi$}
\label{fig:planelines}
\end{figure}

\medskip
Hence there is no unique  or canonical choice of a tropical model of a
classical space, and tropical modification is the tool used to pass
between the different models. There is a single space which
captures in a sense all the possible tropical models,  it is known as
the Berkovich space of the classical object.  As one can 
imagine the
structure of a Berkovich 
space is extremely complicated due to the  infinitely many different tropical
models. For example, the Berkovich  line is an infinite tree. 
In a sense, a tropical model gives us just a snapshot of the complicated Berkovich  space and
modifications are what allow us to change the resolution.

\subsection{Higher (co)dimension}\label{sec:codim}

In the first part of this text our main object of study was tropical curves 
in the plane.
 Even in this limited case there are many 
applications to classical
  geometry.  In higher dimensions things become a bit trickier as
 the links between
the tropical and classical worlds become more intricate.

As we mentioned earlier, definitions of Sections \ref{sec:algebra}
    and
 \ref{sec:curves} as well as   Theorem \ref{thm:approx} admit natural
 generalisations to polynomials with any number of variables.
A tropical polynomial in $n$-variables
$P(x_1, \dots , x_n)$ defines a tropical hypersurface in
$\T^n$, which is
the corner locus of the graph of $P$ equipped with some weights.
A tropical hypersurface  
in $\T^n$ 
is glued from flat pieces,  has
dimension $n-1$,  
satisfies a balancing
condition similar to curves, and  
arises as the limit of
amoebas of a family of complex hypersurfaces. 
Because of this we say that all tropical hypersurfaces are
\textit{approximable}.

It is also natural to
consider tropical objects of 
dimension 
less than $n-1$
in
$\T^n$. A \textit{$k$-dimensional tropical variety} is a
weighted polyhedral complex of dimension
$k$  satisfying a generalisation of the balancing condition which
we saw for curves.  It would be too technical and probably useless 
to give here the
balancing condition in full generality. Nevertheless in the case of
curves, i.e.~when $k=1$, the definition is the same
as  
in
Section \ref{sec:balancing}: a tropical curve in $\R^n$ is a graph whose edges
have a rational direction, 
are equipped with positive integer
weights. In addition, the graph must satisfy the balancing condition given in 
Section \ref{sec:balancing} at each vertex.
Unfortunately, Theorem \ref{thm:approx} does not generalise 
to the case of higher codimensions:
there exist tropical varieties, said  to be \textit{not approximable}, 
which do not arise as limits of 
amoebas of  \textit{complex varieties} of the same degree. 
For our purposes, we may think of a complex variety in
$\CC^n$ as being the solution set of a system of polynomial equations
in $n$-variables.

Already in $\T^3$ there  are many examples of  tropical curves which are not
approximable. 
More intricately, there are examples of pairs of tropical varieties $Y
\subset X \subset \T^n$ for which both $Y$ and $X$ are approximable
independently, however  
there do not exist approximations
$\mathcal Y_t$ and $ \mathcal X_t $ of $Y$ and $X$ with
$\mathcal Y_t\subset \mathcal X_t $.
Let us 
give an explicit example of such a pair.

Consider the modified tropical plane 
$\Pi \subset \T^3$ we just encountered,
which is approximated by the amoebas of the complex plane $\P \subset \CC^3$ defined by the equation $x+y +z +1 = 0$.   
Now, the  union of three rays in the directions
$$(-2,-3,0), \qquad (0,1,1), \qquad (2,2,-1)$$
with each ray equipped with  weight $1$ is a tropical curve $C \subset
\T^3$. We may check that  the balancing condition is satisfied:
$$(-2, -3, 0)+ ( 0, 1,1 )+ (2, 2, -1)=0 .$$
Moreover $C$ is contained in the tropical
plane $\Pi$,
since  we may express the three rays as 
$$(-2, -3, 0) = 2u_1 + 3u_2, \qquad (0, 1, 1) = u_0 + u_1, \qquad (2,
2, -1) = 2u_0 + 3u_3,$$ where the $u_i$ are the directions of the rays
of $\Pi$ given in Section \ref{sec:modification}. 
The tropical curve $C$ is approximable\footnote{For example one can check that $C=\lim_{t\to+\infty} \Log_t(\C)$ where
$\C=\{(u^2(u-1), u^3, (u-1)), \ u\in\CC\}$.},
 however there are no
approximations $\C_t$ and $\P_t$ of $C$ and $\Pi$ with $\C_t\subset
\P_t$.
The full proof of this statement involves several steps, and we only
focus here on the essential one. 
Hence let us
assume that we may take both families $\P_t$ and
$\C_t$ to be constant,
with $\P_t$
equal to 
the above plane 
$\P$ for all $t$.
That is to say we are reduced
to showing that there does not exist a complex curve $\C\subset\P$ such
that $C=\lim_{t\to+\infty} \Log_t(\C)$.

\begin{figure}
\begin{center}
\includegraphics[width=2.5cm,  angle=0]{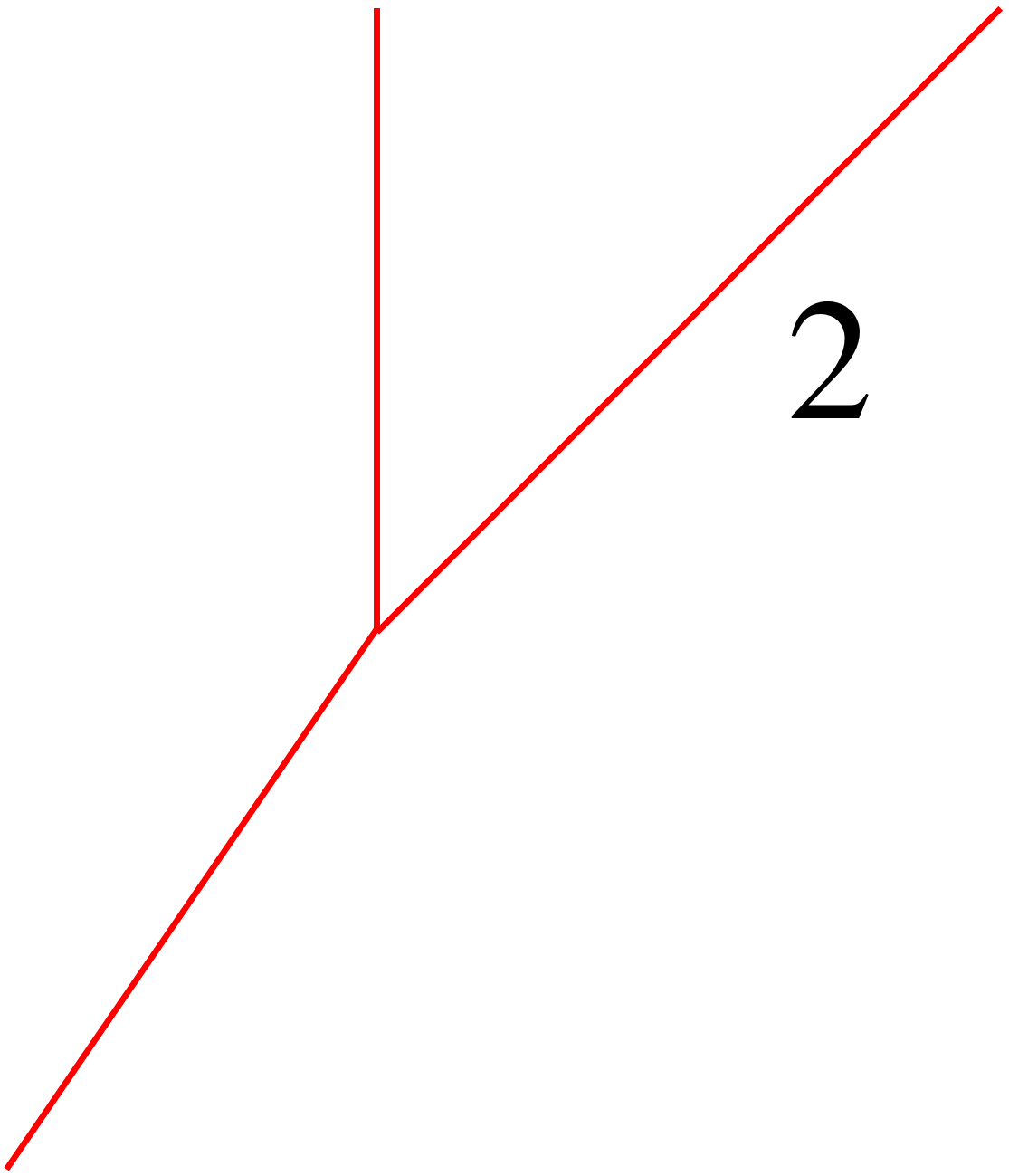} \hspace{3cm} 
\includegraphics[width=2.5cm,  angle=0]{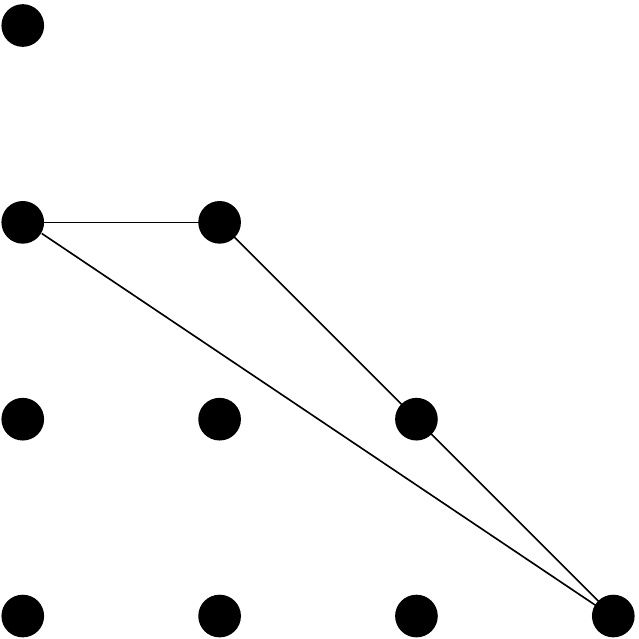}
\end{center}
\caption{The projection of the curve $C \subset P$ to $\T^2$ along with its dual polygon}
\label{fig:cusp}
\end{figure}

Suppose on the contrary that such a complex curve $\C$ exists.
The projection of $C$ to $\TT^2$ which 
 forgets the $z$-coordinate is a
tropical curve $C'$ as defined in
Section \ref{sec:curves}. 
The curve $C'$ and its 
Newton polygon
 are shown in Figure \ref{fig:cusp}. 
A polynomial
with such a  Newton polygon is  of degree $3$ and has the form $P(x, y) = ax^3 + by^2 + \dots$, where the monomials
 after the dots have higher degree. A complex curve in $\CC^2$ with
 such an equation has a type of singularity at $(0, 0)$ known as  a
 \textit{cusp}.  
Hence the ray of $C$ in  the  direction $(-2,
 -3, 0) =  2u_1 + 3u_2$ tells us that the projection of $\C$ has a
 cusp at $(0,0)$. Since $\C$ is contained in $\P$ and
 $\P$ projects bijectively
to $\CC^2$, we deduce that $\C$
 itself has a cusp at $(0,0,-1)$.
 In a similar way, the ray in direction  $2u_0
 + 3u_3$ implies that 
$\C$
must have a second  
cusp \textit{at infinity}\footnote{This statement makes sense in the framework
  of projective geometry.}.
However it is known that a plane curve of degree $3$ can have at most
one cusp,
which contradicts the existence of $\C$. 
This tells us that our  tropical curve is not approximable
by a complex algebraic curve of degree $3$ contained in  the plane $\P$.  

Many other examples of non-approximable tropical objects exist,
including lines in surfaces,  curves in space, and even linear
spaces. What must be stressed is that from a purely combinatorial view
point it is
very difficult to distinguish these pathological
tropical objects from approximable ones. This is just one of the
problems that makes tropical geometry continue to be a challenging and
exciting field of research especially in terms of its relations to classical 
geometry.

\section{References}\label{ref}
 The references given below contain the proofs of statements
  contained throughout the text. They also present the many directions in which tropical geometry is
 developing. The references to the  literature go far beyond what we mention here; it would be impossible to provide a complete list.  However, we hope to give an account of the multitude of perspectives on tropical geometry. 

Firstly, we should note that 
tropical curves first appeared under the
name of \emph{$(p, q)$-webs}, in the work of the physicists
O. Aharony, A. Hanany, and B.  Kol \cite{Aharony2}. 
In mathematics, roots of tropical geometry can be traced back, at least to the work of Bergman \cite{Berg71} on logarithmic limit sets of algebraic varieties, of Bieri and Groves \cite{BiGr} on non-archimedean valuations, of Viro on patchworking (see references below), as well to the study of toric varieties, which  provide a first relation between convex polytopes and algebraic geometry, see for example \cite{Ful}, \cite{Kho86}. 

\vspace{1ex}

There are other  general  introductions to tropical geometry.  The texts
\cite{BIT} and  \cite{St5} 
are aimed at the interested reader with a minimal background in
mathematics, similar to the one assumed in this text. There the
authors
emphasize
different applications of tropical geometry, enumerative geometry in
\cite{BIT}, and phylogenetics in    \cite{St5}.
More confident readers  can also read the works of 
\cite{St2}, \cite{Gath1}, \cite{Mik9}, \cite{Mik-whatis},  \cite{Mik-It-11}, \cite{MacL12}, \cite{St7}, and \cite{MikRau}.  
For established geometers, we recommend 
the state of the art \cite{Mik8} and \cite{Mik3}.

In 
  particular stable intersections of tropical 
curves in $\R^2$ are discussed in detail
  in \cite{St2}. The proof of Bézout Theorem we gave is contained in
  \cite{Gath1}. Foundations of a more sophisticated tropical
  intersection theory are exposed in \cite{Mik3}, and further developed in  \cite{AlRa1}, \cite{KatzTool}, \cite{Shaw1}.

Amoebas of algebraic varieties were introduced in \cite{GKZ}. The interested reader can 
also see the texts  \cite{Passare}, \cite{ViroWhat}, and the more sophisticated \cite{PassRull}. 
For a deeper look at patchworking, amoebas, and Maslov's dequantisation, as well as their applications to
Hilbert's 16th problem
we point you to the texts, 
 \cite{V9}, \cite{V11}, \cite{Mik8} and \cite{Mik11} as well as the website \cite{Voueb}. 
 We especially recommend the text \cite{IV2}, dedicated to a large audience, which explains how patchworking has been used to disprove the \emph{Ragsdale conjecture} which has been open for decades.

For more on  tropical hyperfields we point the reader towards  
\cite{ViroHyp}. As we mentioned in the text, there are other enrichments of the tropical semi-field, some examples can be found in 
\cite{Izha}, \cite{Fuzzy} and \cite{Conn}.  
Tropical modifications were introduced by Mikhalkin in the above mentioned text \cite{Mik3}, 
and their relations to Berkovich spaces  
can be found in \cite{Pay4} and 
\cite{Pay5}. 
This perspective from  Berkovich spaces allows one to
relate tropical geometry to algebraic geometry not only over $\CC$ but
over any  field equipped with a so-called \emph{non-archimedean
  valuation} (e.g. the $p$-adic numbers equipped with the $p$-adic
valuation). For this point of view, we refer for
example to  \cite{Rab1}, \cite{Pay2}, \cite{Pay6} and reference therein. 
The approximation problem  has
been studied mostly in the case of curves. 
For tropical curves 
see \cite{Mik1},
 \cite{Mik3}, \cite{Spe2}, \cite{NS}, \cite{Nishinou},
 \cite{Tyomkin},  \cite{Katz1}, \cite{ABBR13}. For a look at
the approximation problem  
 of curves in surfaces as in the example provided in the last
section, see \cite{BogKat},  \cite{Br17}, and  \cite{GathSchW}. 
Among all sources of motivation to study the latter problem, we refer to the very 
nice study of tropical lines in tropical surfaces in $\RR^3$ by Vigeland in \cite{Vig1} and \cite{Vig2}.

\medskip
To end this introduction to tropical geometry, we want to point out that this subject
has very successful applications to many fields other than Hilbert's 16th problem, such 
as enumerative geometry \cite{Mik1},  algebraic geometry \cite{Tev1}, mirror symmetry \cite{Gross}, mathematical biology \cite{Ard}, \cite{SturmBio}, computational complexity \cite{Gaub} 
computational geometry \cite{Yu}, \cite{Cueto}, and algebraic statistics \cite{CuetoBoltz}, \cite{SturmBio} just
to name a few...
\\

\bibliographystyle {alpha}
\bibliography {Biblio.bib}

\end{document}